\documentclass[10pt]{amsart}

\usepackage[dvipsnames, rgb]{xcolor}
\usepackage{amssymb, amsmath, amsthm, amsfonts, stmaryrd, mathrsfs}
\usepackage{graphicx}
\usepackage{tikz}
\usepackage{tikz-cd}
\usetikzlibrary{shapes,arrows}
\usepackage{float}
\usepackage{hvfloat}
\usepackage[width=\textwidth]{caption}
\usepackage{rotating}
\usepackage{adjustbox}

\usepackage{url}
\usepackage[all,arc,2cell]{xy}
\UseAllTwocells
\usepackage{enumerate}

\usepackage{multicol}

\usepackage{hyperref}
  \definecolor{dark-red}{rgb}{0.6,0.15,0.15}
   \definecolor{dark-blue}{rgb}{0.15,0.15,0.6}
   \definecolor{medium-blue}{rgb}{0,0,0.5}
   
   \definecolor{cbf-bluedark}{RGB}{17, 112, 170}
           \definecolor{cbf-blue}{RGB}{95, 162, 206}
               \definecolor{cbf-bluepale}{RGB}{163, 204, 233}
               
                        \definecolor{cbf-bluepale2}{RGB}{162, 200, 236}
        
      \definecolor{cbf-orangedark}{RGB}{200, 82, 0}
          \definecolor{cbf-orange}{RGB}{252, 125, 11}
      \definecolor{cbf-orangepale}{RGB}{255, 188, 121}

\hypersetup{
    colorlinks, 
    linkcolor=dark-red,
    citecolor=dark-blue, urlcolor=medium-blue
}

\usepackage[nameinlink,capitalise,noabbrev]{cleveref}

\numberwithin{equation}{section}

\newtheorem{thm}{Theorem}[section]
\newtheorem{theorem}{Theorem}[section]

\newtheorem{cor}{Corollary}[section]

\newtheorem{prop}{Proposition}[section]

\newtheorem{lem}{Lemma}[section]
\newtheorem{lemma}{Lemma}[section]

\theoremstyle{definition}
\newtheorem{defn}{Definition}[section]

\newtheorem{notation}{Notation}[section]

\newtheorem{rem}{Remark}[section]
\newtheorem{warn}{Warning}[section]

\newtheorem*{openprob}{Open Problem}

\makeatletter
\let\c@equation=\c@thm
\let\c@lem=\c@thm
\let\c@theorem=\c@thm
\let\c@lemma=\c@thm
\let\c@Theorem=\c@thm
\let\c@Lemma=\c@thm
\let\c@cor=\c@thm
\let\c@corollary=\c@thm
\let\c@Corollary=\c@thm
\let\c@conj=\c@thm
\let\c@conjecture=\c@thm
\let\c@prop=\c@thm
\let\c@proposition=\c@thm
\let\c@Proposition=\c@thm
\let\c@defn=\c@thm
\let\c@definition=\c@thm
\let\c@Definition=\c@thm
\let\c@notation=\c@thm
\let\c@note=\c@thm
\let\c@exmp=\c@thm
\let\c@ex=\c@thm
\let\c@exmps=\c@thm
\let\c@rem=\c@thm
\let\c@warn=\c@thm
\let\c@claim=\c@thm
\let\c@convention=\c@thm
\let\c@conventions=\c@thm
\let\c@quest=\c@thm
\let\c@facts=\c@thm
\let\c@slogan=\c@thm
\let\c@ass=\c@thm
\makeatother

\newcommand{\F}{\mathbb{F}}

\newcommand{\Z}{\mathbb{Z}}
\newcommand{\W}{\mathbb{W}}

\newcommand{\G}{\mathbb{G}}

\newcommand{\ZZ}{\mathbb{Z}}
\newcommand{\FF}{\mathbb{F}}
\newcommand{\QQ}{\mathbb{Q}}
\newcommand{\GG}{\mathbb{G}}

\def\SS{\mathbb{S}}

\newcommand{\xra}{\xrightarrow}

\def\makeop#1{\expandafter\def\csname #1\endcsname{\mathop{\mathrm{#1}}\nolimits}}

\makeop{Gal}
\makeop{id}
\makeop{Mod}
\makeop{Hom}
\makeop{Tot}
\makeop{gr}
\makeop{Out}
\makeop{Hom}
\makeop{Ext}
\makeop{End}
\makeop{Aut}
\makeop{Tor}
\makeop{ev}
\makeop{hocolim}
\makeop{holim}
\makeop{map}
\makeop{id}
\makeop{colim}
\makeop{im}
\makeop{pic}
\newcommand{\br}[1]{\llbracket #1 \rrbracket} 
\newcommand{\kappabar}{\bar{\kappa}}

\newcommand{\WW}{\mathbb{W}}
\newcommand{\sC}{\mathscr{C}}

\usepackage{xargs} 
\usepackage[textwidth=2.5cm]{todonotes}
\newcommandx{\irina}[2][1=]{\todo[linecolor=red,backgroundcolor=red!25,bordercolor=red,#1]{#2}}
\newcommandx{\cuong}[2][1=]{\todo[linecolor=orange,backgroundcolor=orange!25,bordercolor=orange,#1]{#2}}
\newcommandx{\agnes}[2][1=]{\todo[linecolor=blue,backgroundcolor=blue!25,bordercolor=blue,#1]{#2}}
\newcommandx{\paul}[2][1=]{\todo[linecolor=green,backgroundcolor=green!25,bordercolor=green,#1]{#2}}
\newcommandx{\vesna}[2][1=]{\todo[linecolor=magenta,backgroundcolor=magenta!25,bordercolor=magenta,#1]{#2}}
\newcommandx{\vesnainline}[2][1=]{\todo[linecolor=magenta,backgroundcolor=magenta!25,bordercolor=magenta,#1,inline]{#2}}
\newcommandx{\agnesinline}[2][1=]{\todo[linecolor=blue,backgroundcolor=blue!25,bordercolor=blue,#1,inline]{#2}}

\newcommandx{\tagnes}[2][1=]{\todo[linecolor=green,backgroundcolor=green!25,bordercolor=green,#1]{#2}}

\makeatletter
\newcommand{\mylabel}[2]{#2\def\@currentlabel{#2}\label{#1}}
\makeatother

\newcommand{\EE}{\mathbf{E}}
\newcommand{\MM}{\mathbf{M}}

\newcommand{\red}{r}

\newcommand{\wchi}{\widetilde{\chi}}
\newcommand{\etachi}{\eta\chi}

\newcommand{\hbone}{\left[\frac{\sigma \eta}{v_1^2}\right]}
\newcommand{\gbone}{\left[\frac{\sigma \eta^2}{v_1^3}\right]}
\newcommand{\hbonebar}{\left[\frac{\sigma\eta\chi}{v_1^2}\right]}

\newcommand{\ghbone}{\left[\frac{\sigma\eta^3}{v_1^4}\right] }
\newcommand{\ghbonebar}{\left[\frac{\sigma\eta^3 \chi }{v_1^4}\right]} 
\newcommand{\hchi}{\left[\frac{\eta \chi}{v_1}\right]}
\newcommand{\gchi}{\left[\frac{\eta^2 \chi}{v_1^2}\right]}

\newcommand{\kchi}{\left[\frac{\eta^4 \chi}{v_1^4}\right]}

\newcommand{\gwchi}{\left[\frac{\eta^2 \wchi}{v_1^2}\right]}

\newcommand{\hbtwo}{[hb_2]}

\newcommand{\btwobar}{[\overline{b}_2]}

\newcommand*\act{{\mkern 0.5mu \cdot\mkern 0.5mu}}
\newcommand*\piact{\pi\act}

\title[Cohomology of the Morava stabilizer group]{Cohomology of the Morava stabilizer group through the duality resolution at $n=p=2$}
\date{\today}

\author[Beaudry]{Agn\`es Beaudry}
\address{Department of Mathematics, University of Colorado Boulder, Campus Box 395, Boulder, CO, 80309, USA}
\email{agnes.beaudry@colorado.edu}

\author[Bobkova]{Irina Bobkova}
\address{Department of Mathematics, Texas A\&M University, College Station, TX, 77843, USA}
\email{ibobkova@tamu.edu}

\author[Goerss]{Paul G. Goerss}
\address{Department of Mathematics, Northwestern University,  2033 Sheridan Road, Evanston, Illinois, 60208, USA}
\email{pgoerss@math.northwestern.edu}

\author[Henn]{Hans-Werner Henn}
\address{Institut de Recherche Math\'ematique Avanc\'ee, C.N.R.S. et 
Universit\'e de Strasbourg, rue Ren\'e Descartes, 67084 Strasbourg Cedex, France}
\email{henn@math.unistra.fr}

\author[Pham]{Viet-Cuong Pham}
\email{phamvietcuonga2@gmail.com}

\author[Stojanoska]{Vesna Stojanoska}
\address{Department of Mathematics, University of Illinois, Urbana-Champaign, 273 Altgeld Hall
1409 W. Green Street, 
Urbana, IL 61801, USA}
\email{vesna@illinois.edu}

\begin{document}

\begin{abstract}
We compute the continuous cohomology of the Morava stabilizer group with coefficients in Morava $E$-theory, $H^*(\GG_2, \EE_t)$, at $p=2$, for $0\leq t < 12$, using the Algebraic Duality Spectral Sequence. Furthermore, in that same range, we compute the $d_3$-differentials in the homotopy fixed point spectral sequence for the $K(2)$-local sphere spectrum. These cohomology groups and differentials play a central role in $K(2)$-local stable homotopy theory.
\end{abstract}
 
\maketitle
\setcounter{tocdepth}{1}
\tableofcontents

\section{Introduction}\label{sec:intro}


The main purpose of this article is to analyze the cohomology of the Morava stabilizer group with coefficients in the homotopy groups of Morava $E$-theory at the prime $n=p=2$ in a small range.
We aim to explain how these groups can be computed using the algebraic duality resolution, which appeared in \cite{BeaudryRes}, from the already existing mod 2 computations of \cite{BGH, BeaudryTowards}.

To be more precise, let us establish some notation. Let $\Gamma$ be the formal group law of the super-singular elliptic curve $C$ with Weierstrass equation $y^2+y=x^3$ over $\FF_4$. Define 
\[\mathbb{S}_2 =\Aut_{\F_4}(\Gamma)\] 
 to be the automorphism group of $\Gamma$. The Galois group $\Gal=\Gal(\F_4/\F_2)$ acts on $\mathbb{S}_2$ and we let
\[\mathbb{G}_2 =\Aut_{\F_4}(\Gamma) \rtimes \Gal.\]
We let $\EE = E(\F_4, \Gamma)$ be the Lubin--Tate theory (a.k.a Morava $E$-theory) associated to the pair $(\F_4,\Gamma)$ and $\EE_t=\pi_t\EE$.

Our \emph{first} goal is to analyze the \emph{continuous}\footnote{In this paper, group cohomology $H^*$ always denotes \emph{continuous cohomology}.} cohomology
$H^*(\GG_2, \EE_t)$
 in the range $0\leq t<12$ using the algebraic duality resolution. For the reader who likes charts, the final result is depicted in \cref{fig:EinfS2}. For the one who prefers lists, see \cref{fig:table2}. We have a number of other figures -- and a table -- to visually guide the reader through the intermediate calculations.

The importance of these cohomology groups to homotopy theory lies in the fact that
the $K(2)$-local $\EE$-based ANSS can be identified with the homotopy fixed points spectral sequence
\[E_2^{*,*}=H^*(\GG_2, \EE_* )\Longrightarrow \pi_{*}L_{K(2)}S^0.\]
See \cite[Appendix A]{DH}. 
Our \emph{second} result is the computation of the $d_3$ differentials in this spectral sequence in our range.

The spectrum $L_{K(2)}S^0$ is one of the fundamental objects of study in chromatic homotopy theory. 
It is closely related to the $\EE$-local sphere $L_2S^0$. The homotopy of $L_{K(2)}S^0$ can be recovered from that of $L_2S^0$ via the chromatic fracture square. Furthermore, a calculation of the $E_2$-term of the ANSS for $L_2S^0$ appears in the work of Shimomura--Wang \cite{SW} and information about the $E_2$-term of ANSS for $L_{K(2)}S^0$ can be extracted from \cite{SW}. However, such an extraction process is technical and misses the  organizational principle offered by the duality resolution perspective, which has been key to many recent results about the $K(2)$-local category both at the primes $2$ and $3$. For example, the duality resolution was used to study the Hopkins' Chromatic Splitting Conjecture and compute the rational homotopy of the $K(2)$-local sphere in \cite{GoerssSplit, BGH}. At $p=3$, the duality resolution was used to compute the $K(2)$-local Picard group in \cite{GHMRPicard}. Our work here at $p=2$ is designed to be used by the authors to compute the group of exotic elements in the $K(2)$-local Picard group at $p=2$ \cite{BBGHPS_Pic}.

We also note the close relationship between $H^*(\GG_2, \WW)$ and the cohomology of the Morava Stabilizer Algebra. Ravenel was the first to compute an associated graded for this cohomology ring in \cite[Theorem (3.4)]{RavCoh} (although his setup and method are quite different than the duality resolution approach).

\begin{openprob}
 In principle, our computation could be attempted for all $t$ rather than the small range we explore. 
 We invite any interested reader to take on this challenge.
\end{openprob}

Before describing the approach, we define some key subgroups of $\SS_2$ and introduce important element of its cohomology.

\subsection{Subgroups}\label{sec:subgroups}
The automorphisms $\Aut_{\FF_4}(C)$ of the elliptic curve $C$ over $\F_4$ are well-known (see \cite[Appendix A]{Silverman}) and form a group of order $24$, denoted by
\[G_{24} := \Aut_{\F_4}(C) \cong Q_8 \rtimes \F_4^{\times}.\]
We remind the reader that $ \F_4^{\times}$ is a cyclic group of order $3$. Since we will be working in a 2-complete setting, the quaternion subgroup plays the most important role here. 
The action of $ \F_4^{\times}$ stabilizes the center $(\pm 1)$ of $Q_8$ so $G_{24}$ has a subgroup
\[C_6 = (\pm 1) \rtimes  \F_4^{\times}.\]
Since $\Aut_{\FF_4}(C)$ embeds into $\SS_2=\Aut_{\FF_4}(\Gamma)$, these are also subgroups of $\SS_2$.

The group $\SS_2$ admits a surjective homomorphism to the group of units in the $2$-adic integers, called the determinant (see, e.g., \cite[\S 1.2]{ghmr}) and denoted
 \[\det: \SS_2 \to \Z_2^{\times} \cong (\pm 1) \times (1+4\Z_2).\] 
If we compose the determinant with the quotient map of $\Z_2^{\times} $ by the torsion subgroup $(\pm 1)$, we get a surjective homomorphism 
\begin{equation}\label{eq:zeta}
\zeta \colon \SS_2 \to  \Z_2^{\times}/(\pm 1) \cong \Z_2.
\end{equation}
The group $\SS_2^1$ is defined to be the kernel of $\zeta$, and we have a split extension
\[1 \to \SS_2^1 \to \SS_2 \xrightarrow{\zeta} \Z_2 \to 1 . \]
Any choice of $\psi \in \SS_2$ which maps to a topological generator of $\Z_2 = \Z_2^{\times}/(\pm 1)$ determines an isomorphism
\[\SS_2 \cong \SS_2^1 \rtimes \Z_2.\]
The group $G_{24}$ is contained in $\SS^1_2$ since $\Z_2$ is torsion free. 

In \cite{BeaudryRes, BeaudryTowards}, a useful choice 
of a splitting is made, defined using the element $\pi\in \SS_2$, which we recall now. For this, we note for the Witt vectors 
\[\WW:=W(\F_4) \cong \ZZ_2[\omega]/{(\omega^2+\omega+1)}\]
there is a natural inclusion $\WW^{\times}\subseteq \SS_2$. See \cite[Appendix 2]{ravgreen}.

\begin{defn}\label{rem:pialpha}
We let
\[\pi = 1+2\omega \in \W^\times \subseteq \SS_2,\]
where $\omega$ is a primitive third root of unity in $\WW^{\times}\subseteq \SS_2$.
We also define
\[\alpha = \frac{1-2\omega}{\sqrt{-7}}, \quad \sqrt{-7} \equiv 1 \mod 4 . \]
\end{defn}

Note that $\alpha \in \SS^1_2$, whereas $\pi \not\in \SS^1_2$. 

\begin{rem}\label{rem:defG24prime}
Although the subgroup $G_{24}$ is unique up to conjugacy in $\SS_2$, the subgroups $G_{24}$ and
\[ G_{24}' := \pi G_{24} \pi^{-1}\]
are not conjugate in $\mathbb{S}^1_2$. Both these subgroups of $\SS^1_2$ will appear later.
\end{rem}

Recall from \cite[\S 3]{BeaudryTowards} that $\mathbb{S}_2$ is the group of units in 
\[  \W\langle T \rangle/ (T^2+2, T \omega - \omega^{\sigma}T).\]
We let 
\[F_{i/2}\SS_2 = \{x \in \SS_2 \mid x \equiv 1 \mod T^i \}.\]
\begin{defn}\label{defn:K}
Let
\[
K := \overline{\langle \alpha, F_{3/2}\SS_2 \rangle} .
\]
We let $K^1 = \SS^1_2 \cap K$. 
\end{defn}

In the definition, the notation means that $K$ is the closure of the subgroup of $\SS_2$ generated by $\alpha$ and the elements of $F_{3/2}\mathbb{S}_2 $.
The group $K$ is a Poincar\'e duality group of dimension $4$, and an open normal subgroup of $\SS_2$ \cite{BeaudryRes}. There is a split exact sequence
\[ 1 \to K \to \SS_2 \to G_{24} \to 1,\]
which exhibits $G_{24}$ as a quotient rather than a subgroup. We thus get an isomorphism \cite{BeaudryRes}
\[\SS_2 \cong K \rtimes G_{24}.\]

\subsection{Cohomology Classes}\label{rem:chizeta}\label{sec:cohomologyclasses}
Ultimately, we want to access classes in the cohomology of $\GG_2$. 
 However, the inclusion $\SS_2 \to \GG_2$ induces an isomorphism
\[ H^*(\GG_2, \WW) \xrightarrow{\cong} H^*(\SS_2, \WW)^{\Gal}.\]
Furthermore, as $\Gal$-modules, we have an isomorphism
\[H^*(\SS_2, \WW) \cong H^*(\SS_2, \ZZ_2)\otimes_{\ZZ_2} \WW\]
with $H^*(\SS_2, \ZZ_2)$ viewed as the trivial $\Gal$-module and $\WW$ equipped with its canonical Galois action (see, e.g., \cite[Lemma 1.32]{BobkovaGoerss}). Since $\W \cong \Z_2[\Gal]$ as a $\Gal$-module, its fixed points are $\Z_2$, and in turn,   
any class in $H^*(\SS_2, \ZZ_2)$ naturally corresponds to an element in $H^*(\GG_2, \WW)$.

The homomorphism $\zeta$ defined in \eqref{eq:zeta} gives rise to our first important class
\begin{equation}\label{def:zeta}
\boxed{\zeta \in H^1(\SS_2, \Z_2) \cong H^1(\GG_2, \WW).}
\end{equation}

\begin{rem}The image of $\zeta$ in $H^1(\GG_2,\EE_0)$ survives to detect a same named homotopy class 
 $\zeta \in \pi_{-1}L_{K(2)}S^0$.
This class is discussed by Devinatz--Hopkins in \cite[Proposition 8.2]{DH}.
\end{rem}

Another important class is obtained by composing the determinant with the reduction modulo 4 map,
\begin{equation*}
\chi \colon \SS_2 \to  \Z_2^{\times}/(1+4\Z_2) \cong \Z/2.
\end{equation*}
This gives a class 
$\chi \in H^1(\SS_2,\Z/2)$ 
and its Bockstein is an element of order $2$ which is our second key class (see \cite{BGH} for details)
\begin{equation}\label{def:wchi}
 \boxed{\wchi \in H^2(\SS_2, \ZZ_2) \cong H^2(\GG_2, \WW).}
 \end{equation}

\begin{rem}
Under the maps induced by inclusion of subgroups (i.e., the restrictions), the classes $\chi$ and $\wchi$ map to non-trivial classes in the cohomology of $\GG^1_2 = \SS^1_2 \cap \GG_2$ and $\SS^1_2$ (with coefficients $\Z/2$ for $\chi$ and $\Z_2$ for $\wchi$), while $\zeta$ maps to zero in the cohomology of $\GG^1_2$ and $\SS^1_2$ (with $\Z_2$ coefficients).
\end{rem}

The third important class comes from the periodicity generator of the cohomology of $G_{24}$.
The composition of group homomorphisms $G_{24} \subseteq \SS_2 \to \SS_2/K \cong G_{24}$ gives an algebra splitting
\[H^*(G_{24},\ZZ_2) \xrightarrow{s} H^*(\SS_2, \ZZ_2) \xrightarrow{\mathrm{res}} H^*(G_{24},\ZZ_2)  \]
where the second map is the restriction in cohomology induced by the inclusion. Since $H^*(G_{24},\ZZ_2) \cong \ZZ_2[k]/8k$
for a class $k\in H^4(G_{24},\ZZ_2)$, we have
\begin{equation}\label{eq:defk}
\boxed{ k \in H^4(\SS_2, \ZZ_2) \cong H^4(\GG_2, \WW)} 
\end{equation}
We will see later that $k$ is related to the famous  element $\kappabar \in \pi_{20}S^0$.

\begin{rem}\label{rem:splitting}
We have a commutative diagram
\[\xymatrix{
G_{24} \ar[r] \ar@{=}[d] & \mathbb{S}_2^1 \ar[r] \ar[d]&  \mathbb{S}_2^1/K^1 \cong G_{24} \ar[d]^-{\cong} \\
G_{24} \ar[r]  & \mathbb{S}_2 \ar[r]   &  \mathbb{S}_2/K \cong G_{24}  
}\] 
so we get compatible splittings in the sense that the diagram
\[\xymatrix{
H^*(\mathbb{S}_2/K, \ZZ_2) \ar[d]^-{\cong} \ar[r]^-s &  H^*(\mathbb{S}_2, \ZZ_2) \ar[d] \ar[r] &  H^*(G_{24},\ZZ_2 )\ar@{=}[d] \\
H^*(\mathbb{S}_2^1/K^1, \ZZ_2) \ar[r]^-s &  H^*(\mathbb{S}_2^1, \ZZ_2) \ar[r] &  H^*(G_{24},\ZZ_2 )  
}\]
commutes, 
where the horizontal composites are isomorphisms.
The splittings also exist if we replace the coefficients $\ZZ_2$ by $\WW$.
\end{rem}

\subsection{Strategy and organization of the paper}

The strategy is to work up the iterated group extensions and perform the following sequence of computations
\begin{equation}  \label{eq:steps}
H^*(\SS^1_2, \EE_*/2)  \rightsquigarrow H^*(\SS^1_2, \EE_*)  \rightsquigarrow H^*(\SS_2, \EE_*) \rightsquigarrow H^*(\GG_2, \EE_*)
\end{equation}
where the arrows mean ``then''.

Let $\MM$ be a (graded) profinite continuous left $\Z_2[\![\SS^1_2]\!]$-module. We will call such an $\MM$ an \emph{$\SS_2^1$-module} for short.  The algebraic duality spectral sequence (ADSS) is a spectral sequence with signature
\begin{equation}\label{eq:adss-first-mention}
E_1^{p,q}(\MM):=H^q(F_p, \MM) \Rightarrow H^{p+q}(\SS_2^1, \MM)
\end{equation}
where $F_0=G_{24}$, $F_1=F_2 = C_6$ and $F_3=G_{24}'$.

The paper is organized by following the steps outlined in \eqref{eq:steps}. We start by recalling the (background) input to the ADSS for $\EE_*/2$ and $\EE_*$ in \Cref{sec:Cohomology}. The left-most step in \eqref{eq:steps}, i.e., the computation of $H^*(\SS^1_2, \EE_*/2)$, was performed in \cite{BeaudryTowards} using the ADSS and our calculation is based on that work.
We recall the main results of this computation in \Cref{sec:BeaudryTowards}. Our paper illustrates how simple it is to extract a computation of $H^*(\SS^1_2, \EE_*) $ from that of  $H^*(\SS^1_2, \EE_*/2)$ given enough information about the $d_1$-differentials of the ADSS, and we do most of this work in \Cref{sec:main}. We then determine $H^*(\SS_2,\EE_*)$ in our range $0 \leq t < 12$, performing the next step of \eqref{eq:steps} in \Cref{sec:Z2action}. Taking the Galois fixed points in the last step of \eqref{eq:steps} is rather straightforward and accomplished in the brief \Cref{sec:gal}. 

The last section of the paper is devoted to the computation of $d_3$-differentials in the homotopy fixed point spectral sequence 
\[ H^*(\G_2, \EE_*) \Rightarrow \pi_* L_{K(2)}S^0,\]
whose sources are in the range $0 \leq t < 12$. This is bootstrapped from the complete calculation of the $v_1$-localized analogous spectral sequence for the Moore spectrum done in \cite{BGH}.  Again, we invite the interested readers to study this spectral sequence for all values of $t$ as well as the higher differentials.

\subsection*{Acknowledgements}
This material is based upon work supported by 
the National Science Foundation under grants No. DMS-2005627, DMS-1906227 and DMS-1812122.
 A large portion of this work was conducted at the Max Planck Institute for Mathematics in Bonn and the authors would like to thank the MPIM for the hospitality. The authors would also like to thank the Hausdorff Research Institute for Mathematics for the hospitality in the context of the Trimester program Spectral Methods in Algebra, Geometry, and Topology, funded by the Deutsche Forschungsgemeinschaft (DFG, German Research Foundation) under Germany's Excellence Strategy -- EXC-2047/1 -- 390685813.
 
 The authors would also like to thank Mark Behrens and Mike Hill for helpful conversations.

\section{The cohomology of $G_{24}$ and $C_6$.}\label{sec:Cohomology}

The goal of this section is to gather information about cohomology of $G_{24}$ and $C_6$ with coefficients in $\EE_*$ and $\EE_*/2$. 
This information is required for describing the $E_1$-page of \eqref{eq:adss-first-mention} when $M=\EE_*$ and $\EE_*/2$. The results in this section are all well-known, and we simply collect them here to establish notation.

\subsection{Cohomology $H^*(C_6, \EE_*)$}
Recall that
\[\EE_* \cong \WW\br{u_1}[u^{\pm 1}], \quad u\in \EE_{-2}.\]
The central $C_2 =\{\pm 1 \} \subset \GG_2$ acts trivially on $\EE_0\cong \WW[\![u_1]\!]$ and by multiplication by $-1$ on $u$. With this action, 
\[
H^*(C_2, \EE_*) =\WW[\![u_1]\!][[u^{2}]^{\pm 1}, \alpha]/(2\alpha),
\]
where $\alpha \in H^1(C_2, \EE_2)$ is the image of the generator of $H^1(C_2 , \ZZ_{sgn})$ under the map which
sends the generator of the sign representation $\ZZ_{sgn}$ to $u^{-1}$. 

\begin{rem}
The square brackets around a cohomology class such as $[u^2]$ indicate that this is an indecomposable element in the cohomology ring.
\end{rem}

Since the pairing 
\[ H^1(C_2,\ZZ_{sgn})\otimes H^1(C_2,\ZZ_{sgn}) \to H^2(C_2,\ZZ) \] is an isomorphism and $[u^2]$ is an invariant unit in $\EE_0$, 
the class
\[g:=\alpha^2[u^2]\] generates $H^*(C_2,\EE_0)$ as a polynomial $\EE_0$-algebra.

The action of $C_3\cong \FF_4^{\times} =\langle \omega \rangle$ is given by 
\begin{align}\label{eq:omegaAction}
\omega_* u=\omega u  \qquad \omega_*u_1 =\omega u_1 \qquad \omega_*\alpha=\omega^{2}\alpha
\end{align}
and we have 
\[
H^*(C_6, \EE_*) \cong H^*(C_2, \EE_*)^{\FF_4^{\times}}.
\]
To describe this ring of invariants, define the following $\F_4^{\times}$-invariant elements
\begin{align*} 
w_5&:=u^{-2}\alpha \in H^1(C_6, \EE_6) & j_0&:=u_1^3 \in H^0(C_6,\EE_0) \\ 
[v_1v_2]&:=u_1u^{-4} \in H^0(C_6, \EE_8) & [v_2^2]&:=u^{-6} \in H^0(C_6, \EE_{12}).
\end{align*}

One verifies that the elements above generate the $\FF_4^\times$-invariant subring of $H^*(C_2, \EE_*)$. We then have the following explicit description.
\begin{lemma}\label{lem:C6coh}
There is an isomorphism
\[H^*(C_6, \EE_*) \cong \WW[\![j_0]\!]  \big[w_5, [v_1v_2], [v_2^2]^{\pm 1} \big]/ \big( 2w_5, [v_1v_2]^3 - j_0 [v_2^2]^2\big). 
\]
\end{lemma}

We also give special names to the following important invariant elements
\begin{align}\label{eq:C6cohElements}
\eta & := \alpha u_1 = w_5 [v_1v_2][v_2^2]^{-1} &  g&:= w_5^2[v_2^2]^{-1} \\
[v_1^2]&:= u_1^2u^{-2} =  [v_1v_2]^2 [v_2^2]^{-1} &
\mu&:= \eta [v_1^2].\nonumber 
\end{align}

\begin{rem}
Note that $\eta$-multiplication is not always surjective, as for example, 
\[ \eta [v_1^2] =j_0 w_5, \quad \eta^2 [v_1v_2] = j_0 w_5^2,\quad \eta^3= \nu j_0. \]
In \cref{fig:c6E}, this is depicted by the emergence of a new $\eta$-tower (of dots) at the positions where $\eta$-multiplication fails to be surjective. In particular, the different $\eta$-towers in a specific bi-degree are related by $j_0$-multiplication.
\end{rem}

 \begin{figure}[h]
 \includegraphics[page=1, width=0.9\textwidth]{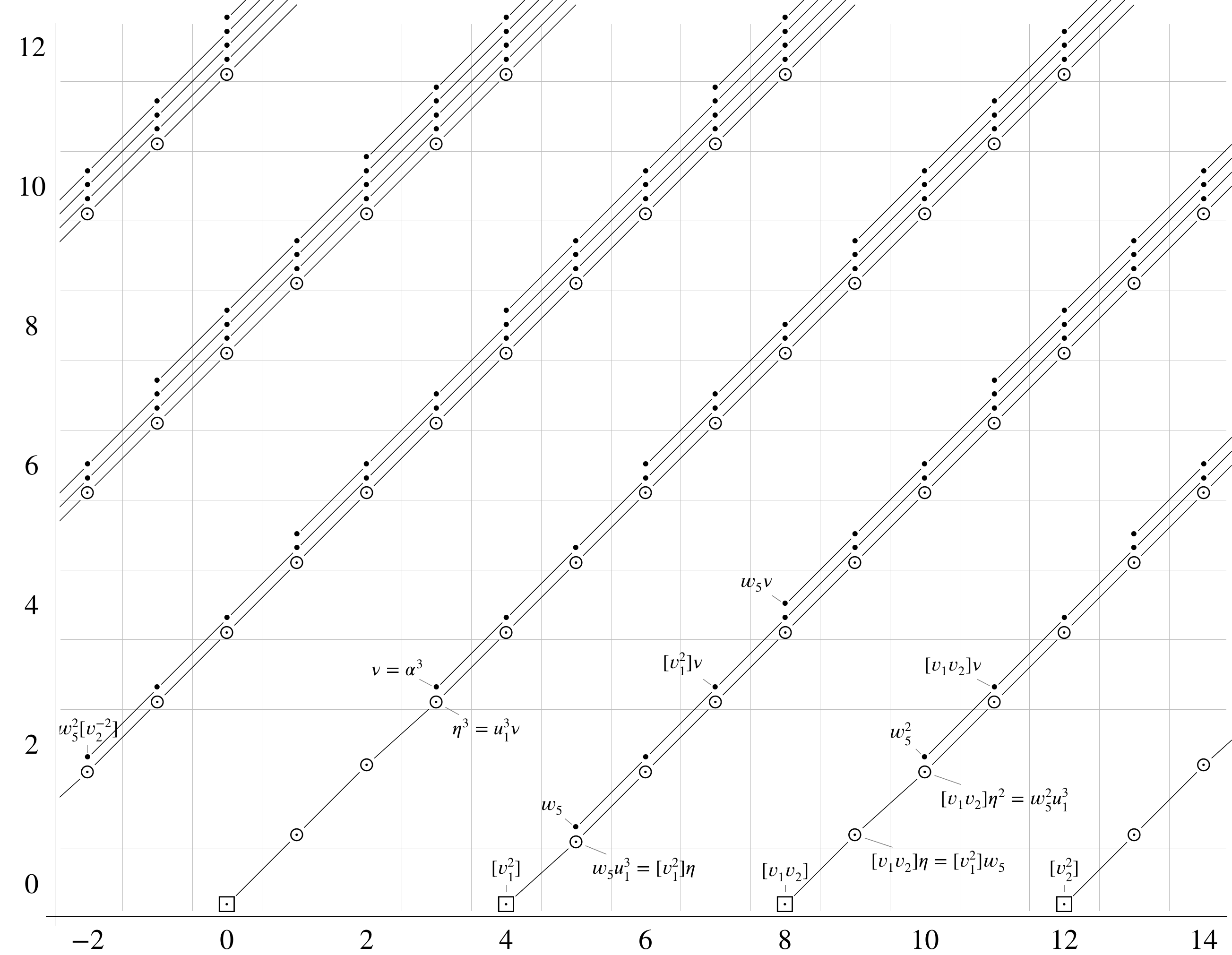}
\caption[$H^*(C_6, \EE_*)$]{$H^s(C_6, \EE_t)$. The horizontal axis is $t-s$, the vertical axis is $s$. A bullet denotes $\FF_4$, $\odot = \FF_4[\![j_0]\!]$ and $\boxdot = \WW[\![j_0]\!]$. The slope 1 lines are multiplication by $\eta$.}
\label{fig:c6E}
 \end{figure}
 
\subsection{Cohomology $H^*(C_6, \EE_*/2)$}
In a similar way, we can compute the cohomology of $C_6$ with coefficients in $\EE_*/2 \cong \FF_4[\![u_1]\!][u^{\pm 1}]$. 
Namely, $C_2$ acts trivially on $\EE_*/2$, giving 
\[ H^*(C_2, \EE_*/2) \cong \FF_4[\![u_1]\!][u^{\pm 1}] [h],\] 
where $h\in H^1(C_2,\EE_0/2)$.
 Note that $H^*(C_2,\EE_*/2)$ is a module over $H^*(C_2,\EE_*)$ and 
\[h = \alpha u.\] 
The action of $\FF_4^\times$ then follows from \eqref{eq:omegaAction}, and we can simply read off the invariants.

\begin{lemma}\label{lem:cohC6mod2}
Let $ v_1=u_1u^{-1}$ and $v_2=u^{-3}$. Then
\[
H^*(C_6, \EE_*/2)=\FF_4[\![j_0]\!][v_1, v_2^{\pm 1}, h]/(v_1^3-v_2 j_0)
\]
In this cohomology ring, $\eta=v_1h$.
\end{lemma}

It will be important to understand the map on cohomology induced by the reduction modulo $2$ 
\[\red \colon \EE_* \to \EE_*/2. \]
In the case of $C_6$, this is easily deduced from the definition of the elements.
\begin{lem}\label{lem:maprforC6}
The map
\[\red \colon H^*(C_6, \EE_*) \to H^*(C_6, \EE_*/2)  \]
is the ring map determined by
\begin{align*}
\red([v_1v_2])&= v_1v_2, & \red([v_2^2])&= v_2^2, & \red(j_0) &= j_0,  & \red(w_5) &= hv_2.
\end{align*}
\end{lem}

 \begin{figure}[h]
 \includegraphics[page=1, width=0.9\textwidth]{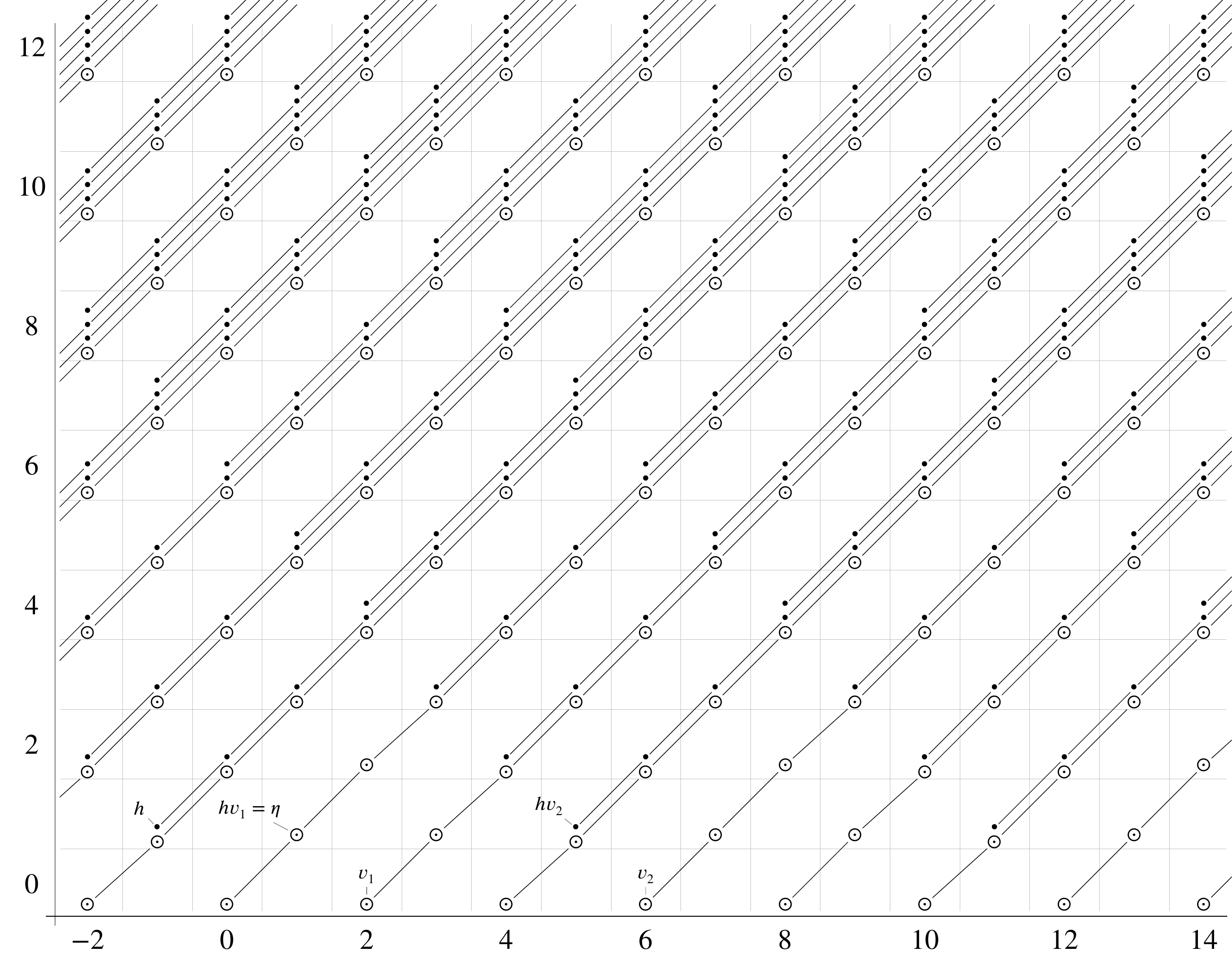}
\caption[$H^*(C_6, \EE_*/2)$]{$H^s(C_6, \EE_t/2)$. The horizontal axis is $t-s$, the vertical axis is $s$, $\bullet=\FF_4$ and $\odot = \FF_4[\![j_0]\!]$; and $v_1^3=v_2j_0$. Lines of slope 1 denote multiplication by $\eta$.
}
\label{fig:c6e2}
 \end{figure}

\subsection{Cohomology $H^*(G_{24}, \EE_*)$}

The computation of $H^*(G_{24}, \EE_*)$ is quite technical since the group action on $\EE_*$ is highly nontrivial. See, for example, \cite[\S 2.4]{BeaudryTowards}. It can be helpful to identify $H^*(G_{24}, \EE_*)$ with the completion of the cohomology of the Hopf algebroid $(A, \Gamma)$ of Weierstrass elliptic curves and their isomorphisms.
This computation is originally due to Hopkins and Mahowald, and accounts of it can be found in \cite{tbauer}, \cite{Rezk512} and \cite{tmfbook}, although a translation may be required in each case. Such translation can be found in \cite{BobkovaGoerss}, and we collect the main points below.

As in \cite{Strickland}, we use the Weierstrass curve
\[C   : y^2 +3u_1u^{-1}xy+(u_1^3-1)u^{-3}y=x^3, \] 
which is a universal deformation of a supersingular elliptic curve over $\FF_4[u^{\pm 1}]$.
Our Morava $E$-theory is then the Landweber exact theory whose formal group law is that of $C$. See \cite[\S 2]{BeaudryTowards}.
For the invariants we have
\[
H^0(G_{24},\EE_\ast) \cong \WW[\![j]\!][c_4,c_6,\Delta^{\pm 1}]/(c_4^3 - c_6^2 - 1728\Delta,\Delta j - c_4^3),
\]
where 
\begin{align}\label{eq:modularforms}
c_4&:=9u_1u^{-4}(u_1^3+8), & 
c_6&:=27u^{-6}(u_1^6-20u_1^3-8), \\
\Delta&:= 27 u^{-12} ( u_1^3-1)^3,  & \nonumber
j&:=c_4^3/\Delta.
\end{align}

\begin{rem}
Since $C_6 \subseteq G_{24}$, $H^0(G_{24},\EE_\ast)$ is a subring of the ring $H^0(C_6,\EE_*)$ described in \Cref{lem:C6coh}. In particular, under the restriction 
 \[ i \colon H^*(G_{24},\EE_*) \to H^*(C_6,\EE_*),\]
we have
\begin{align*}
i(c_4) &=9 [v_1v_2] (j_0 + 8), & i(c_6)&=27 [v_2^2](j_0^2-20j_0-8), &
i(\Delta) &=  27[v_2^2]^2 (j_0-1)^3.
\end{align*}
\end{rem}

\begin{rem}\label{rem:classesincoh}
To describe the higher cohomology of $G_{24}$, we need the following standard elements, and we specify their images under the map $i$: 
\begin{enumerate}[(1)]
\item $\eta \in H^1(G_{24},\EE_2)$ maps to $\eta \in H^1(C_6,\EE_2)$ defined in \eqref{eq:C6cohElements} and generates this group as a module over $H^0(G_{24},\EE_0) \cong \WW[\![j]\!]$; 
\item $\nu \in H^1(G_{24},\EE_4) \cong \W/4$ is a generator. This class maps to zero under $i$; 
\item $\mu \in H^1(G_{24},\EE_6)$ which corresponds to the Greek family element $\alpha_3$. It maps to $\eta [v_1^2] = \mu$ under $i$; 
\item $\epsilon \in H^2(G_{24},\EE_{10})$ and $\kappa \in H^2(G_{24},\EE_{16})$ detect the same named elements in homotopy (also called $\beta_2$ and $\beta_3$). They both map to zero under $i$;
\item $k \in H^4(G_{24},\EE_0)$ is the the periodicity generator of $H^*(G_{24}, \EE_0)$.
Under $i$ the element $k$ maps to $g^2 \in H^4(C_6, \EE_0)$.
\end{enumerate}
\begin{rem}\label{rem:kappabar}
The element $k$ is related to the homotopy class $\kappabar \in \pi_{20}{S^0}$ via
\[k\Delta=\kappabar.\]
\end{rem}

See also \Cref{sec:Hurewicz} for a discussion of the relationship with the Adams--Novikov spectral sequence $E_2$-page.
For more on classes in $H^*(G_{24},\EE_*)$ detecting elements in homotopy groups of spheres, see \cite[Ch.13]{tmfbook} and \cite{BehrensQuigley}.
\end{rem}

\begin{thm}[{\it{c.f.}{\cite[Theorem 2.15]{BobkovaGoerss}}}]\label{G24}
There is an isomorphism 
\[
H^\ast(G_{24},\EE_\ast) \cong
H^0(G_{24},E_\ast)[\eta,\nu,\mu,\epsilon,\kappa,k]/R,
\] 
where $R$ is the ideal generated by 
\begin{gather*}
2\eta,\quad 4\nu,  \quad 2\mu, \quad \nu c_4, \quad \mu c_4-\eta c_6 , \quad \nu c_6,  \quad \mu c_6-\eta c_4^2 ;\\
 \eta\nu,  \quad 2\nu^2,  \quad  2\epsilon,\quad \mu\nu,  \quad \mu^2-\eta^2c_4, \quad 2\kappa,  \quad c_4\epsilon, \quad c_6\epsilon, \quad c_4\kappa, \quad c_6\kappa;\\
\eta\epsilon-\nu^3, \quad \nu\epsilon, \quad \mu\epsilon, \quad \mu\kappa,  \quad \epsilon\kappa ;\\
\epsilon^2, \quad  8k,\quad  \nu^2\kappa-4k\Delta,  \quad c_4k-\eta^4,  \quad  \kappa^2, \quad c_6k-\eta^3\mu.
\end{gather*}
\end{thm}

The cohomology $H^*(G_{24}, \EE_*)$ is displayed in \Cref{K2localE2}. 

\begin{figure}[h]
 \includegraphics[page=1, width=1\textwidth]{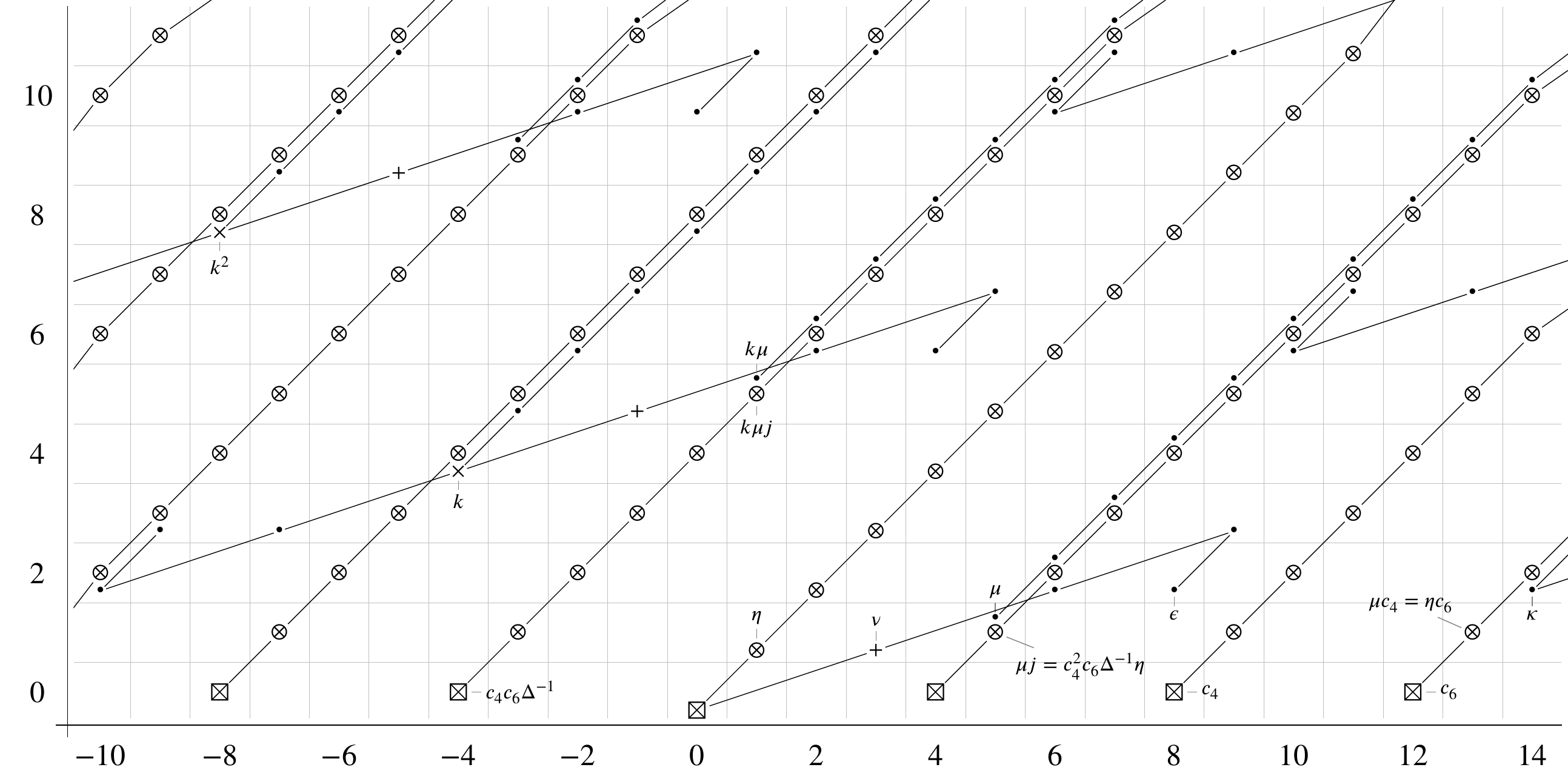}
\caption[$H^*(G_{24}, \EE_*)$]{$H^*(G_{24}, \EE_*)$, where $\otimes=\mathbb{F}_4\br{j}$, $\boxtimes=\WW\br{j}$, $\bullet$ is $\W/2$, $+$ is $\WW/4$ and $\times$ is $\W/8$. The $\W/8$ are generated by  classes of the form $\Delta^ik^j$ and the $\W/4$ by their $\nu$ multiples. 
The solid lines of slope 1 represent multiplication by $\eta$ and solid lines of slope $1/3$ denote multiplication by $\nu$. The infinite $\eta$-towers in the same bi-degrees are related by multiplication by $j$. Multiplication by $\Delta$ gives an isomorphism $H^s(G_{24}, \EE_t)\cong H^{s}(G_{24}, \EE_{t+24})$ for all $t, s$ and multiplication by $k$ gives an isomorphism $H^s(G_{24}, \EE_t)\cong H^{s+4}(G_{24}, \EE_{t})$ for $s>0$.}
\label{K2localE2}
 \end{figure}
 
 \subsection{Cohomology $H^*(G_{24}, \EE_*/2)$}
A reference for this computation is Appendix by Henn in \cite{BeaudryTowards}. 
First, the invariants are given by
\begin{align}\label{eq:mod2invariants}
H^0(G_{24},\EE_\ast/2) \cong \FF_4[\![j]\!][v_1, \Delta^{\pm 1}]/(v_1^{12}-j\Delta)
\end{align}
where 
\[v_1 := u_1u^{-1}\] 
while $\Delta$ and $j$ are the reduction modulo $2$ of the same named classes in \eqref{eq:modularforms}.

\begin{rem}\label{rem:xyboundaries}
To describe $H^*(G_{24},\EE_\ast/2)$, we need to introduce some classes:
\begin{enumerate}[(1)]
\item The classes $\eta$, $\nu$, $k$ are the reductions modulo 2 of the classes of \Cref{rem:classesincoh}.
\item Since $2\nu^2=2\kappa=0$, the classes $\nu^2$ and $\kappa$ are in the image of the connecting homomorphism 
\[\partial \colon H^1(G_{24}, \EE_*/2) \to H^2(G_{24}, \EE_*)\]
for the exact sequence $\EE_* \xrightarrow{2} \EE_* \to \EE_*/2$.
So there are classes $x \in H^1(G_{24},\EE_8/2)$ and $y \in H^1(G_{24},\EE_{16}/2)$ such that $\partial(x)=\nu^2$ and $\partial(y)=\kappa$. 
\end{enumerate}
\end{rem}

\begin{rem}
Note that $\partial(v_1 x) =  \epsilon$.
\end{rem}

\begin{thm}[\Cref{fig:c6e2v0}, {\it{c.f.}} {\cite[Appendix]{BeaudryTowards}}]\label{G24-mod2}
There is an isomorphism 
\[
H^\ast(G_{24},\EE_\ast/2) \cong
H^0(G_{24},\EE_\ast/2)[ \eta, \nu, x, y,k]/R,
\]
where $R$ is the ideal defined by the relations
\begin{gather*}
v_1\nu, \qquad v_1^2x \qquad v_1y;\\
\eta \nu, \quad \nu x-v_1\eta x, \quad \eta y-v_1x^2, \quad xy, \quad y^2-\nu^2\Delta; \\
\eta^2 x-\nu^3, \quad x^3-\nu^2y;\\
\eta^4-v_1^4k. 
\end{gather*}
\end{thm}

 \begin{figure}[h]
 \includegraphics[page=1, width=\textwidth]{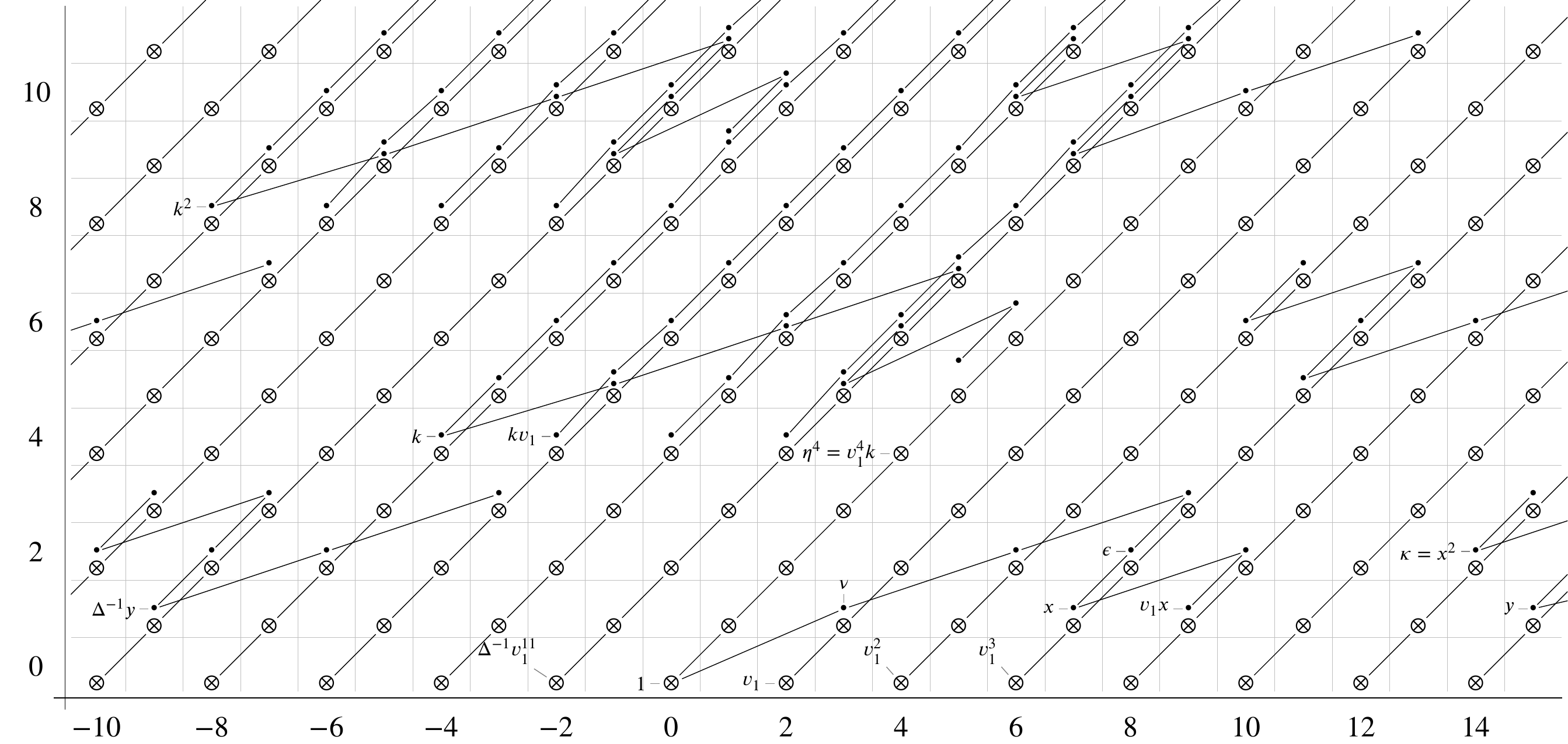}
\caption[$H^*(G_{24}, \EE_*/2)$]{$H^*(G_{24}, \EE_*/2)$, where $\otimes  =\F_4[\![j]\!]$ and $\bullet = \F_4$. Lines of slope one are multiplication by $\eta$, lines of slope 1/3 are multiplication by $\nu$. The  infinite $\eta$-towers in the same bi-degree are related by multiplication by $j$. Multiplication by $\Delta$ gives an isomorphism $H^s(G_{24}, \EE_t/2)\cong H^{s}(G_{24}, \EE_{t+24}/2)$ for all $t, s$ and multiplication by $k$ gives an isomorphism $H^s(G_{24}, \EE_t/2)\cong H^{s+4}(G_{24}, \EE_{t}/2)$ for $s\geq 0$.}
\label{fig:c6e2v0}
 \end{figure}
 
Finally, we write down the map on cohomology induced by $r \colon \EE_* \to \EE_*/2$.
\begin{lemma}\label{lem:from-integral-to-mod2}
The map 
\[
\red: H^*(G_{24}, \EE_*) \to H^*(G_{24}, \EE_*/2) 
\]
is the ring homomorphism determined by
\begin{align*}
r(j)&=j & r(c_4) &= v_1^4 & r(c_6)&=v_1^6 & r(\Delta)=\Delta \\
r(\eta)&=\eta & r(\nu)&=\nu & r(\mu)&=v_1^2\eta  \\
r(\epsilon) &=\eta x & r(\kappa)&=x^2 & r(k)&=k . 
\end{align*}
\end{lemma}

\section{The Algebraic Duality Spectral Sequence for $\EE_*/2$}\label{sec:BeaudryTowards}
This section consists of recollections of computations in \cite{BeaudryTowards}.

\subsection{The construction of the spectral sequence}
The  \emph{algebraic duality resolution} is an exact sequence
\[  0\to \sC_3 \to \sC_2 \to \sC_1 \to \sC_0 \to \Z_2 \to 0 \]
of $\SS_2^1$-modules, where 
\[\sC_0 = \Z_2[\![\SS_2^1/G_{24}]\!], \quad \sC_1=\sC_2=\Z_2[\![\SS_2^1/C_{6}]\!], \quad \text{and} \quad \sC_3 = \Z_2[\![\SS_2^1/G_{24}']\!].\]
We write $F_0=G_{24}$, $F_1=F_2 =C_6$, and $F_3 = G_{24}'$, so that $\sC_p = \ZZ_2[\![\SS_2^1/F_p]\!]$ for $0\leq p\leq 3$ (see \cref{sec:subgroups}). We let $\sC_p=0$ otherwise.

  \begin{warn}
  Note that $\sC_3$ is  not isomorphic to $\Z_2[\![\mathbb{S}_2^1/G_{24}]\!]$ as an $\mathbb{S}_2^1$-module. 
 \end{warn}

 The $\sC_p$'s are not projective $\SS_2^1$-modules, but their projective $\SS_2^1$-resolutions $P_{p,*}$ give rise to a double complex. Mapping this double complex into an $\SS_2^1$-module $\MM$ gives a double complex $E_0^{p,q}(\MM) = \Hom_{\Z_2[\![\SS_2^1]\!]}^c(P_{p, q}, \MM)$, such that the cohomology of its total complex is $H^*(\SS_2^1,\MM)$. The Algebraic Duality Spectral Sequence (ADSS) is the spectral sequence obtained by filtering this double complex by columns (i.e. taking ``vertical cohomology'' in the $(p,q)$-plane).
 
To identify the $E_1$-page of this spectral sequence, we use Shapiro's lemma 
 \[\Ext_{\Z_2[\![\SS_2^1]\!]}^q(\sC_p, \MM) \cong \Ext_{\Z_2[F_p]}^q(\Z_2, \MM) \cong H^q(F_p, \MM) .\] 
 Thus, the spectral sequence takes form
 \begin{equation}
\label{eq:ADSS}
E_1^{p,q}(\MM)=H^q (F_p, \MM) \Longrightarrow H^{p+q} (\SS_2^1, \MM),
\tag{ADSS}
\end{equation}
with differentials $d_r\colon E_r^{p,q}(\MM) \to E_r^{p+r, q-r+1}(\MM)$.

\subsection{Module structure of the spectral sequence}\label{sec:moduleoverE}

The  \ref{eq:ADSS} is not a multiplicative spectral sequence. 
However, if $\MM$ is a (possibly graded) ring, for example when $\MM=\WW$, $\EE_*$ or $\EE_*/2$, the ADSS is a module over its target $H^*(\SS_2^1,\MM)$ \cite[Lemma 4.1.3]{BeaudryTowards}. 
The identification
\[E_1^{p,*}(\MM)\cong H^*(F_p, \MM),  \]
is such that the action of $H^*(\mathbb{S}_{2}^1,\MM)$ is via the restriction
\[ H^*(\mathbb{S}_{2}^1,\MM) \to  H^*(F_p, \MM)  \]
induced by the inclusion $F_p \subseteq \mathbb{S}_2^1$.  

As a consequence, if a class $x\in H^*(\mathbb{S}_{2}^1,\MM) $ maps to zero in $H^*(F_p, \MM)$, then for any $y\in H^*(F_p, \MM)$, the product $xy$ is detected in ADSS filtration $\geq p+1$.

This might sound circular since the spectral sequence converges to $H^*(\SS_2^1,\MM)$ and this is what we aim to compute, but some elements of $H^*(\SS_2^1,\MM)$ are easy to spot, e.g., classes coming from the classical Adams--Novikov $E_2$-page like $\eta$ and $\nu$ (see \cref{sec:Hurewicz}), or the key classes defined in \cref{sec:cohomologyclasses}. 

\begin{warn}
For $p=3$, we may identify $E_1^{3,\ast}(\MM)$ with $H^*(G_{24},\MM)$, but if we do so, there is a non-trivial twist coming from the fact that the last module $\mathscr{C}_3$ in the algebraic duality resolution is $\Z_2[\![\mathbb{S}_2^1/G_{24}']\!]$. Indeed,
the action of $H^*(\mathbb{S}_2^1, \MM) $ on $E_1^{3,*}(\MM) = H^*(G_{24}', \MM)$ is given by the restriction. Therefore, the action on $E_1^{3,*}(\MM)$ after identifying $H^*(G_{24}', \MM) $ with $H^*(G_{24}, \MM)  $ is given by the composition
  \[ H^*(\mathbb{S}_2^1, \MM)  \xrightarrow[\cong]{\pi^{-1}(-)\pi} H^*(\mathbb{S}_2^1,  \MM) \to   H^*(G_{24}',  \MM) . \]
  \end{warn}
  
\subsection{The edge homomorphism and $k$-linearity}
We quickly discuss the edge homomorphism of the ADSS for  $E_*^{\bullet,*}(\WW)$,  $E_*^{\bullet,*}(\EE_0)$. 
Everything in this section also holds true if we replace $\WW$ by $\FF_4$ and $\EE_0$ by $\EE_0/2$, with the obvious adjustments.

The inclusion $\WW \to \EE_0$ by the unit induces 
a commutative diagram
\[ 
\xymatrix{
H^q(\SS_2^1, \WW) \ar[d]^{\cong} \ar@{->>}[r] & E_{\infty}^{0,q}(\WW) \ar[d]^-{\cong} \ar[r]^-{\subseteq}_-{\text{edge}}  & E_1^{0,q}(\WW) \ar[d] \cong H^q(G_{24}, \WW) \ar[d] \\
H^q(\SS_2^1, \EE_0) \ar@{->>}[r] & E_{\infty}^{0,q}(\EE_0) \ar[r]^-{\subseteq}_-{\text{edge}}  & E_1^{0,q}(\EE_0) \cong H^q(G_{24}, \EE_0).
} \]
The left isomorphism is an instance of the Chromatic Vanishing Conjecture \cite[Theorem 6.1.6]{BGH}. The middle is isomorphism
 \cite[Lemma 6.1.5]{BGH}. It follows from \cref{rem:splitting} that the edge homomorphism of $E_*^{\bullet,*}(\WW)$ splits, since the edge is the restriction induced by the inclusion $G_{24} \subseteq \SS_2^1$. Using the diagrams of \cref{rem:splitting}, we get the following result. See also \cite[Lemma 5.2.2]{BGH}.

\begin{lem} \label{lem:edge}
The edge homomorphism of the ADSS $E_*^{\bullet,*}(\WW)$
\[ed : H^*(\SS_2^1,\WW ) \to E_1^{0,*}(\WW) \cong H^*(G_{24},\WW),\]
is split via a ring map which decomposes as
\[H^*(G_{24},\WW) \xrightarrow{s} H^*(\SS_2,\WW) \to  H^*(\SS_2^1,\WW )^{\ZZ_2} \xrightarrow{\subseteq} H^*(\SS_2^1,\WW ) \xrightarrow{ed} H^*(G_{24},\WW).\] 
\end{lem}

\begin{cor}\label{prop:kpc}\label{lem:kfixed}
The element $k \in H^4(\GG_2, \WW)$ of \eqref{eq:defk} has non-trivial image in $H^4(\SS_2^1, \EE_0)$
detected by $k \in E_1^{0,4}(\EE_0)$ in the ADSS. The ADSS $E_*^{\bullet,*}(\EE_*)$
is thus a module over $\W[k]$ and so the differentials are $k$-linear.  
\end{cor}
\begin{proof}
From \Cref{lem:edge} and the diagram preceeding it, we get an element $k \in H^4(\mathbb{S}_2^1, \EE_0)^{\ZZ_2}$ detected by $k\in E_1^{0,4}(\EE_0)$ in the ADSS. The rest follows from the module structure of the spectral sequence over $H^*(\SS_2^1, \EE_*)$ (see \cref{sec:moduleoverE}).
\end{proof}

\subsection{The spectral sequence for $\EE_*/2$}
The \ref{eq:ADSS} with $\MM=\EE_*/2$ has been computed in  \cite{BeaudryTowards}. We summarize some key features of this computation. Note that these are ``computational'' facts that one learns from \cite{BeaudryTowards}, and they do not all have ``theoretical'' justifications. 
\begin{enumerate}[(a)]
\item The spectral sequence collapses at the $E_2$-page. Since it's a spectral sequence of $\F_4$-vector spaces, there are no additive extensions.
\item The spectral sequence is a module over the subring  (see \cite[Theorem 1.2.1]{BeaudryTowards})
\[\FF_4[v_1, \eta, k]/(\eta^4-v_1^4k) \subseteq H^*(\SS_2^1,\EE_*/2).\]
\item The spectral sequence is determined by the cohomology when $q=0$, i.e., by the cohomology of the chain complex
{\tiny
\[\xymatrix{
0 \ar[r] &E_1^{0,0}(\EE_*/2) \ar@{=}[d]\ar[r]^-{d_1} & E_1^{1,0}(\EE_*/2) \ar@{=}[d]\ar[r]^-{d_1} &  E_1^{2,0}(\EE_*/2) \ar@{=}[d] \ar[r]^-{d_1} &E_1^{3,0}(\EE_*/2) \ar[r]\ar@{=}[d] & 0 \\
0 \ar[r] &H^0(G_{24},\EE_*/2) \ar@{=}[d]\ar[r]^-{d_1} & H^0(C_6,\EE_*/2) \ar@{=}[d]\ar[r]^-{d_1} &  H^0(C_6,\EE_*/2) \ar@{=}[d] \ar[r]^-{d_1} &H^0(G_{24},\EE_*/2) \ar[r]\ar@{=}[d] & 0 \\
0 \ar[r] & \frac{\F_4[\![j]\!][v_1, \Delta^{\pm 1}]}{v_1^{12}-j\Delta} \ar[r]^-{d_1} &  \frac{\F_4[\![j_0]\!][v_1, v_2^{\pm 1}]}{v_1^3-j_0v_2} \ar[r]^-{d_1} &  \frac{\F_4[\![j_0]\!][v_1, v_2^{\pm 1}]}{v_1^3-j_0v_2}\ar[r]^-{d_1} & \frac{\F_4[\![j]\!][v_1, \Delta^{\pm 1}]}{v_1^{12}-j\Delta} \ar[r] & 0,  } \]}
which is completely described in Theorem 1.2.1 of \cite{BeaudryTowards}.
This is because of the following facts:
\begin{enumerate}[i.]
\item $d_1 \colon E_1^{0,*}(\EE_*/2) \to E_1^{1,*}(\EE_*/2)$ maps any $v_1$-torsion class to zero. Any $v_1$-free class in positive $q$ degree is a multiple of $\eta$ or $k$. 
\item $d_1 \colon E_1^{1,*}(\EE_*/2) \to E_1^{2,*}(\EE_*/2)$ is an $h$-linear map.
Any class in positive $q$-degree in $E_1^{p,*}(\EE_*/2)$ for $p=1,2$ is a multiple of $h$.
\item $d_1 \colon E_1^{2,*}(\EE_*/2) \to E_1^{3,*}(\EE_*/2)$ can also be thought of as an $h$-linear map, even if there is no element called $h$ in $E_1^{3,*}(\EE_*/2)$, as follows. Note that the image of $d_1$ lies in the subring 
\[\frac{ \F_4[\![j]\!][\eta, v_1, k, \Delta^{\pm 1}]}{(v_1^{12} - j\Delta, \eta^4 - v_1^4k)}\subseteq H^*(G_{24},\EE_*/2). \]
Then for any $x$, we have
\[d_1(hx)= v_1^{-1}\eta d_1(x). \]
\end{enumerate}
\end{enumerate}

Computing the $E_2$-page of the \ref{eq:ADSS} for $\EE_*/2$ is complicated, and the interested reader should consult \cite[Theorem 1.2.2]{BeaudryTowards} for the full results. Here we only focus on the range $0 \leq \ast < 12$, and we can be even more explicit.
 As in {\it loc.cit.}, we choose suitable generators $\Delta_n \in E_1^{0,0}(\EE_*/2), b_n \in E_1^{1,0}(\EE_*/2), \overline{b}_n \in E_1^{2,0}(\EE_*/2)$ and $\overline{\Delta}_n \in E_1^{3,0}(\EE_*/2)$ with 
\[
\Delta_n \equiv \Delta^n ,\qquad b_n \equiv \overline{b}_n\equiv v_2^n, \qquad \overline{\Delta}_n \equiv \Delta^n  
\]
where the congruences are modulo $v_1$.\footnote{In \cite{BeaudryTowards}, $\overline{\Delta}_n \equiv (\Delta')^n$ where $\Delta'$ is the image of $\Delta$ under the isomorphism $H^*(G_{24}, \EE_*/2) \cong H^*(G_{24}', \EE_*/2)$ induced by conjugation by $\pi$. In this paper, we most often apply this isomorphism implicitly, so we obscure this distinction. The only place where we will need to remember the difference is in the proof of \cref{thm:techyd1} below.}  These generators play well with the differentials, and in our range they consist of
\[\Delta_0, b_0, b_1, b_2, \overline{b}_0, \overline{b}_1, \overline{b}_2, \overline \Delta_0.\]
All these are $d_1$-cycles, and we use them to summarize the needed calculation.

\begin{rem}\label{rem:defdbs}
The classes $\Delta_0$, $b_0$, $\overline{b}_0$, $\overline{\Delta}_0$ are equal to the units in $E_1^{p,0}(\EE_*/2)$, see \cite[Remark 5.3.3]{BeaudryTowards}. However, the definition of $b_1$, $b_2$, $\overline{b}_1$ and  $\overline{b}_2$ is more complicated. In \cite{BeaudryTowards}, the element $b_1$ is defined in Proposition 5.2.4, $b_2$ in the proof of Proposition 5.2.1, $\overline{b}_1$ in Lemma 5.3.2 and $\overline{b}_2$ in Proposition 5.3.1.
\end{rem}

\begin{thm}[{\cite[Theorem 1.2.1]{BeaudryTowards}}]\label{thm:Einfmod2}
The $d_1$-differentials in the \ref{eq:ADSS} for $\MM=\EE_*/2$ for $q=0$ are determined by

\begin{align*}
&d_1(\Delta_n) & &= &&
\begin{cases}
v_1^{6\cdot 2^r} b_{2^{r+1}(1+4t)},    &n=2^r(1+2t)\\
0, &n=0
\end{cases}\\
&
d_1(b_n) & &= &&
\begin{cases}
v_1^{3\cdot 2^r} \overline{b}_{2^{r+1}(1+2t)}, &n=2^r(3+4t)\\
v_1^{3\cdot 2^{r+1}} \overline{b}_{1+2^{r+1}(1+4t)}, &n=1+2^{r+2}(1+2t)\\
0, &\text{otherwise}
\end{cases}\\
&
d_1(\overline{b}_n) & &=&&
\begin{cases}
v_1^{3(2^{r+1}+1)} \overline{\Delta}_{2^{r}(1+2t)}, &n=1+2^{r+1}(3+4t)\\
0, &\text{otherwise}
\end{cases}
\end{align*}
and $v_1$-linearity.
Consequently, using $\eta$ and $k$ linearity, we have
\begin{align*}
E_2^{p,*}(\EE_0/2) &\cong  \begin{cases}
\F_4[k]\{\Delta_0, \nu^2y\Delta_{-1} \}&  \text{$p=0$}  \\
\F_4[h] \{b_0 \} & \text{$p=1$}  \\
\F_4[h] \{\overline{b}_0 \} & \text{$p=2$}  \\
\F_4[k]\{\overline{\Delta}_0, \nu^2y\overline{\Delta}_{-1}\} &  \text{$p=3$}  
\end{cases}
\end{align*}
\begin{align*}
E_2^{p,*}(\EE_2/2) &\cong  \begin{cases}
\F_4[k]\{v_1\Delta_0, \eta \Delta_0\} &  \text{$p=0$}  \\
\F_4[h]\{  v_1 b_0 \} & \text{$p=1$}  \\
\F_4[h]\{ v_1  \overline{b}_0 \} & \text{$p=2$}  \\
\F_4[k]\{v_1\overline{\Delta}_0, \eta\overline{\Delta}_0\} &  \text{$p=3$} 
\end{cases} 
\end{align*}
\begin{align*}
E_2^{p,*}(\EE_4/2) &\cong \begin{cases} 
\F_4[k]\{ v_1^2\Delta_0,\nu \Delta_0, v_1\eta\Delta_0, \eta^2\Delta_0\} & p=0 \\
\F_4[h]\{v_1^2b_0\} &  p=1 \\
\F_4[h]\{ v_1^2\overline{b}_0\} & p=2 \\
\F_4[k]\{v_1^2 \overline\Delta_0, \nu \overline\Delta_0,  v_1\eta\overline\Delta_0, \eta^2\overline\Delta_0\}   & p=3
\end{cases}
\end{align*}
\begin{align*}
E_2^{p,q}(\EE_6/2) &\cong  \begin{cases}
\F_4[k]\{v_1^3\Delta_0, v_1^2\eta\Delta_0, v_1\eta^2\Delta_0, \eta^3 \Delta_0\} & p=0 \\
\F_4[h]\{  v_1^3 b_0,  b_1  \} & p=1  \\
\F_4[h]\{ v_1^3 \overline{b}_0,  \overline{b}_1 \} & p=2 \\
\F_4[k]\{v_1^3\overline\Delta_0, v_1^2\eta\overline\Delta_0, v_1\eta^2\overline\Delta_0, \eta^3 \overline\Delta_0\}&  p=3
\end{cases}
\end{align*}
\begin{align*}
E_2^{p,q}(\EE_8/2) &\cong  \begin{cases}
\F_4[k]\{v_1^4\Delta_0,  x\Delta_0, v_1^3\eta\Delta_0,  \nu^2 \Delta_0, v_1^2\eta^2\Delta_0, v_1\eta^3 \Delta_0\} & p=0 \\
\F_4[h]\{  v_1^4 b_0,  v_1b_1  \} & p=1  \\
\F_4[h]\{ v_1^4 \overline{b}_0,  v_1\overline{b}_1 \} & p=2 \\
\F_4[k]\{v_1^4\overline\Delta_0, x\overline\Delta_0, v_1^3\eta\overline\Delta_0,  \nu^2 \overline\Delta_0,  v_1^2\eta^2\overline\Delta_0, v_1\eta^3 \overline\Delta_0\}&  p=3
\end{cases}
\end{align*}
\begin{align*}
E_2^{p,q}(\EE_{10}/2) &\cong  \begin{cases}
\F_4[k]\{v_1^5\Delta_0, v_1x\Delta_0,   v_1^4\eta\Delta_0, \eta x \Delta_0, v_1^3\eta^2\Delta_0, v_1^2\eta^3 \Delta_0\} & p=0 \\
\F_4[h]\{  v_1^5 b_0,  v_1^2b_1  \} & p=1  \\
\F_4[h]\{ v_1^5 \overline{b}_0,  v_1^2\overline{b}_1 \} & p=2 \\
\F_4[k]\{v_1^5\overline\Delta_0,  v_1x\overline\Delta_0,  v_1^4\eta\overline\Delta_0,  \eta x \overline\Delta_0, v_1^3\eta^2\overline\Delta_0, v_1^2\eta^3 \overline\Delta_0\}&  p=3.
\end{cases}
\end{align*}

\begin{align*}
E_2^{p,q}(\EE_{12}/2) &\cong  \begin{cases}
\F_4[k]\{v_1^6\Delta_0,   v_1^5\eta\Delta_0, v_1 \eta x \Delta_0, v_1^4\eta^2\Delta_0, v_1^3\eta^3 \Delta_0, \nu^3\Delta_0\} & p=0 \\
\F_4[h]\{  v_1^6 b_0,  v_1^3b_1, b_2  \} & p=1  \\
\F_4[h]\{ v_1^6 \overline{b}_0,  v_1^3\overline{b}_1, \overline{b}_2 \} & p=2 \\
\F_4[k]\{v_1^6\overline\Delta_0,   v_1^5\eta\overline\Delta_0, v_1 \eta x\overline \Delta_0, v_1^4\eta^2\overline\Delta_0, v_1^3\eta^3\overline \Delta_0, \nu^3\overline\Delta_0\}&  p=3.
\end{cases}
\end{align*}
\end{thm}

The spectral sequence is depicted in \Cref{fig:ADSS-Mod2-E1} and \Cref{fig:ADSS-Mod2}.

In \cite[Theorem 8.2.5]{BGH}, the structure of $v_1^{-1}H^*(\SS_2^1, \EE_*/2)$ is computed as an algebra. Namely, it is shown that 
\begin{align*}
v_1^{-1}H^*(\SS_2^1, \EE_*/2) \cong \F_4[v_1^{\pm 1}, \eta, \sigma, \chi]/(\sigma^2, \chi^3)
\end{align*}
where 
\begin{enumerate}
\item $\chi$ is detected by $b_0$, 
\item $\chi^2$ by $\overline{b}_0$, 
\item $\sigma$ by $v_1 b_1$,  
\item $\sigma \chi$ by $a_0h\overline{b}_1+a_1 v_1^4\overline{b}_0$ for $a_0 \in \F_4^{\times}$ and $a_1 \in \F_4$; and
\item $\sigma \chi^2$ by $v_1^4\overline{\Delta}_0$.
\end{enumerate}
This, and the fact that the action of $H^*(\SS_2^1, \EE_*/2) $ on the ADSS preserves filtration allows us to describe many products. We use this to name classes in $H^*(\SS_2^1, \EE_*/2)$. See \cref{fig:tableV0}. 

\newpage
\begin{figure}[H]
 \includegraphics[page=1, width=0.48\textwidth]{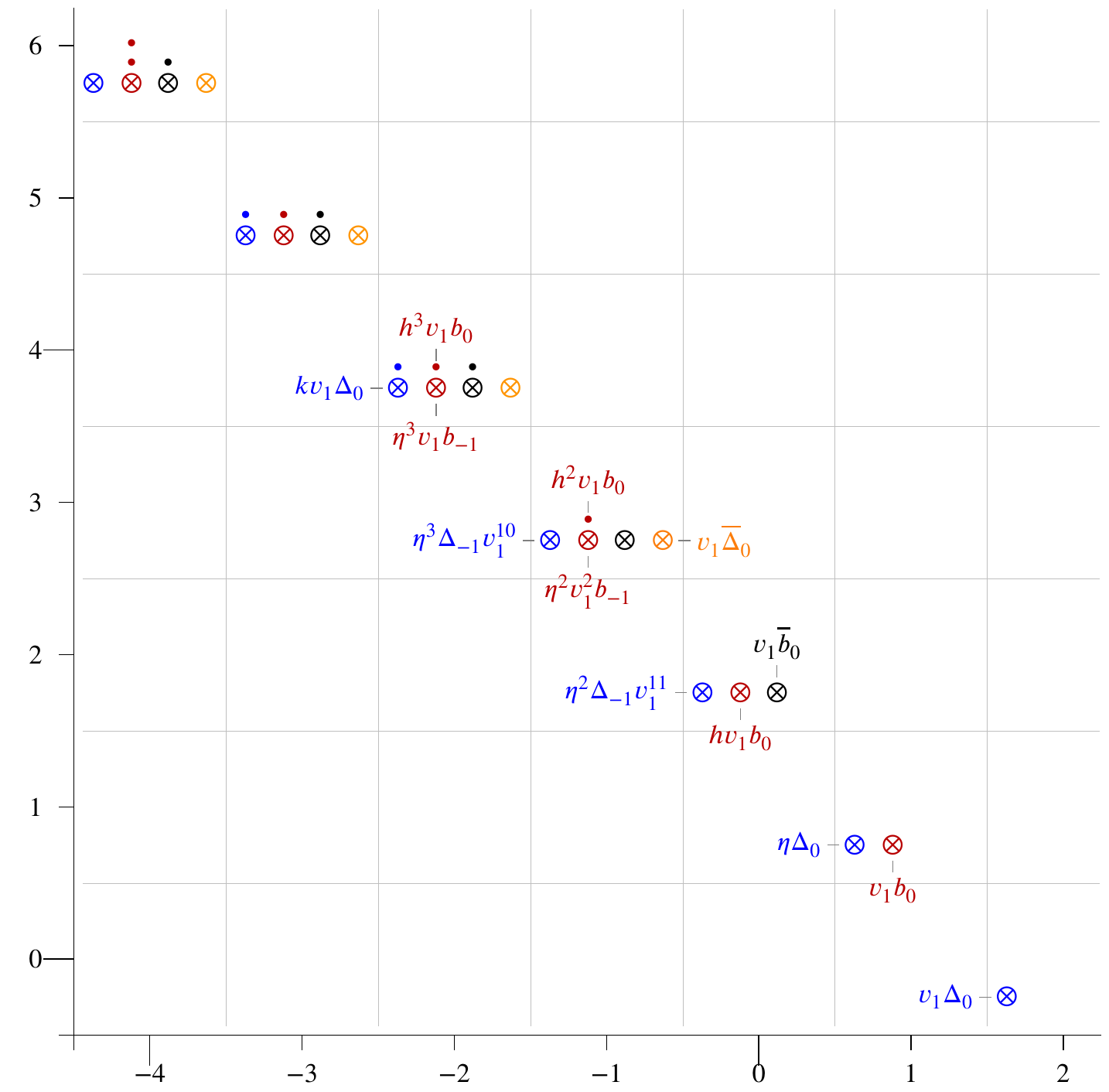}
  \includegraphics[page=1, width=0.48\textwidth]{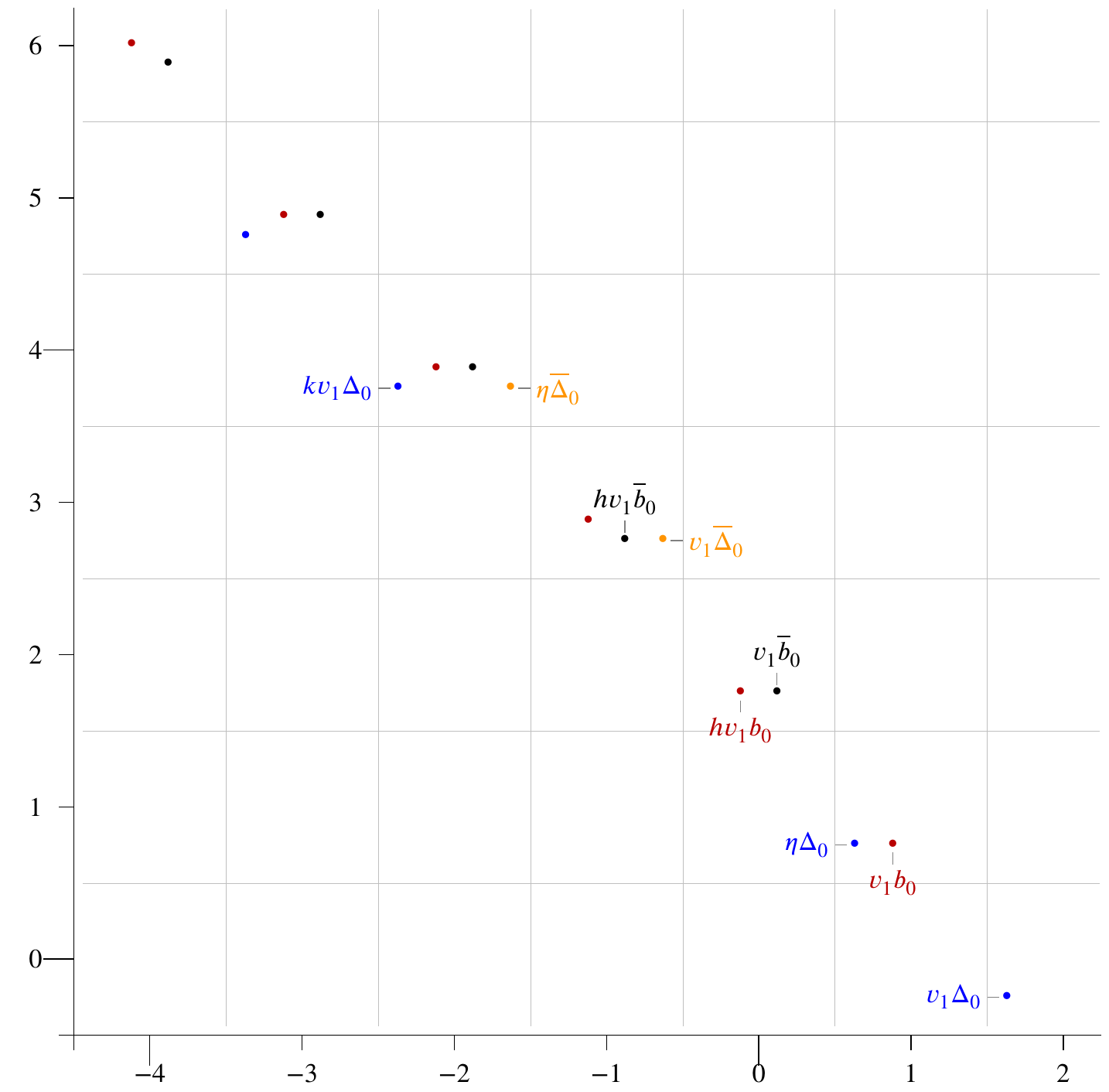}
 \caption[$E_1$ and $E_2$-pages of the ADSS for $\EE_2/2$]{Left: $E_1$-page of the \eqref{eq:ADSS} $E_1^{p,q}(\EE_t/2)$ for $t=2$. This is drawn in Adams grading so that the vertical axis is $p+q$ and the horizontal axis is $t-p-q$. The ADSS preserves $t$, and so in this picture, we have fixed $t=2$. Classes in {\color{blue} blue} and  {\color{YellowOrange} orange} come from \cref{fig:c6e2v0} and have filtration $p=0,3$ respectively. Classes in {\color{BrickRed} red} and black come from \cref{fig:c6e2} and have filtration $p=1,2$ respectively. All ADSS differentials raise $s=p+q$ by 1 and decrease $t-s$ by $1$. The $d_r$-differential raises $p$ by $r$. Right: The $E_2$-page. Compare with \cref{fig:ADSS-Mod2} below.}
\label{fig:ADSS-Mod2-E1}
\end{figure}

\begin{figure}[H]
 \includegraphics[page=1, width=1\textwidth]{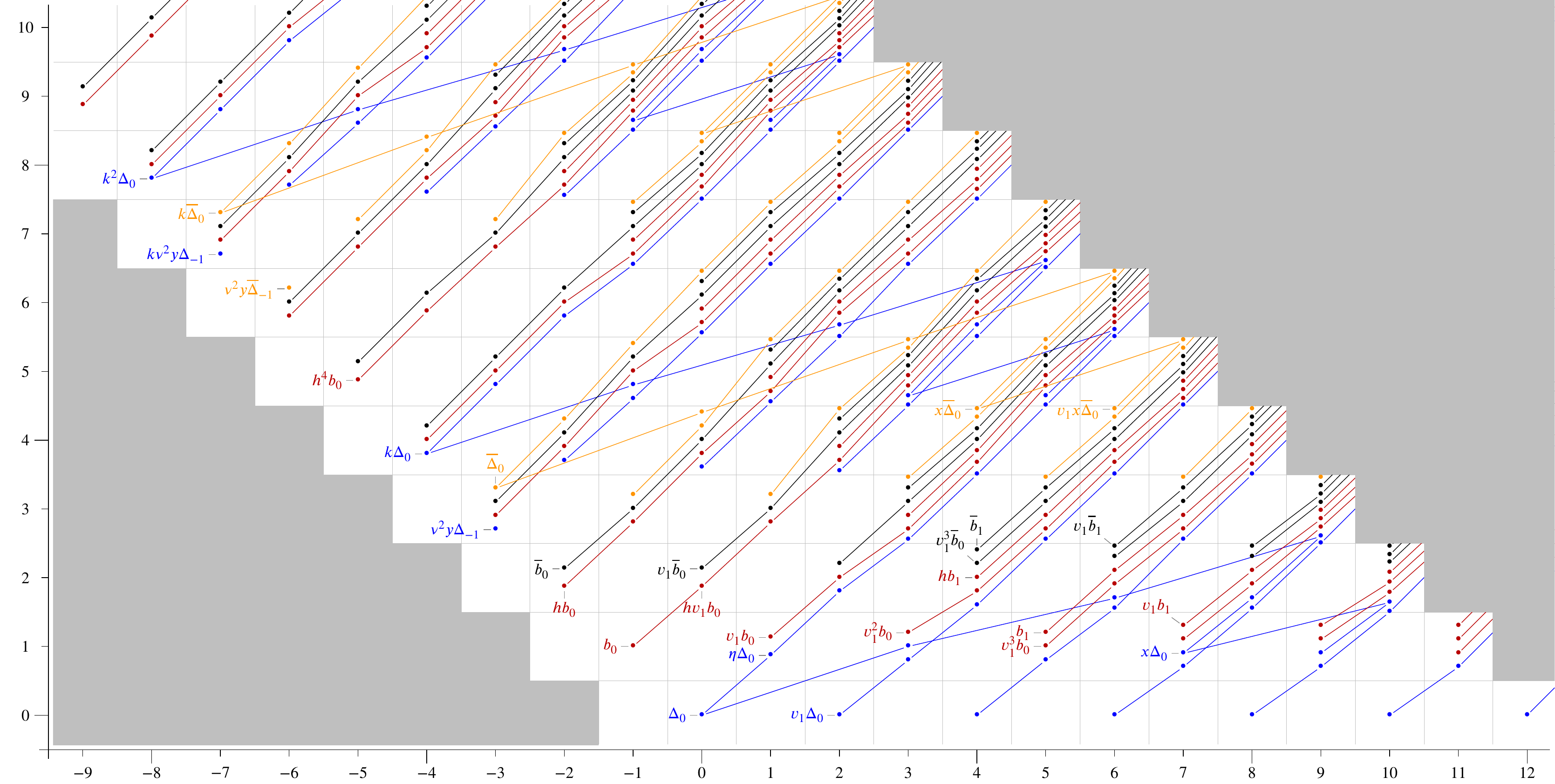}\caption[The $E_2$-page of the ADSS for $\EE_*/2$]{The $E_2$-page of the \eqref{eq:ADSS} for $M=\EE_*/2$. The figure is drawn in Adams grading corresponding to $H^{p+q}(\SS_2^1, \EE_t)$: The vertical axis is $p+q$ and the horizontal axis is $t-p-q$. The ADSS filtration is indicated by the color. {\color{blue}Blue} corresponds to $p=0$, {\color{BrickRed} red} represents $p=1$, black represents $p=2$ and {\color{YellowOrange} orange} $p=3$. A $\bullet=\F_4$. Lines of slope 1 are $\eta$ multiplication and those of slope $1/3$ are $\nu$ multiplication. }
\label{fig:ADSS-Mod2}
\end{figure}

\begin{table}[H]
  \begin{adjustbox}{addcode={\begin{minipage}{\width}}{\caption{%
      Generators for the cohomology groups $H^s(\SS_2^1, \EE_t/2)$ and $H^s(\GG_2^1, \EE_t/2)$. The colors correspond to the ADSS filtration.
 \label{fig:tableV0}     
      }\end{minipage}},rotate=90,center}
      \includegraphics[width=\textheight]{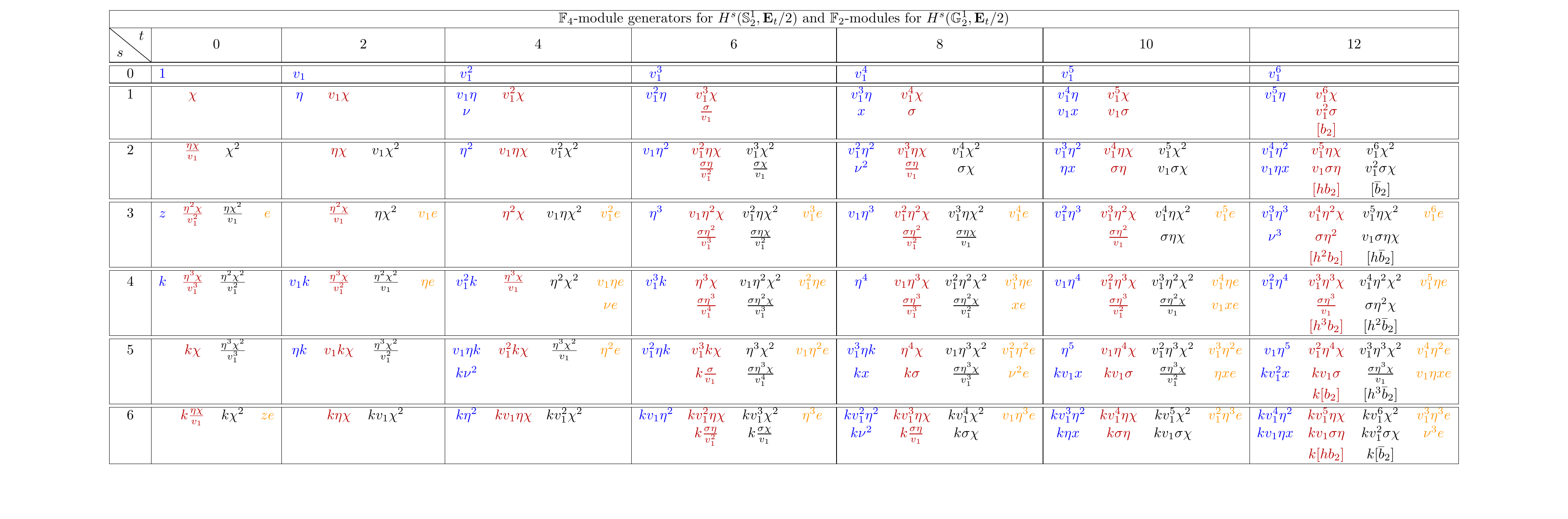}%
  \end{adjustbox}
\end{table}

\section{The Algebraic Duality Spectral Sequence for $\EE_*$}\label{sec:main}
 
In this section, we compute the \ref{eq:ADSS} with coefficients in $\EE_t$ \
\begin{equation*}
E_1^{p,q}(\EE_t)=H^q(F_p, \EE_t) \Longrightarrow H^{p+q}(\SS^1, \EE_t),
\end{equation*}
in the range $0\leq t <12$.
The first step is to compute the $E_2$-term, which is the cohomology $H^{\bullet}(E_1^{\bullet,*}(\EE_t))$ of the chain complex $( E_1^{\bullet,*}(\EE_t), d_1)$ with differentials 
\[d_1 \colon E_1^{\bullet,*}(\EE_t) \to E_1^{\bullet+1,*}(\EE_t) .\] 
Expanded out,  $H^{\bullet}(E_1^{\bullet,*}(\EE_t))$ denotes the cohomology of the chain complex
\[E_1^{\bullet, *}(\EE_t)= \{ 0 \to H^*(G_{24}, \EE_t) \xrightarrow{d_1} H^*(C_6, \EE_t) \xrightarrow{d_1} H^*(C_6, \EE_t) \xrightarrow{d_1} H^*(G_{24}, \EE_t) \to 0 \}. \]

In the range of focus, there will be no higher differentials and we will see that $E_2^{*,*}(\EE_t)\cong E_{\infty}^{*,*}(\EE_t)$.

\subsection{The strategy}
We will bootstrap from the computation of the \ref{eq:ADSS} for $\EE_*/2$ to that for $\EE_*$. For each $t$, we have an exact sequence of $\mathbb{S}_2^1$-modules,
\[  0 \to \EE_t \xrightarrow{2} \EE_t \to \EE_t/2 \to 0 \ ,\]
which gives long exact sequences
\begin{equation}\label{eq:lesE2} \ldots \to E_1^{p,q}(\EE_{t})  \xrightarrow{2}  E_1^{p,q}(\EE_{t}) \xrightarrow{r} E_1^{p,q}(\EE_{t}/2) \xrightarrow{\partial} E_1^{p,q+1}(\EE_{t})  \xrightarrow{2} \ldots 
\end{equation}
for each $0\leq p\leq 3$.
If we denote
\[F_1^{p,q}(\EE_{t}) :=  E_1^{p,q}(\EE_{t})/2,\]
let $\ker_2(M)$ denote the kernel of the multiplication by $2$ map, and $\im_2(M)$ denote its image;
then we have 
a commutative diagram as in \Cref{fig:bigdiagraminitial}, where the columns are extracted from \eqref{eq:lesE2} and the horizontal maps are induced by the $d_1$-differentials in the \ref{eq:ADSS}. 
 \begin{figure}
\[\xymatrix{
 \im_2(E_1^{0,q}(\EE_{t})) \ar[d]^-{\subseteq} \ar[r] &  \im_2(E_1^{1,q}(\EE_{t})) \ar[r]\ar[d]^-{\subseteq}& \im_2(E_1^{2,q}(\EE_{t})) \ar[r] \ar[d]^-{\subseteq}& \im_2(E_1^{3,q}(\EE_{t})) \ar[d]^-{\subseteq}  \\
E_1^{0,q}(\EE_{t}) \ar@{->>}[d] \ar[r] &E_1^{1,q}(\EE_{t}) \ar[r] \ar@{->>}[d] & E_1^{2,q}(\EE_{t})  \ar[r]  \ar@{->>}[d] & E_1^{3,q}(\EE_{t}) \ar@{->>}[d]  \\ 
 F_1^{0,q}(\EE_{t}) \ar[r] \ar[d]^-{\subseteq} &F_1^{1,q}(\EE_{t})\ar[d]^-{\subseteq}\ar[r] & F_1^{2,q}(\EE_{t})  \ar[r] \ar[d]^-{\subseteq}& F_1^{3,q}(\EE_{t}) \ar[d]^-{\subseteq} \\ 
 E_1^{0,q}(\EE_{t}/2) \ar[r]  \ar@{->>}[d] &E_1^{1,q}(\EE_{t}/2)\ar[d] \ar[r] & E_1^{2,q}(\EE_{t}/2) \ar[d]  \ar[r]& E_1^{3,q}(\EE_{t}/2)   \ar@{->>}[d]  \\ 
  \ker_2(E_1^{0,q+1}(\EE_{t})) \ar[r] &    \ker_2(E_1^{1,q+1}(\EE_{t})) \ar[r] &   \ker_2(E_1^{2,q+1}(\EE_{t})) \ar[r] &   \ker_2(E_1^{3,q+1}(\EE_{t}) 
}\]
\caption[A commutative diagram of $E_1$-terms for the ADSS]{A commutative diagram of $E_1$-terms for the ADSS where $q\geq 0$ and $t\in \Z$. The first three rows and the last three rows each form short exact sequences of chain complexes. Hence, we get two long exact sequences on the cohomology of these chain complexes.}
\label{fig:bigdiagraminitial}
\end{figure}

So, what do we know and what remains to be computed? Let us summarize the knowns, so far.
\begin{itemize}
\item Each group in this diagram is well-known as it is either a subgroup or subquotient of the cohomology groups
\[H^q(G_{24},\EE_t), \quad \text{or} \quad H^q(C_6, \EE_t);\]
\item In particular, when $q>0$, \Cref{lem:C6coh} gives that the higher $C_6$-cohomology of $\EE_t$ is all 2-torsion, implying that $\im_2(E_1^{1,q}(\EE_t)) = 0 = \im_2(E_1^{2,q}(\EE_t))$;
\item Every vertical map in this diagram is understood from results in \Cref{sec:Cohomology};
\item The cohomology of the fourth row, $H^{\bullet}(E_1^{\bullet,*}(\EE_t/2))$ for any $t$ is computed in \cite{BeaudryTowards}, and discussed in detail in \Cref{sec:BeaudryTowards} for $0\leq t <12$.
\end{itemize}
Furthermore, as we will see, in many cases, the groups $E_1^{p,q}(\EE_t)$ are zero based on the parity of $q+t/2$, and that these groups when non-trivial are often 2-torsion (i.e., $2E_1^{p,q}(\EE_t)=0$), which simplifies this diagram even further. For these reasons, the comparison of $E_1^{\bullet,q}(\EE_t) $ and $E_1^{\bullet,q}(\EE_t/2)$ is for the most part straightforward. The crucial missing piece for doing this analysis for all $t$ would be a computation of the differential $d_1 \colon E_1^{\bullet,0}  (\EE_t) \to E_1^{\bullet+1,0}  (\EE_t)$ (i.e., for $q=0$). When $q=0$, we have torsion free classes and so the comparison with $E_1^{\bullet,q}(\EE_t/2) $ is trickier.

\subsection{Some general cases}
We begin with some general remarks that simplify the diagram of \Cref{fig:bigdiagraminitial} in various cases before restricting the range of $t$. We note that $\EE_t=0$ if $t$ is odd, so henceforth, we always assume that $t$ is even. Then there are two over-arching cases, when $q+t/2$ is odd and when $q+t/2$ is even. This division is coming from the fact that the restriction of $\EE_t$ to $(\pm1)\subseteq G_{24}$ is the sign representation when $t/2$ is odd and the trivial representation when $t/2$ is even.

\subsubsection{The case $q+t/2$ odd}

The computation of the $E_2$-term of the ADSS for $\EE_t$ when $q+t/2$ is odd is straightforward.

\begin{lemma}\label{lem:q+t/2odd}
If  $q+t/2$ is odd, there are isomorphisms
\[E_1^{p,q}(\EE_t)  \cong  E_2^{p,q}(\EE_t).  \] 
\end{lemma}
\begin{proof}
When $q+t/2$ is odd, by \Cref{lem:C6coh} (depicted in \Cref{fig:c6E}) we have that
\[
E_1^{2,q}(\EE_t)=E_1^{1,q}(\EE_t)= H^q(C_6, \EE_t)=0,
\]
hence all the differentials $d_1: E_1^{p,q}(\EE_t)\to E_1^{p+1, q}(\EE_t)$ are zero in this case.
\end{proof}

\subsubsection{Generalities for $q+t/2$ even, $q>0$}

Our next goal is to analyze the situation when $q>0$, $q+t/2$ even. This assumption implies that the diagram in \Cref{fig:bigdiagraminitial} simplifies to the diagram of \Cref{fig:bigdiagram}. Indeed, by \Cref{lem:C6coh}, we know that $H^{q+1}(C_6,\EE_t)$ is zero, so the corresponding $\ker_2$ terms vanish as well.

\begin{figure}
\[\xymatrix{
 \im_2(E_1^{0,q}(\EE_{t})) \ar[d]^-{\subseteq} \ar[r] &  0 \ar[r]\ar[d] & 0 \ar[r] \ar[d]& \im_2(E_1^{3,q}(\EE_{t})) \ar[d]^-{\subseteq}  \\
E_1^{0,q}(\EE_{t}) \ar@{->>}[d] \ar[r] &E_1^{1,q}(\EE_{t}) \ar[r]\ar[d]^-{\cong} & E_1^{2,q}(\EE_{t})  \ar[r] \ar[d]^-{\cong}& E_1^{3,q}(\EE_{t}) \ar@{->>}[d]  \\ 
 F_1^{0,q}(\EE_{t}) \ar[r] \ar[d]^-{\subseteq} &F_1^{1,q}(\EE_{t}) \ar[d]^-{\cong}\ar[r] & F_1^{2,q}(\EE_{t})  \ar[r] \ar[d]^-{\cong}& F_1^{3,q}(\EE_{t}) \ar[d]^-{\subseteq} \\ 
 E_1^{0,q}(\EE_{t}/2) \ar[r]  \ar@{->>}[d] &E_1^{1,q}(\EE_{t}/2)\ar[d] \ar[r] & E_1^{2,q}(\EE_{t}/2) \ar[d]  \ar[r]& E_1^{3,q}(\EE_{t}/2)   \ar@{->>}[d]  \\ 
  \ker_2(E_1^{0,q+1}(\EE_{t})) \ar[r] &  0 \ar[r] & 0 \ar[r] &   \ker_2(E_1^{3,q+1}(\EE_{t}) 
}\]
\caption[Diagram of ADSS's in case $q>0 $ and $q+t/2$ even]{The ADSS in the case $q>0$ and $q+t/2$ even.}
\label{fig:bigdiagram}
\end{figure}

In degrees $q>0$, the cohomology of $G_{24}$ with coefficients in $\EE_*$ has two periodicities. The first is a periodicity in the topological degree $t$, which comes from multiplication by the modular form $\Delta \in H^0(G_{24},\EE_{24})$. The second is in the cohomological degree $q$, coming from multiplication by the element $k\in H^4(G_{24},\EE_0)$. Therefore, the cases we consider depend on $(q,t)$ modulo $(4,24)$.

In \Cref{figure:cases}, we highlight the three different cases (for $q>0$), and we discuss each of them separately in the next three subsections.

 \begin{figure}
 \includegraphics[page=1, width=\textwidth]{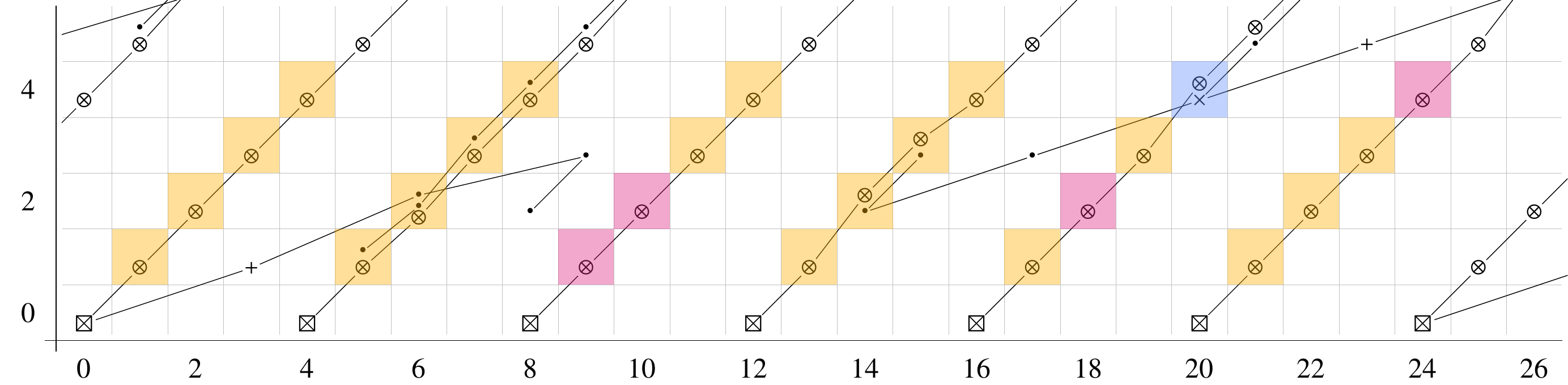}
 \caption[Cases for the analysis of the passage from $E_1$ to $E_2$]{Cases for the analysis of the passage from $E_1$ to $E_2$. In the yellow cases, the top and bottom row of \Cref{fig:bigdiagram} are zero. In the red cases, the top row is zero, and in the blue case, the bottom row is zero.}
 \label{figure:cases} 
 \end{figure}

\subsubsection{The cases $q>0$, $q+t/2$ even and $(q,t)\neq (0,0), (0,4), (1,10), (2,12)$ or $(2,20)$ modulo  $(4, 24)$}
These are the cases highlighted in yellow in  \Cref{figure:cases}.

In these cases, when $p=0,3$,
\begin{itemize}
\item $\im_2(E_1^{p,q}(\EE_t))=0$, and
\item $\ker_2(E_1^{p,q+1}(\EE_t)) \subseteq E_1^{p,q+1}(\EE_t) =0$.
\end{itemize}
So the diagram of \Cref{fig:bigdiagram} simplifies to give isomorphisms of \emph{chain complexes}
\[ E_1^{\bullet,q}(\EE_t) \cong  F_1^{\bullet,q}(\EE_t) \cong E_1^{\bullet,q}(\EE_t/2)   ,\]
hence the cohomology groups of these chain complexes are also isomorphic and we finished the proof of the following lemma. 
\begin{lemma}\label{lem:q+t/2evenGen} 
If  $q>0$, $q+t/2$ even and $(q,t)\neq (0,0), (0,4), (1,10), (2,12), (2,20)$ modulo  $(4, 24)$, then
\[E_2^{p,q}(\EE_t)\cong E_2^{p,q}(\EE_t/2) . \]
\end{lemma}

\subsubsection{The cases $q>0$, $q+t/2$ even and $(q,t) \equiv (0,4), (1,10), (2,12), (2,20)$ modulo $(4, 24)$}
These are the cases highlighted in red in \Cref{figure:cases}.
In these cases, $\im_2(E_1^{0,q}(\EE_t))=\im_2(E_1^{3,q}(\EE_t))=0$ so the first row of the diagram of \Cref{fig:bigdiagram} is zero. Therefore, the map between the second and third rows is an isomorphism, so composing from the second to fourth row, the diagram simplifies to:
\begin{equation}\label{eq:3rows}\xymatrix{
 E_1^{0,q}(\EE_{t}) \ar[r] \ar[d]^-{\subseteq}_-{r} &E_1^{1,q}(\EE_{t}) \ar[d]^-{\cong}_-{r} \ar[r] & E_1^{2,q}(\EE_{t})  \ar[r] \ar[d]^-{\cong}_-{r}& E_1^{3,q}(\EE_{t}) \ar[d]^-{\subseteq}_-{r} \\ 
 E_1^{0,q}(\EE_{t}/2) \ar[r]  \ar@{->>}[d] &E_1^{1,q}(\EE_{t}/2)\ar[d] \ar[r] & E_1^{2,q}(\EE_{t}/2) \ar[d]  \ar[r]& E_1^{3,q}(\EE_{t}/2)   \ar@{->>}[d]  \\ 
  \ker_2(E_1^{0,q+1}(\EE_{t})) \ar[r] &  0 \ar[r] & 0 \ar[r] &   \ker_2(E_1^{3,q+1}(\EE_{t}) .
}
\end{equation}

This diagram is a short exact sequence of chain complexes: each row is a chain complex and the columns are exact sequences by construction of \Cref{fig:bigdiagraminitial}. This short exact sequence induces a long exact sequence on cohomology of these chain complexes.
Because of the zeros in the second and third columns, this long exact sequence breaks up into the exact sequence
\begin{equation}\label{eq:lesfirstrowzero}
 0 \to E_2^{0,q}(\EE_t) \xrightarrow{r} E_2^{0,q}(\EE_t/2) \xra{\partial}   \ker_2(E_1^{0,q+1}(\EE_{t}))  \xra{\delta} E_2^{1,q}(\EE_{t}) \to E_2^{1,q}(\EE_{t}/2) \to 0 ,
 \end{equation}
the isomorphism
\[ E_2^{2,q}(\EE_{t}) \xrightarrow{\cong} E_2^{2,q}(\EE_{t}/2),  \]
and the short exact sequence
\[ 0 \to E_2^{3,q}(\EE_{t}) \xrightarrow{r} E_2^{3,q}(\EE_{t}/2) \xrightarrow{\partial} \ker_2(E_1^{3,q+1}(\EE_{t}))  \to 0. \]

To analyze these exact sequences, we first restrict attention to the case of $(q,t) $ congruent to $(1,10), (2,12), (2,20)$, as we have easier and more uniform behavior in this case.

\begin{lemma}\label{lem:q+t/2evenCases1}
If $q>0$, $q+t/2$ even and $(q,t) \equiv (1,10), (2,12), (2,20)$ modulo $(4, 24)$, then there are isomorphisms
\[ E_2^{p,q}(\EE_{t}) \xrightarrow{\cong} E_2^{p,q}(\EE_{t}/2),  \quad p=1,2\]
and short exact sequences
\[ 0 \to E_2^{p,q}(\EE_{t}) \xrightarrow{r}  E_2^{p,q}(\EE_{t}/2) \xrightarrow{\partial}   \ker_2(E_1^{p,q+1}(\EE_{t}))  \to 0, \quad p=0,3 . \] 
\end{lemma}
\begin{proof}
We will show the map $\partial \colon E_2^{0,q}(\EE_{t}/2) \to   \ker_2(E_1^{0,q+1}(\EE_{t}))$ in \eqref{eq:lesfirstrowzero} is surjective. Then the claims all follow from the previous discussion.
The map $\partial$ is induced by the map on $E_1$-pages
 \begin{equation}\label{eq:deltaE1}
 \partial \colon E_1^{0,q}(\EE_{t}/2) \to \ker_2(E_1^{0,q+1}(\EE_{t}))
 \end{equation}
  coming from the long exact sequence on $E_1^{p,*}(-)$ induced by the short exact sequence $\EE_t \xrightarrow{2} \EE_t \to \EE_t/2$, as
defined in \eqref{eq:lesE2}.
In all the cases we are considering here, we have
\[ \ker_2(E_1^{0,q+1}(\EE_{t}))  \cong E_1^{0,q+1}(\EE_{t}) \cong H^{q+1}(G_{24}, \EE_t),\]
and we know these groups explicitly. Namely, there are suitable integers $n,m$ such that
\begin{align*}
&\text{if } (q,t)\equiv(1,10), & &
\F_4\{\epsilon k^n\Delta^m\} \cong E_1^{0,2+4n}(\EE_{10+24m})\cong H^{2+4n}(G_{24}, \EE_{10+24m}); \\
&\text{if } (q,t)\equiv(2,12), & &
\F_4\{\nu^3k^n\Delta^m\}  \cong E_1^{0,3+4n}(\EE_{12+24m})\cong H^{3+4n}(G_{24}, \EE_{12+24m}); \\
&\text{if } (q,t)\equiv(2,20), & &
\F_4\{\nu\kappa k^n\Delta^m\} \cong  E_1^{0,3+4n}(\EE_{20+12m})\cong H^{3+4n}(G_{24}, \EE_{20+24m}).
\end{align*}
Before passing to cohomology, the map $\partial$ of \eqref{eq:deltaE1} is surjective since from \Cref{rem:xyboundaries}, we have
\[\partial(v_1x k^n\Delta^m)=\epsilon k^n\Delta^m, \quad \partial(\eta v_1x k^n\Delta^m)=\nu^3 k^n\Delta^m, \quad \partial(\nu y k^n\Delta^m)=\nu \kappa k^n\Delta^m. \]
Now from \cite[Theorem~1.2.2]{BeaudryTowards}, we see that $v_1x k^n\Delta^m$, $\eta v_1x k^n\Delta^m$  and $\nu y k^n\Delta^m$ survive to $E_2^{0,q}(\EE_{t}/2)$ and so $\partial$ in \eqref{eq:lesfirstrowzero} is indeed surjective at the $E_2$-page, which proves the claims.
\end{proof}

In the next result, we discuss the case $(q,t)=(0,4) \mod (4,24)$.

\begin{lemma} \label{lem:q+t/2:04}
Suppose $q>0$, $q+t/2$ even and $(q,t) \equiv (0,4)$ modulo  $(4, 24)$. Then there are isomorphisms
\[ E_2^{2,q}(\EE_{t}) \xrightarrow{\cong} E_2^{2,q}(\EE_{t}/2)   \] 
and short exact sequences
\[ 0 \to E_2^{3,q}(\EE_{t}) \to  E_2^{3,q}(\EE_{t}/2) \to   \ker_2(E_1^{3,q+1}(\EE_{t})) \to 0. \]
When 
$t=4$, 
there are  short exact sequences
\[ 0 \to E_2^{0,q}(\EE_4) \to  E_2^{0,q}(\EE_4/2) \to   \ker_2(E_1^{0,q+1}(\EE_4)) \to 0 \]
and isomorphisms
\[E_2^{1,q}(\EE_4) \xrightarrow{\cong}  E_2^{1,q}(\EE_4/2). \]
When $t \neq 4$,
there are isomorphisms
\[ E_2^{0,q}(\EE_{t}) \xrightarrow{\cong}  E_2^{0,q}(\EE_{t}/2)\] 
and short exact sequences
\[ 0 \to  \ker_2(E_1^{0,q+1}(\EE_{t})) \to E_2^{1,q}(\EE_{t}) \to  E_2^{1,q}(\EE_{t}/2) \to   0. \]
\end{lemma}
\begin{proof}
Let $m,n$ be integers such that  $q=4n >0$ and $t=4+24m$. 
We have from \Cref{G24} and \Cref{G24-mod2}
\begin{align*}
E_1^{0,q}(\EE_{t}/2) &\cong H^{4n}(G_{24}, \EE_{4+24m}/2)\cong  \FF_4[\![j]\!]\{ v_1^2k^n \Delta^{m}  \}, \\
E_1^{0,q}(\EE_{t}) &\cong H^{4n}(G_{24}, \EE_{4+24m})\cong \F_4[\![j]\!]\{c_4^2c_6 k^n \Delta^{m-1}\}.
\end{align*}
The diagram \eqref{eq:3rows} then becomes
\begin{equation*}
\resizebox{1\hsize}{!}{\xymatrix{
\F_4[\![j]\!]\{c_4^2c_6 k^n \Delta^{m-1} \}  \ar[r]^-{d_1} \ar[d]^-{\subseteq}_-{r} &E_1^{1,4n}(\EE_{4+24m}) \ar[d]^-{\cong}\ar[r]^-{d_1} & E_1^{2,4n}(\EE_{4+24m})  \ar[r]^-{d_1} \ar[d]^-{\cong}& \F_4[\![j]\!]\{c_4^2c_6 k^n \Delta^{m-1} \}  \ar[d]^-{\subseteq}_-{r} \\ 
 \F_4[\![j]\!]\{v_1^2k^n\Delta^m \} \ar[r]^-{d_1}  \ar@{->>}[d]_{\partial} &E_1^{1,4n}(\EE_{4+24m}/2)\ar[d] \ar[r]^-{d_1} & E_1^{2,4n}(\EE_{4+24m}/2) \ar[d]  \ar[r]^-{d_1}& \F_4[\![j]\!]\{v_1^2k^n\Delta^m \}   \ar@{->>}[d]^{\partial}  \\ 
\F_4\{2\nu k^n  \Delta^m\}  \ar[r] &  0 \ar[r] & 0 \ar[r] &  \F_4\{2\nu k^n  \Delta^m\},}
}
\end{equation*}
where we have $r(c_4^2c_6\Delta^{-1})=jv_1^2$ and $\partial (v_1^2k^n \Delta^m)=2\nu k^n \Delta^m$.
Therefore, if $d_1(v_1^2k^n\Delta^m)$ is zero, then the map $\partial:E_2^{0,q} \to \ker_2(E_1^{0,q+1}(\EE_t))$ in \eqref{eq:lesfirstrowzero} is surjective, while if $d_1(v_1^2k^n\Delta^m) \neq 0$, then $\partial $ is zero and $\delta: \ker_2(E_1^{0,q+1}(\EE_t)) \to E_2^{1,q}(\EE_t)$ is injective. From \cite[Theorem 1.2.1]{BeaudryTowards}, the first case occurs when $m=0$, and the latter when $m\neq 0$. Corresponding to these cases, \eqref{eq:lesfirstrowzero} breaks up into the claimed exact sequences and isomorphisms.
\end{proof}

\subsubsection{The case $q>0$, $q+t/2$ even and $(q,t) \equiv (0,0)$ modulo  $(4, 24)$}\label{sec:q+t/2zero}
This case is highlighted in blue in \Cref{figure:cases}; it is more subtle so we do not do a complete explicit analysis, but we collect some general remarks.
For one, in this case, the last row of the diagram \Cref{fig:bigdiagram} is zero, so
we get short exact sequences
\[ 0 \to   \im_2(E_1^{0,q}(\EE_{t})) \to E_2^{0,q}(\EE_{t}) \to  E_2^{0,q}(\EE_{t}/2) \to 0,  \]
isomorphisms
\[E_2^{1,q}(\EE_{t}) \xrightarrow{\cong}  E_2^{1,q}(\EE_{t}/2)\]
and a long exact sequence
\[0 \to E_2^{2,q}(\EE_{t}) \to E_2^{2,q}(\EE_{t}/2) \xrightarrow{\delta} \im_2(E_1^{3,q}(\EE_{t})) \to E_2^{3,q}(\EE_{t}) \to E_2^{3,q}(\EE_{t}/2) \to 0 \ . \]

\subsubsection{The case $q=0$ and $t/2$ even.}\label{sec:qzero}
This is the final and most subtle case. To analyze this case, one considers the diagram (cf. \Cref{fig:bigdiagraminitial})
\begin{align}\label{eq:q=0complexes}
\xymatrix{
E_1^{0,0}(\EE_{t}) \ar[d]^-{\times 2}_{\subseteq} \ar[r] &E_1^{1,q}(\EE_{t}) \ar[r]  \ar[d]^-{\times 2}_{\subseteq}  & E_1^{2,0}(\EE_{t})  \ar[r]  \ar[d]^-{\times 2}_{\subseteq}  & E_1^{3,0}(\EE_{t})  \ar[d]^-{\times 2}_{\subseteq}  \\
E_1^{0,0}(\EE_{t})  \ar[r] \ar@{->>}[d] &E_1^{1,0}(\EE_{t}) \ar[r]  \ar@{->>}[d] & E_1^{2,0}(\EE_{t})  \ar[r] \ar@{->>}[d] & E_1^{3,0}(\EE_{t}) \ar@{->>}[d]  \\
 E_1^{0,0}(\EE_{t}/2) \ar[r]   &E_1^{1,0}(\EE_{t}/2)  \ar[r] & E_1^{2,0}(\EE_{t}/2)   \ar[r]& E_1^{3,0}(\EE_{t}/2)     
}\end{align}
whose columns are short exact. However, this is a rather subtle analysis which requires information about the $d_1$-differentials, which we provide in \Cref{thm:techyd1} below, though only in cases relevant to our range.

\subsection{Technical results about the $d_1$-differential}
In this section, we prove some technical facts about the $d_1$-differential. Understanding this section will require diving deeper into \cite{BeaudryTowards,BGH} than some readers might want to, but these facts can also be taken for granted without impairing readability.

\begin{notation}[{\it{c.f.} \cref{rem:defdbs}}]
\label{notn:ClassDefs}
We let $\Delta_0$, $b_0$, $\overline{b}_0$, $\overline{\Delta}_0$ be the units in $E_1^{p,0}(\EE_0)$ for $p=0,1,2,3$, respectively. From the discussion around Definition 5.2.3 of \cite{BeaudryTowards}, $d_1(c_4\Delta_0)$ is zero mod $16$ and there is an element which we name
\[v_1b_1 := d_1(c_4\Delta_0)/16 \in E_1^{1,0}(\EE_8).\]
Note that $v_1b_1$ is not a product despite the name. This element is called $b_{1,0,0}$ in \cite[Remark 4.7]{BeaudryAlpha} and there it is shown to detect $\alpha_{4/4}$.

We let $v_1\overline{b}_1 \in E_1^{2,0}(\EE_8)$ be any class which reduces to the same named class in $E_1^{2,0}(\EE_8/2)$ and has the property that $v_1\overline{b}_1\equiv v_1v_2 \mod (4,2u_1^3, u_1^6)$. Note that by Lemma 5.3.2 of \cite{BeaudryTowards}, the class $\overline{b}_1 \in E_1^{2,0}(\EE_8/2)$ is congruent to $v_2$ modulo $u_1^6$ so such a choice is possible.
\end{notation}

\begin{theorem}\label{thm:techyd1}
For $d_1\colon E_1^{p,0}(\EE_*) \to E_1^{p+1,0}(\EE_*)$ differential of the ADSS, we have:
\begin{align}
d_1(b_0) &= 2\overline{b}_0  \mod 8 & d_1 &\colon E_{1}^{1,0} (\EE_0) \to E_{1}^{2,0} (\EE_0) \label{d1} \\
d_1(v_1^2b_0) &= 2v_1^2\overline{b}_0  \mod 4 & d_1 &\colon E_{1}^{1,0} (\EE_4) \to E_{1}^{2,0} (\EE_4)  \label{d2}  \\
d_1(c_4\Delta_0) &\equiv 16v_1b_1 \mod 32  &  d_1 &\colon E_{1}^{0,0} (\EE_8) \to E_{1}^{1,0} (\EE_8)   \label{d3} \\ 
d_1(v_1^4b_0) &\equiv 2v_1^4\overline{b}_0 \mod 8  &  d_1 &\colon E_{1}^{1,0} (\EE_8) \to E_{1}^{2,0} (\EE_8)  \label{d4}   \\ 
d_1(v_1\overline{b}_1) &\equiv 2c_4 \overline{\Delta}_0 \mod (4,j) &  d_1 &\colon E_{1}^{2,0} (\EE_8) \to E_{1}^{3,0} (\EE_8)  \label{d5} 
\end{align}
\end{theorem}
\begin{proof}
The differential \eqref{d1} follows from \cite[Theorem 2.5.1]{BGH}, see also Figure 4 of \cite{BGH}. 
The element $v_1$  is fixed modulo $2$ (by all of $\GG_2$), and so $v_1^2 \in \EE_4$ is fixed modulo $4$. So the $d_1$-differentials are $v_1^2$-linear modulo $4$. In particular, for $d_1 \colon E_{1}^{1,0} (\EE_4) \to E_{1}^{2,0} (\EE_4)$, we have
\[d_1(v_1^2 b_0) \equiv v_1^2d_1(b_0) \mod 4.\]
This implies the differential \eqref{d2}. The differential \eqref{d3} follows from our definition of $b_1$, but the essential content for it is proved in \cite[Proof of Prop. 8.2.9]{BGH}. 
The differential \eqref{d4} uses the $v_1^4$-linearity of $d_1$ modulo $8$.

The differentials \eqref{d5} requires more work.
For the rest of the proof, we use the notation of \Cref{rem:pialpha}, and let $i,j,k$ be the usual elements of $Q_8$.

To prove \eqref{d5}, we note that from \cite[Thm. 1.1.1]{BeaudryTowards} (and the fact that the elements $\pi$ and $\alpha$ commute), 
\[d_1 \colon E_1^{2,0}(\EE_*) \to E_1^{3,0}(\EE_*) \cong H^*(G_{24}', \EE_*)\] 
is given by the action of the element 
\begin{equation}\label{eq:lastmap}\pi(1+i+j+k)\pi^{-1}(1-\alpha^{-1}) \in \Z_2[\![\mathbb{S}_2^1]\!].\end{equation}
In particular, in this proof, we remember that we have been implicitly using the isomorphism
\[H^*(G_{24}', \EE_*) \xrightarrow{\cong} H^*(G_{24}, \EE_*)\] 
induced by conjugation by $\pi^{-1}$ in our presentation. We do not apply this isomorphism here when computing the map \eqref{eq:lastmap}.

We will use the formulas of \cite[Sec. 6.2]{BeaudryTowards}. First, recall that any element of $\gamma \in \SS_2$ can be written as
\[\sum_{i\geq 0} a_iT^i\]
for elements $a_i\in \WW$ such that $a_i^4-a_i=0$, and $a_0\neq 0$, and that 
\[F_{2/2}\SS_2 = \{x\in \SS_2 \mid x\equiv 1 \mod T^2 \}.\]
We first compute the action of $\pi^{-1}(1-\alpha^{-1})$ on $v_1v_2$. Both $\pi^{-1}$ and $\alpha^{-1}$ are in $F_{2/2}\SS_2$.

From Theorem 6.2.2 of \cite{BeaudryTowards}, we have that
\begin{align*}
\gamma_* u_1 &= t_0(\gamma)u_1+2t_0(\gamma)^{-1}t_1(\gamma) \mod 4 \\
\gamma_*u &= t_0(\gamma)u .
\end{align*}
(Note that in Theorem 6.2.2, the second formula is stated modulo $2$, but this formula holds integrally, see the discussion around (2.3.1) of that reference.)
From Propositions 6.3.10 and 6.3.9 of \cite{BeaudryTowards}, for $\gamma$ such that $a_0=1$ and $a_1=a_3=0$, we have
\begin{align*}
t_1=t_1(\gamma)&= a_2^2u_1 \mod (2,u_1^6) \\
t_0=t_0(\gamma)&= 1+2a_2+(a_2+a_2^2)u_1^3 \mod (4,2u_1^2,u_1^6).
\end{align*}
Using that $v_1=u_1u^{-1}$ and $v_2=u^{-3}$, we compute:
\begin{align*}
\gamma_*(v_1v_2) &= (t_0u_1+2 t_0^{-1}t_1) t_0^{-4}u^{-4}  \mod 4 \\
&= t_0^{-3}u_1u^{-4}+2t_0^{-5}t_1u^{-4} \mod 4  \\
&=(1+2a_2+(a_2+a_2^2) u_1^3)u_1u^{-4} + 2a_2^2u_1u^{-4} \mod (4,2u_1^3, u_1^6)\\
&=v_1v_2 + 2(a_2+a_2^2)v_1v_2+ (a_2+a_2^2)v_1^4  \mod (4,2u_1^3, u_1^6).
\end{align*}
The elements $\alpha^{-1}$ and $\pi^{-1}$ satisfy the conditions placed on $\gamma$ above with $a_2=\omega$. Therefore, we have
\[\pi^{-1}(1-\alpha^{-1})_*(v_1v_2) \equiv 2v_1v_2+v_1^4 \mod (4,2u_1^3, u_1^6) , \]
where we've used the fact $v_1^4$ is fixed modulo $8$, and hence modulo $4$.
It was computed that $\pi(1+i+j+k)_*(v_1v_2) \equiv v_1^4 \mod 2$ in the proof of \cite[Prop. 8.2.10]{BGH}. We also have
\[\pi(1+i+j+k)_*v_1^4 \equiv 0 \mod 4.\]
Therefore, 
\begin{align*}
\pi(1+i+j+k)\pi^{-1}(1-\alpha^{-1})_*(v_1v_2)
&\equiv 2v_1^4 \mod (4, 2u_1^3,u_1^6) \ .
\end{align*}
Since $v_1\overline{b}_1  \equiv v_1v_2 \mod (4, 2u_1^3,u_1^6) $, the same formula holds for $v_1\overline{b}_1$.
Finally, since $E_1^{3,0}(\EE_8)\cong \W[\![j]\!] \{c_4\overline{\Delta}_0\}$ with $j\equiv v_1^{12} \mod 2$ and $c_4 \equiv v_1^4 \mod 2$, the congruence implies $d_1(v_1\overline{b}_1) \equiv 2c_4 \overline{\Delta}_0$ modulo $(4,j)$.
\end{proof}

\subsection{Specific results in the range $0\leq t<12$}\label{sec:Einf}
We next state the results we need in our range. These results only require additional arguments when $q=0$ and $t/2$ is even, as the analysis above and the computation of \cite[Thm. 1.2.2]{BeaudryTowards} give the stated claims otherwise.

\begin{lem}[$t=0$]\label{lem:t=0}
There are isomorphisms
\[E_2^{p,*}(\EE_0) \cong  \begin{cases}
\W[k]/(8k)\{\Delta_0 \}&  \text{$p=0$}  \\
\F_4[g] \{gb_0 \} & \text{$p=1$}  \\
\F_4[g] \{\overline{b}_0 \} & \text{$p=2$}  \\
\W[k]/(8k)\{\overline{\Delta}_0\} &  \text{$p=3$}  
\end{cases} \]
where $k \in E_2^{p,4}(\EE_0)$ for $p=0,3$ and $g\in E_2^{p,2}(\EE_0)$ for $p=1,2$.
\end{lem}

\begin{rem}
Here and below, the class $g \in H^2(C_6,\EE_0) $ is the one introduced in \eqref{eq:C6cohElements}; by \Cref{lem:maprforC6}, it reduces to $h^{2}$ mod $2$.
\end{rem}
\begin{proof}
When $q$ is odd, \Cref{lem:q+t/2odd} gives the claim -- everything is zero in those bidegrees. When $q\equiv 2 \mod 4$, then \Cref{lem:q+t/2evenGen} and \Cref{thm:Einfmod2} apply to give that in these bidegrees we have 
$\F_4 \{g^{(q-2)/2}gb_0\}$ and $\F_4 \{g^{q/2}\overline{b}_0\}$ for $p=1,2$ 
respectively, and zeros when $p=0$ or $3$. When $q \equiv 0 \mod 4$ and positive, then \Cref{sec:q+t/2zero} applies (but there are some non-trivial extensions to be solved), and for $q=0$ we have \Cref{sec:qzero}. We refer the reader to \cite[Lemma 6.1.5]{BGH} for these remaining details.
\end{proof}

\begin{lem}[$t=2$]\label{lem:t=2}
There are isomorphisms
\[E_2^{p,*}(\EE_2) \cong  \begin{cases}
\F_4[k]\{\eta \Delta_0\} &  \text{$p=0$}  \\
\F_4[g]\{  \eta b_0 \} & \text{$p=1$}  \\
\F_4[g]\{  \eta  \overline{b}_0 \} & \text{$p=2$}  \\
\F_4[k]\{\eta\overline{\Delta}_0\} &  \text{$p=3$} 
\end{cases} \] 
\end{lem}
\begin{proof}
In this case, when $q$ is even, then \Cref{lem:q+t/2odd} ~gives that $E_2^{p,q}(\EE_2) \cong E_1^{p,q}(\EE_2)$, and we find these groups are zero from \Cref{G24} and \Cref{lem:C6coh}. When $q$ is odd, we are in the situation of \Cref{lem:q+t/2evenGen}, so we read off the answer from \Cref{thm:Einfmod2}.
\end{proof}

\begin{lem}[$t=4$]\label{lem:caset4}
There are isomorphisms
\[E_2^{p,*}(\EE_4) \cong \begin{cases} 
\W/4[k]\{\nu {\Delta}_0\} \oplus  \F_4[k] \{\eta^2 {\Delta}_0\} & p=0 \\
\F_4[g]\{gv_1^2b_0\} &  p=1 \\
\F_4[g]\{ v_1^2\overline{b}_0\} & p=2 \\
\W/4[k]\{\nu \overline{\Delta}_0\} \oplus  \F_4[k] \{\eta^2 \overline{\Delta}_0\}  & p=3
\end{cases}\]
\end{lem}
\begin{proof}

Again, we use \Cref{lem:q+t/2odd} for the case $q$ odd, \Cref{lem:q+t/2evenGen} when $q \equiv 2 \mod 4$, and \Cref{lem:q+t/2:04} when $q>0$ and divisible by 4.
The only case that remains is when $q=0$. We consider the diagram (cf. \Cref{fig:bigdiagraminitial})
\[\xymatrix@C=1pc{ 
0 \ar[r] & E_{1}^{0,0} (\EE_4)\ar[r] \ar[d]^-{\times 2} & E_{1}^{1,0} (\EE_4) \ar[r]  \ar[d]^-{\times 2} & E_{1}^{2,0} (\EE_4) \ar[r]  \ar[d]^-{\times 2} & E_{1}^{3,0} (\EE_4) \ar[r] \ar[d]^-{\times 2} & 0
\\
0 \ar[r] & E_{1}^{0,0} (\EE_4)\ar[r]  \ar[d] & E_{1}^{1,0} (\EE_4) \ar[r]  \ar[d] & E_{1}^{2,0} (\EE_4) \ar[r]  \ar[d]& E_{1}^{3,0} (\EE_4) \ar[r]  \ar[d] & 0 
\\
0 \ar[r] & \F_4[\![j]\!]\{c_4^2c_6/\Delta\} \ar[r]  \ar[d]^-{\subseteq} & \F_4[\![j_0]\!] \{[v_1^2]  \} \ar[r] \ar[d]^-{=} &  \F_4[\![j_0]\!] \{[v_1^2]  \} \ar[r] \ar[d]^-{=}  & \F_4[\![j]\!]\{c_4^2c_6/\Delta\} \ar[r] \ar[d]^-{\subseteq} & 0 \\
0 \ar[r] & E_{1}^{0,0} (\EE_4/2)   \ar[d]\ar[r]  & E_{1}^{1,0} (\EE_4/2) \ar[d] \ar[r]  & E_{1}^{2,0} (\EE_4/2) \ar[d] \ar[r]  & E_{1}^{3,0} (\EE_4/2) \ar[r]  \ar[d] & 0  \\
0 \ar[r] & \F_4\{ v_1^2\Delta_0\} \ar[r]  &  0 \ar[r]  &  0 \ar[r]  &  \F_4\{ v_1^2\overline{\Delta}_0\} \ar[r]   & 0 
}\]
where the middle row is  $F_1^{\bullet,0}(\EE_4) := E_1^{\bullet,0}(\EE_4)/2 $.
Noting that there is no $2$-torsion on the $E_1$-page in degrees $q=0$, the first three and last three rows again form an exact sequence of chain complexes.

From  \Cref{thm:Einfmod2}, we have
\[E_2^{p,0}(\EE_4/2)  \cong  \begin{cases} 
\F_4\{ v_1^2\Delta_0\} & p=0 \\
\F_4\{ v_1^2b_0\} & p=1 \\
\F_4\{ v_1^2\overline{b}_0\} & p=2 \\
\F_4\{ v_1^2\overline{\Delta}_0\} & p=3\ ,
\end{cases}\]
which implies (using the exact sequence of chain complexes formed by the last three rows) that
\[ F_2^{p,0}(\EE_4) \cong \begin{cases} 
0 & p=0,3 \\
\F_4\{ v_1^2b_0\} & p=1 \\
\F_4\{ v_1^2\overline{b}_0\} & p=2  \ .
\end{cases} \]
From this and the exact sequence of chain complexes formed by the first three rows, we have that
\[  E_{2}^{p,0} (\EE_4) \xrightarrow{\times 2}E_{2}^{p,0} (\EE_4)  \to F_2^{p,0}(\EE_4) =0 \quad p=0,3.\]
So the multiplication by $2$ map on $E_{2}^{p,0} (\EE_4)$ is surjective for $p=0,3$. But $E_{2}^{p,0} (\EE_4)$ is 2-complete, so this implies that $E_{2}^{p,0} (\EE_4) =0$ when $p=0,3$ as claimed.  

The remaining terms assemble in a long exact sequence
\begin{align*}
&0 \to  E_{2}^{1,0} (\EE_4) \xrightarrow{ 2}E_{2}^{1,0} (\EE_4)  \to 
 F_2^{1,0}(\EE_4)   
  \xrightarrow{\delta}   E_{2}^{2,0} (\EE_4) \xrightarrow{ 2}E_{2}^{2,0} (\EE_4)  \to    
 F_2^{2,0}(\EE_4)  
 \to  0.
\end{align*}
As stated in \Cref{thm:techyd1},  $d_1(v_1^2b_0) =2v_1^2\overline{b}_0 \mod 4$ in the ADSS for $\EE_4$; here, the class $v_1^2 b_0 \in E_1^{1,0}(\EE_4)$ maps to the homonymous class in $F_1^{1,0}(\EE_4)$ which is a $d_1$ cycle. Chasing the diagram, this implies that $\delta(v_1^2b_0) \equiv v_1^2\overline{b}_0 \mod 2$  
under the connecting homomorphism 
$\delta \colon F_2^{1,0}(\EE_4)  \to E_2^{2,0}(\EE_4) $. Since $v_1^2\overline{b}_0\neq 0$ in
$E_{2}^{2,0} (\EE_4)/2 \cong F_{2}^{2,0}(\EE_4)$,
$\delta$ is injective (as it is non-zero and $\delta$ is $\W$-linear).
In particular, $E_{2}^{1,0} (\EE_4)  \xrightarrow{2}  E_2^{1,0}(\EE_4)$ is  an isomorphism, making the 2-complete group $E_{2}^{1,0} (\EE_4)$ trivial.

We are left with the exact sequence 
\[0 \to \F_4\{v_1^2b_0\}  \xrightarrow{\delta}  E_{2}^{2,0} (\EE_4) \xrightarrow{ 2}E_{2}^{2,0} (\EE_4)  \to    \F_4\{v_1^2\overline{b}_0\} \to 0, \] 
in which we know that $\delta(v_1^2b_0) \equiv v_1^2\overline{b}_0 \mod 2$, so we get an isomorphism
\[E_{2}^{2,0} (\EE_4)/\im(\delta) \xrightarrow[\cong]{2}E_{2}^{2,0} (\EE_4)/\im(\delta)\]
As before, this implies that $E_{2}^{2,0} (\EE_4)/\im(\delta) =0$, thus $\delta$ induces an isomorphism 
$E_{2}^{2,0} (\EE_4)\cong \F_4\{v_1^2\bar{b}_0\}$.
\end{proof}

The proof of the next result is completely analogous to that of \Cref{lem:t=2}.

\begin{lem}[$t=6$]
There are isomorphisms 
\[ E_2^{p,q}(\EE_6) \cong  \begin{cases}
\F_4[k]\{\mu\Delta_0, \eta^3 \Delta_0\} &  \text{$p=0$}  \\
\F_4[g]\{  v_1^2 \eta b_0, h b_1  \} & \text{$p=1$}  \\
\F_4[g]\{ v_1^2 \eta \overline{b}_0, h \overline{b}_1 \} & \text{$p=2$}  \\
\F_4[k]\{\mu\overline{\Delta}_0, \eta^3 \overline{\Delta}_0\} &  \text{$p=3$}  \\
\end{cases} \]
where $hb_1$ and $h\overline{b}_1$ are the unique classes in $E_2^{1,1}(\EE_6)$ and $E_2^{2,1}(\EE_6)$, respectively, which reduce to the same named classes in $E_2^{1,1}(\EE_{6}/2)$ and $E_2^{2,1}(\EE_{6}/2)$.
\end{lem}

\begin{rem}
The classes $hb_1$ and $h\overline{b}_1$ are detected modulo $j_0$ by $w_5$ in $E_1^{1,1}(\EE_6)$, respectively in $E_1^{2,1}(\EE_{6})$.
\end{rem}

For the next statement, recall that $v_1b_1$ was already defined in \cref{notn:ClassDefs}.

\begin{lem}[$t=8$]\label{lem:caset8}
There are isomorphisms
\[E_2^{p,*}(\EE_8) \cong 
\begin{cases} 
\F_4[k]\{\nu^2{\Delta}_0, \eta \mu{\Delta}_0,  kc_4 {\Delta}_0\} & p=0 \\
\W/16 \{v_1b_1\} \oplus \F_4[g]\{\eta^2v_1^2 b_0, \eta hb_1\} & p=1\\
\F_4[g]\{v_1^4\overline{b}_0, \eta h\overline{b}_1\} & p=2\\
\F_4[k]\{c_4\overline{\Delta}_0, \nu^2\overline{\Delta}_0, \eta \mu\overline{\Delta}_0  \}  & p=3 
\end{cases}\]
where $ k c_4 \overline{\Delta}_0 = \eta^4\overline{\Delta}_0$ and also  $g v_1b_1 = \eta hb_1$.
 Furthermore, the class $\sigma$ is detected by $v_1b_1$ modulo $2$.
\end{lem}
\begin{proof}
As before, \Cref{lem:q+t/2odd} gives us zeros for all $p$ when $q $ is odd. When $q$ is even and positive, \Cref{lem:q+t/2evenGen} and \Cref{thm:Einfmod2} give the claim, using \Cref{lem:maprforC6,lem:from-integral-to-mod2} to get the claimed integral representatives.
Again, we only need to provide additional details when $q=0$.
In this case, we have the exact sequence of chain complexes 
\[ 0\to E_1^{\bullet,0 }(\EE_8) \xrightarrow{2} E_1^{\bullet,0}(\EE_8) \to E_1^{\bullet,0}(\EE_8/2) \to 0 \ . \]
The claims follow from an analysis analogous to that of \Cref{lem:caset4} of the associated long exact sequence on $E_2$-terms, using the information provided in \Cref{thm:techyd1}.
\end{proof}

\begin{lem}[$t=10$]
There are isomorphisms 
\[E_2^{p,*}(\EE_{10}) \cong \begin{cases}
\F_4[k]\{\eta c_4 \Delta_0, \epsilon \Delta_0, \eta^2 \mu\Delta_0 \} & p=0 \\
\F_4[g]\{v_1^4\eta b_0, v_1^2hb_1 \}  & p=1 \\
\F_4[g]\{v_1^4\eta \overline{b}_0, v_1^2h\overline{b}_1 \}& p=2 \\
\F_4[k]\{\eta c_4 \overline{\Delta}_0,\epsilon\overline{\Delta}_0, \eta^2 \mu\overline{\Delta}_0  \}& p=3  \ 
\end{cases}\]
with relation $k\eta c_4 \overline{\Delta_0}= \eta^5\overline{\Delta}_0$.
\end{lem} 

\begin{proof}
When $q$ is even, we are in the situation of \Cref{lem:q+t/2odd}, which gives that $E_2^{0,2*}(\EE_{10}) \cong \F_4[k]\{\epsilon {\Delta}_0  \} $,  similarly $E_2^{3,2*}(\EE_{10})  \cong \F_4[k]\{\epsilon \overline{\Delta}_0 \}  $, and $E_2^{p,2*}(\EE_{10}) = 0$ for $p=1,2$. When $q \equiv 1 \mod 4$, we are in the situation of \Cref{lem:q+t/2evenCases1}, which gives $E_2^{p,q}(\EE_{10}) \cong E_2^{p,q} (\EE_{10}/2)$ when $p=1,2$, so by \Cref{thm:Einfmod2}, we get the claimed terms in these bidegrees. We use that $\partial (k^n v_1 x) = k^n \epsilon $ in the case $p=0,3$ and $q \equiv 1 \mod 4$. When $q \equiv 3 \mod 4$, we use \Cref{lem:q+t/2evenGen} to read off the answer from \Cref{thm:Einfmod2}, using \Cref{lem:from-integral-to-mod2} to get the claimed integral representatives.
\end{proof}

This concludes the computation of the $E_2$-page of the \ref{eq:ADSS} for $\EE_t$ in the range $0\leq t <12$. 

\subsection{Higher differentials and additive extensions}

It turns out that the rest of the spectral sequence is very simple to analyze in this range: There is very little room for higher differentials, and when there is a possibility, simple arguments show that no differentials occur. A similar story holds for the extensions, with one exception as we will see.
\begin{lem}\label{lem:collapseE2}
In the range $0\leq t <12$, the algebraic duality spectral sequence $E_r^{*,*}(\EE_*)$ collapses at the $E_2$-page, i.e., there are no higher differentials.
\end{lem}

\begin{proof}
A straightforward inspection using degree arguments and the structure of the spectral sequence as a module over $H^*(\mathbb{S}_2^1, \EE_*)$ shows that there are no higher differentials in this range.  For example, all classes in filtration $p=0$ in this range are multiples of $k$ with elements in the image of the map from the ANSS (see \Cref{sec:Hurewicz} below). 
Hence these are permanent cycles. Using linearity of the $\eta$-differentials then excludes the possibility of $d_2$-differentials with source in filtration $p=1$.
\end{proof}

\begin{lem}
In the range $0\leq t<12$, there are no exotic additive extensions in the $E_{\infty}$-page of the \ref{eq:ADSS} $E_\infty^{*,*}(\EE_t)$.
\end{lem}
\begin{proof}
 Using the fact that $2\eta = 4\nu = 2\nu^2=0$, we can rule out any additive extensions except in $H^4(\SS_2^1,\EE_0)$. We explain how this goes for $q=3$ and $t=0$: We have an exact sequence 
 \[0 \to \WW\{\overline{\Delta}_0\} \cong \EE_{\infty}^{3,0} \to H^3(\SS_2^1, \EE_0) \to \EE_{\infty}^{1,2}  \cong \F_4 \{gb_0\} \to 0. \]
 Let $[x]$ be the class in cohomology detected by $x$ in the ADSS. If $2[gb_0]$ is detected by $a\overline{\Delta}_0$ for $a\in \WW$, then $a \eta \overline{\Delta}_0 =0$, which implies $a =2a'$ for $a'\in \WW$, since $\eta \overline{\Delta}_0 \neq 0$ and has order $2$. But then $[gb_0] -a'[\overline{\Delta}_0]$ is detected by $gb_0$ and is a class of order 2, giving a splitting of the exact sequence. The other cases for $(q,t)$ are even more straightforward.

For $q=4$, we instead use the splitting of the edge homomorphism. We have
  \[0 \to \FF_4\{g\overline{b}_0\} \cong \EE_{\infty}^{2,2} \to H^4(\SS_2^1, \EE_0) \to \EE_{\infty}^{0,4}  \cong \WW/8 \{k\} \to 0  .\]
 Since the edge homomorphism splits (\Cref{lem:kfixed}), $k$ has order $8$ and so the second extension splits.
\end{proof}
The result is depicted in \Cref{fig:Einf}, in which the class $\Delta_0$ is written as $1$.

  \begin{figure}[h]
    \includegraphics[width=\textwidth]{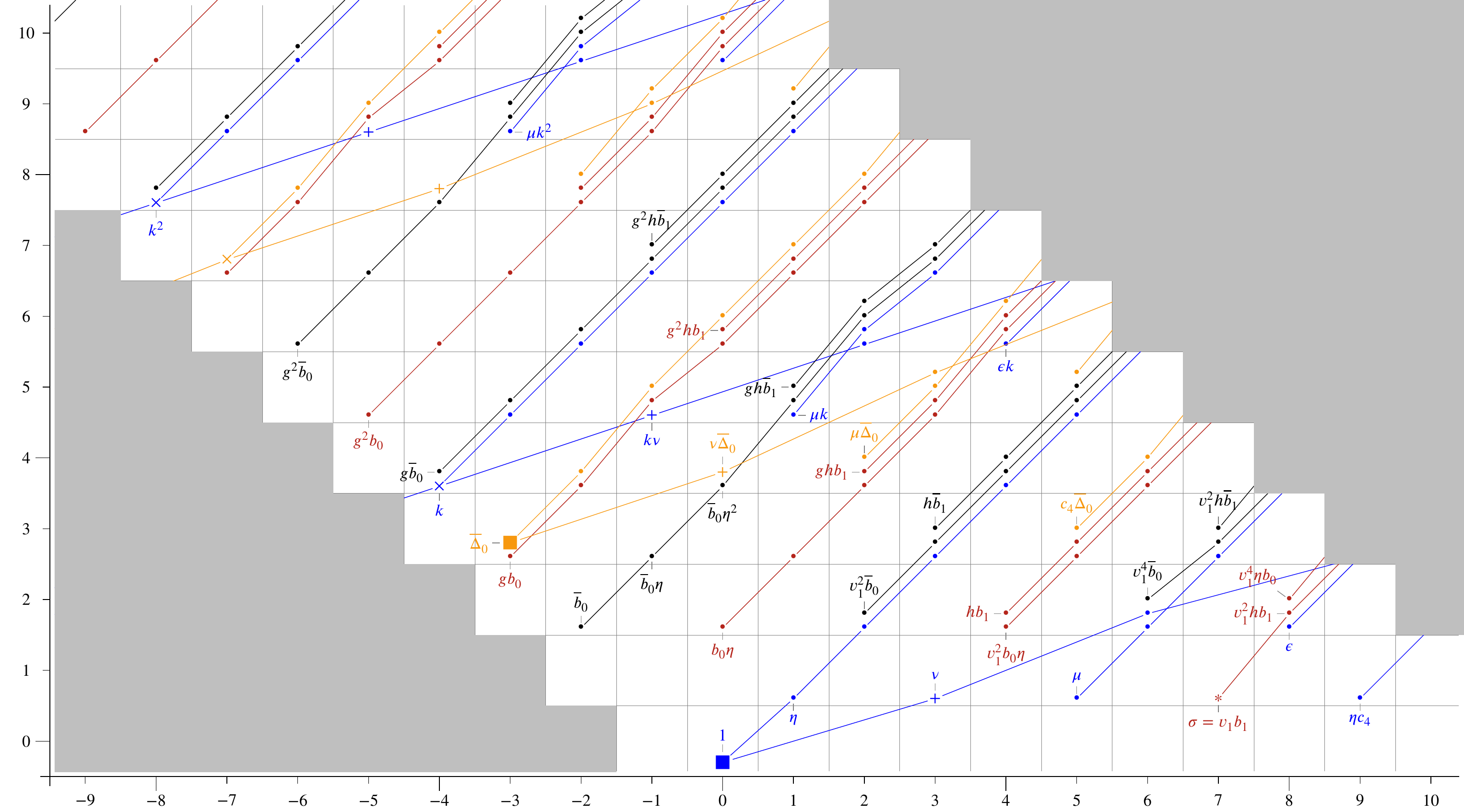}
\caption[The $E_{\infty}$-page of ADSS $E_*^{*,*}(\EE_*)$]{The $E_{\infty}$-page of ADSS $E_*^{*,*}(\EE_*)$. Because of the absence of additive extensions, this is also $H^*(\SS_2^1, \EE_*)$. A $\blacksquare = \W$, a $\bullet = \F_4$, $+ = \W/4$, $\times =\W/8$ and $\ast =\W/16$. The grading and the ADSS filtration is as in \cref{fig:ADSS-Mod2}. Classes in {\color{blue} blue} and  {\color{YellowOrange} orange} come from \cref{fig:c6e2v0} and have filtration $p=0,3$ respectively. Classes {\color{BrickRed} red} and black come from \cref{fig:c6e2} and have filtration $p=1,2$ respectively. }
\label{fig:Einf}
\end{figure}

\section{Organizing the Cohomology of $\SS_2^1$}\label{sec:name}

In this section we give names to the elements in $H^*(\SS_2^1, \EE_t)$ in the range $0\leq t <12$, which was computed in the previous section, and describe the structure of $H^*(\SS_2^1, \EE_t)$ in more detail. As much as possible, we want to state our answer in terms of known elements coming from the Adams--Novikov spectral sequence, and in terms of classes that played an important role in the study of $L_{1}L_{K(2)}S^0$ in \cite{BGH}. 
\subsection{Naming conventions}\label{sec:nameconventions} Some of our conventions are summarized here:
 \begin{enumerate}[(1)]
 \item
If $b \in E_{\infty}^{p,q}(\EE_t)$ is an element of the ADSS that detects a class $\beta$ in cohomology, then we write
 \[ \beta = [b] \in H^{p+q}(\SS_2^1, \EE_t). \]
 
 \item 
 Some classes have names that look like products, these names are chosen to be suggestive of relations among the elements, and we also use $[-]$ to indicate that this is not a product despite the notation. An example is the element which we will call $[\etachi]$, discussed below. This element is not a multiple of $\eta$. However, $\eta [\etachi]$ is a product of $\eta$ and $[\etachi]$.
 \item \label{itemv1powers} The class $v_1 \in H^0(\mathbb{S}_2^1, \EE_2/2)$ acts on the cohomology groups
 $H^*(\mathbb{S}_2^1, \EE_*/2)$. Given any class $x \in H^*(\mathbb{S}_2^1, \EE_t/2)$, the $v_1$-multiple $v_1 x$ is a well-defined class in $H^*(\mathbb{S}_2^1, \EE_{t+2}/2)$.
We use this to name classes in $H^*(\mathbb{S}_2^1, \EE_*)$ as follows. Suppose that the natural maps
\begin{align}\label{eq:redmod2} 
H^s(\mathbb{S}_2^1, \EE_t) &\to H^s(\mathbb{S}_2^1, \EE_t/2) \\
H^s(\mathbb{S}_2^1, \EE_{t+2n}) &\to H^s(\mathbb{S}_2^1, \EE_{t+2n}/2) \nonumber
\end{align}
 are injective  for some $(s,t)$ and $n$. This is not always the case but often is in our range. 
Suppose further that $y \in H^s(\mathbb{S}_2^1, \EE_t)$ maps to $x \in H^s(\SS_2^1, \EE_t/2)$ and that there is a class in $H^s(\mathbb{S}_2^1, \EE_{t+2n})$ that maps to $v_1^nx$. Then we denote this unique class
 \item In our names below, we use the relations
\begin{align*}
[v_1^2]g &= \eta^2 & v_1h&=\eta
\end{align*}
in $H^*(C_6, \EE_*)$ (\Cref{lem:C6coh} and \eqref{eq:C6cohElements}) and $H^*(C_6,\EE_*/2)$ (\Cref{lem:cohC6mod2}) to justify writing
\[g  = \frac{\eta^2}{v_1^2} \quad \text{and} \quad h =\frac{\eta}{v_1} .\]
The advantage is that, in $H^*(\mathbb{S}_2^1, \EE_*/2)$, both $v_1$ and $\eta$ act on the spectral sequence, and if a class is named
\[\left[ \frac{\eta^2 x}{v_1^2}\right] \]
then $v_1^2\left[ \frac{\eta^2x}{v_1^2}\right] = \eta^2x$ in $H^*(\mathbb{S}_2^1, \EE_*/2)$. 
 \end{enumerate}

\subsection{Classes coming from the ANSS}\label{sec:Hurewicz}
For any closed subgroup $G$ of $\mathbb{G}_2$, the $K(n)$-local $E_n$-based Adams-Novikov Spectral Sequence for the sphere can be identified with the homotopy fixed point spectral sequence \cite[Appendix A]{DH}. From this, it follows from this that the unit  
\[S^0 \to E^{hG}\] of the ring spectrum $E^{hG}$ gives rise to a map from the $BP$-based Adams--Novikov spectral sequence (ANSS) for $S^0$ to the homotopy fixed point spectral sequence for $E^{hG}$. In particular, we have a diagram of spectral sequences
\[\xymatrix{ \Ext_{BP_*BP}^{*,*}(BP_*,BP_*) \ar@{=>}[d] \ar[r] & \ar[r] H^*(\mathbb{G}_2, \EE_*)\ar@{=>}[d]   \ar[r] &  H^*(G, \EE_*)\ar@{=>}[d]  \\
\pi_*S^0 \ar[r] & \pi_* L_{K(2)}S^0 \ar[r] & \pi_*\EE^{hG} \ .
}  \]
See the discussion surrounding Equation~(2.1.9) of \cite{BGH} for a detailed review of this setup.
\begin{defn}\label{rem:hurimagelanguage}
We say that a cohomological class \emph{comes from the ANSS} if it is in the image of the map
$ \Ext_{BP_*BP}^{*,*}(BP_*,BP_*) \to H^*(G, \EE_*) $. 
\end{defn}

The image of the $\alpha$-family in $H^*(\mathbb{S}_2^1, \EE_*)$ was studied in \cite{BeaudryAlpha} and we have:
\begin{align}\label{eq:list1}
  \eta &= \alpha_1 = [\eta \Delta_0] \in  H^1(\SS_2^1, \EE_2), &
 \nu &= \alpha_{2/2} = [\nu \Delta_0] \in  H^1(\SS_2^1, \EE_4),   \\
\mu&=\alpha_3 = [\mu \Delta_0] \in  H^1(\SS_2^1, \EE_6), &
\sigma & = \alpha_{4/4} = [v_1b_1] \in  H^1(\SS_2^1, \EE_8),  \nonumber \\
  \alpha_5 &=  [\eta c_4] \in  H^1(\SS_2^1, \EE_{10}). \nonumber
\end{align}

Beyond the $\alpha$-family, we also have
\begin{align} \label{eq:list2}
\nu^2 &= \alpha_{2/2}^2 = [\nu^2 \Delta_0] \in H^2(\mathbb{S}_2^1, \EE_{8}) \\
\epsilon &= \beta_2 =  [\epsilon \Delta_0] \in H^2(\mathbb{S}_2^1, \EE_{10})  \nonumber \ .
\end{align}

We note that in the range $0\leq t<12$, the map 
$\Ext_{BP_*BP}^{*,*}(BP_*,BP_*) \to  H^*(\mathbb{S}_2^1, \EE_*)$
is an injection (see \cite[Table 2]{ravnovice} or   \cite{IsaksenANSSChart} for charts of the ANSS $E_2$-page).

\subsection{The $\wchi$ and $[\etachi]$ family} 

One of the most important elements is 
\begin{align} \label{eq:list3}
\wchi &= [\overline{b}_0] \in H^2(\mathbb{S}_2^1, \EE_0) 
\end{align}
defined in \eqref{def:wchi}. That $\wchi$ is detected by $\overline{b}_0$ is proved in  \cite[Lemma 5.2.10]{BGH}. By definition, the class $\wchi$ is the Bockstein of a class $\chi \in H^1(\mathbb{S}_2^1,\EE_0/2)$, which is detected by $b_0$ in the ADSS for $\EE_*/2$. There is in fact a Massey product \cite[Theorem 9.1.6 and Theorem 9.1.7]{BGH}
\[ \langle \wchi,2,\eta \rangle = [\eta b_0] \in H^2(\mathbb{S}_2^1, \EE_2)  
\]
with no indeterminacy. Since it is a Massey product of classes that are both $\ZZ_2$ and $\Gal$ invariant, this is in fact a class in $H^2(\GG_2, \EE_2)$.
Further, modulo $2$ this bracket maps to the product $\etachi$ in $H^2(\mathbb{G}_2, \EE_2/2)$, so we introduce the notation
\begin{equation}\label{eq:massey}
\boxed{
[\etachi]:=  \langle \wchi,2,\eta \rangle \in H^2(\mathbb{G}_2, \EE_2)}
\end{equation}
even if this is not a product as there is no class $\chi$ in integral cohomology.

As stated in \cref{lem:t=0}, as a $\W[k]$-module, 
$H^*(\mathbb{S}_2^1, \EE_0)$
is generated by
\begin{align}\label{eq:list4}
1 &=[\Delta_0] \in H^0(\mathbb{S}_2^1, \EE_0) &
\wchi &= [\overline{b}_0] \in H^2(\mathbb{S}_2^1, \EE_0) \\ 
\gchi &= [gb_0] \in H^3(\mathbb{S}_2^1, \EE_0) &
e&=[\overline{\Delta}_0] \in H^3(\mathbb{S}_2^1, \EE_0)\nonumber \\
 \gwchi &= [g\overline{b}_0] \in H^4(\mathbb{S}_2^1, \EE_0) &
\kchi &= [g^2b_0] \in H^5(\mathbb{S}_2^1, \EE_0) \nonumber.
 \end{align}

The next names we give have to do with the $v_1$-periodic nature of the elements $\wchi$ and $[\etachi]$.
Our computations of the ADSS imply that the map 
\begin{align}\label{eq:inj}
H^2(\mathbb{S}_2^1, \EE_*) &\to  H^2(\mathbb{S}_2^1, \EE_*/2), & 0\leq t< 12 \\
H^3(\mathbb{S}_2^1, \EE_*)/(2e) &\to  H^3(\mathbb{S}_2^1, \EE_*/2), & 0\leq t< 12  \nonumber
\end{align}
are injective. To see this, compare \Cref{thm:Einfmod2} and \Cref{sec:Einf}. See also \Cref{fig:Einf}. In $H^*(\mathbb{S}_2^1, \EE_*/2)$, the classes detected by $\eta b_0, hb_1, \overline{b}_0, gb_0, h\overline{b}_1$ are $v_1$-free. See \cite[Theorem 1.2.2]{BeaudryTowards}.

The injectivity of \eqref{eq:inj} and the discussion at the beginning of the section, then leads to the following names to the classes in our range of interest:
\begin{align}\label{eq:list5}
[v_1^{2n}\wchi] &= [v_1^{2n} \overline{b}_0] \in H^2(\mathbb{S}_2^1, \EE_{4n}) \\
[v_1^{2n}\etachi] &= [v_1^{2n}\eta b_0] \in H^2(\mathbb{S}_2^1, \EE_{2+4n}) \nonumber   .
\end{align}

\subsection{Some multiplicative relations}
We have a relation
\[[\etachi]^2=\eta^2\wchi \]
which follows from the fact that $\wchi \equiv \chi^2 \mod 2$. 
Furthermore, the fact that $\chi^3=0$ in $H^*(\mathbb{S}_2^1, \EE_0/2)$ (\cite[Proposition 5.2.13]{BGH}) and the injectivity of the mod $2$ map in the relevant degrees \eqref{eq:redmod2} force the relations
\[ [\etachi]^3=0, \quad \quad  [\etachi]\wchi =0. \]

We also have relations
\begin{align}\label{eq:murel}
  \eta^n[v_1^{2n}\wchi] &= \mu^n \wchi, &  \eta^n[v_1^{2n}\etachi]  &= \mu^n [\etachi]  
 \end{align}
coming from
\[ \mu =[v_1^2]\eta \quad \text{and}\quad g[v_1^2] = \eta^2\]
 in $H^*(C_6, \EE_*)$.

\subsection{The $\sigma$ family}
 Next we turn to describing $\sigma = [v_1b_1]$-multiplication in our range of interest. 
 The following result addresses two hidden $\sigma$-multiplications.

\begin{lem}[{\cite[Theorem 8.2.4]{BGH}}]\label{lem:sigmatimes}
In $H^*(\SS_2^1,\EE_*)$, we have the following products
\begin{align*}
\sigma \wchi &= [c_4\overline{\Delta}_0] \in H^3(\SS_2^1, \EE_{8}); \\
\sigma [\etachi] &= [a_0v_1^2h\overline{b}_1 + a_1\eta v_1^4\overline{b}_0 ] \in H^3(\SS_2^1, \EE_{10}),
\end{align*}
for some $a_0 \in \F_4^{\times}, a_1 \in \F_4$.
\end{lem}

We turn to naming conventions for classes detected by 
\[hb_1,ghb_1, h\overline{b}_1, gh\overline{b}_1,\]
which we include in what we call the ``$\sigma$-family'' because of the important role $b_1$ and $\overline{b}_1$ play in the detection of $\sigma$ as shown above. 
We let
\begin{align}\label{eq:list6}
\hbone &= [hb_1] \in H^2(\mathbb{S}_2^1, \EE_6)  \\
\hbonebar &= [a_0h\overline{b}_1 + a_1\eta v_1^2\overline{b}_0 ] \in H^3(\mathbb{S}_2^1, \EE_6),  \quad a_0 \in \F_4^{\times}, a_1 \in \F_4 \nonumber \\
\ghbone&= [ghb_1] \in H^4(\mathbb{S}_2^1, \EE_6)\nonumber \\
\ghbonebar &= [a_0 gh\overline{b}_1 + a_1 \eta^3\overline{b}_0] \in H^5(\mathbb{S}_2^1, \EE_6) \ . \nonumber 
 \end{align}
 We then have the following relations in $H^s(\SS_2^1, \EE_*)$, as suggested by the names:
 \[[\etachi]  \hbone = \eta   \hbonebar \quad \text{and} \quad [\etachi]  \ghbone =\eta  \ghbonebar   \ . \]
Using the relations $\mu=\eta[v_1^2]$ and $g[v_1^2] =\eta^2$ in $H^*(C_6, \EE_*)$ (\Cref{lem:C6coh} and \eqref{eq:C6cohElements}),
we have
 \[ \mu  \ghbone = \eta^3 \hbone \quad \text{and} \quad  \mu \ghbonebar = \eta^3 \hbonebar \]
Finally, the relationship of these classes to $\sigma$ is
\[\mu \hbone  = \eta^2 \sigma   \quad \text{and} \quad  \mu \hbonebar= \eta\sigma [\etachi]  \ . \]

 \subsection{Statement of the result}

All classes in the range now have been renamed, either explicitly or as a product of classes from \eqref{eq:list1}, \eqref{eq:list2}, \eqref{eq:list3}, \eqref{eq:list4}, \eqref{eq:list5},  \eqref{eq:list6} and the class $k$ of \eqref{eq:defk}.

 \begin{theorem}
 The cohomology $H^s(\SS_2^1, \EE_t)$ in the range $0\leq t<12$, and for $s\leq 6$ is the $\WW$-module generated by the classes listed in \Cref{fig:table}, with the only additive relation given by the order of the class indicated below the element in the table. 
 \end{theorem}

 \begin{rem}
There are two comments to be made about \Cref{fig:table} at this point:
\begin{itemize}
\item For the moment, we've only computed the cohomology of $\SS_2^1$ in \Cref{fig:table}.  The Galois-invariance which will address the $\G_2^1$-cohomology statements in \Cref{fig:table} will be addressed in \Cref{sec:gal}.
\item If we are only interested in the additive structure of $H^s(\SS_2^1, \EE_t)$ in the range $0\leq t<12$, or in fact also its $\W[k,\eta]$-module structure, then \Cref{fig:table} simply represents a tabulation of the results of \Cref{sec:main}, which are also depicted in \Cref{fig:Einf}. However, the new names which were introduced in this section also reflect additional information about the multiplicative structure, such as multiplication by $\sigma$.
\end{itemize}
\end{rem}

\section{The $\Z_2$-action and the cohomology of $\SS_2$}\label{sec:Z2action}
 
Our next goal is to compute the cohomology of $\mathbb{S}_2$ using the extension
\[ 1 \to \mathbb{S}_2^1 \to \SS_2 \xrightarrow{\zeta} \Z_2^{\times}/(\pm 1) \cong \Z_2 \to 1\]
where  $\zeta \in H^1(\SS_2,\ZZ_2)$ is the class defined in \eqref{def:zeta}.
We split this sequence using the element $\pi$ introduced in \Cref{rem:pialpha}, so we also write $\pi$ for the corresponding generator of the quotient.
 
The Serre spectral sequence for the cohomology of the group extension and the fact that $\Z_2$ has cohomological dimension $1$ give the exact sequence
\[
\xymatrix@R=1pc{
0 \ar[r] &  H^1(\ZZ_2, H^{s-1}(\SS_2^1,\EE_t)) \ar@{=}[d] \ar[r] &  H^s(\SS_2,\EE_t) \ar[r] & H^0(\ZZ_2, H^s(\SS_2^1, \EE_t)) \ar[r] \ar@{=}[d]&  0 \\
 & H^{s-1}(\SS_2^1,\EE_t)_{\Z_2} &  & H^s(\SS_2^1, \EE_t)^{\ZZ_2} &
} \]
where $M_{\ZZ_2}$ and $M^{\ZZ_2}$ mean coinvariants and fixed points respectively.

We will show below that the action of $\ZZ_2$ on $H^s(\SS_2^1, \EE_t)$ is trivial in the range $0\leq t<12$.
Hence, in our range, the above exact sequence becomes:
\[
\xymatrix@R=1pc{
0 \ar[r] &  H^{s-1}(\SS_2^1,\EE_t)  \ar[r]^-{\zeta} &  H^s(\SS_2,\EE_t) \ar[r] & H^s(\SS_2^1, \EE_t) \ar[r] &  0 .
} \]

\subsection{The action of $\Z_2$ and the resulting cohomology}
 The first step is to compute the action of $\Z_2$ on $H^*(\SS_2^1,\EE_*)$.
We isolate some key facts about the action in the following remark.
\begin{rem}\label{rem:easyZfixed}
\begin{enumerate}
\item
The action of $\Z_2$ on $H^*(\SS_2^1, \EE_*)$ is through continuous $\WW$-linear ring maps. For a fixed $(q,t)$, $\Z_2$ acts via a continuous homomorphism 
\[\Z_2 \to \Aut_{\W}(H^q(\SS_2^1, \EE_t)).\] 
\item The multiplicative structure of the action implies that for any product $xy$ in $H^*(\SS_2^1,\EE_*)$, $\piact(xy)=(\piact x)(\piact y)$.
\item  Any class coming from the ANSS in the sense of \Cref{rem:hurimagelanguage} is fixed by $\Z_2$ because we have a factorization
\[ \Ext_{BP_*BP}(BP_*,BP_*) \to H^*(\mathbb{S}_2,\EE_*) \to  H^*(\mathbb{S}_2^1,\EE_*)^{\ZZ_2}\subseteq  H^*(\mathbb{S}_2^1,\EE_*). \]
Therefore, the action is linear with respect to multiplication by classes coming from the ANSS. For example, $\piact (\eta y)=\eta (\piact y)$. These classes include $1,\eta,\nu,\mu,\sigma, \alpha_5, \epsilon$ which are all detected in $E_\infty^{0,*}(\EE_*)$ in the ADSS.
\item 
In fact, if $x$ is any element in $H^*(\SS_2^1,\EE_*)$, which is in the image of the restriction map from $H^*(\SS_2,\EE_*)$, then $x$ is automatically $\ZZ_2$-invariant. This includes $k\in H^4(\SS_2^1,\EE_0)$ by \eqref{eq:defk}, $\wchi \in H^2(\SS_2^1,\EE_0)$ by \eqref{def:wchi}, and
$[\etachi] \in H^2(\SS_2^1, \EE_2)$ by \eqref{eq:massey}.
Consequently, the $\ZZ_2$-action is $k, \wchi$, and $[\etachi]$-linear.
\item $\ZZ_2$ also acts on $H^*(\SS_2^1, \EE_*/2)$ and the element $v_1 \in H^0(\SS_2^1, \EE_2/2)$ is fixed. Indeed, $v_1$ is the reduction modulo $2$ of the coefficient of $x^2$ in the $[2]$-series $[2](x)\in \EE_*/2[\![x]\!]$ of the universal deformation of the formal group law used to define our $\EE$-theory. See
 \cite[Theorem 6.2.2]{BeaudryTowards}, for a discussion of the element $v_1$ for our particular choice of curve. As such, $v_1$ is the image of the same named class in $\Ext_{BP_*BP}^{*,*}(BP_*, BP_*V(0))$, and so is in the image of the composite
 \[\Ext_{BP_*BP}^{*,*}(BP_*, BP_*V(0)) \to H^*(\GG_2, \EE_*/2) \to H^*(\SS_2^1, \EE_*/2).\]
 
\end{enumerate}
\end{rem}

\begin{prop}\label{thm:bluegreenfix}
For $0\leq t<12$, all classes in $H^*(\SS_2^1;\EE_t)$ detected in $E_\infty^{p,*}(\EE_t)$ for $p=0,3$ of the ADSS are fixed by $\Z_2$.
\end{prop}
\begin{proof}
The case $p=0$ follows from the fact that in our range these classes are products of elements coming from the ANSS and of $k$, which are all fixed.
For $p=3$, we will show below in \Cref{lem:H3fixed} that $e \in H^3(\mathbb{S}_2^1,\EE_0)$ detected by $[\overline{\Delta}_0]$ is fixed.
 Any other class  in this range detected in filtration $p=3$ of the ADSS is a multiple of fixed classes (e.g. $\sigma \wchi = [c_4\overline{\Delta}_0] \in H^3(\SS_2^1,\EE_8)$).
 \end{proof}

\begin{rem}
\Cref{thm:bluegreenfix} implies that all classes in \Cref{fig:Einf} that are colored in blue or green are fixed by the action of $\ZZ_2$.
\end{rem}

Our next goal is to show that classes in filtration $p=1,2$ of the ADSS in our range are also fixed under the action of $\ZZ_2$.
Let us recall the names of maps in the following  long exact sequence
\begin{equation}\label{eq:diagramrecall}
\xymatrix@C=1pc{
\ldots \ar[r] &H^3(\SS_2^1, \EE_0) \ar[r]^2 &H^3(\SS_2^1, \EE_0) \ar[r]^r & H^3(\SS_2^1, \EE_0/2) \ar[r]^\partial & H^4(\SS_2^1, \EE_0)   \ar[r]^-2 &\ldots
}
\end{equation}

\begin{lem}\label{lem:H3fixed}
The action of $\Z_2$ on 
\[H^3(\SS_2^1, \EE_0) \cong \WW\{e\} \oplus\F_4\left\{\gchi\right\}\] 
is trivial. 
\end{lem}

\begin{proof}
An argument using the order of the classes and the $\WW$-linearity of the action implies that $\gchi$ is fixed. 
By \cite[Proposition 6.2.1]{BGH}, 
the image of the  class $e$ is fixed in 
$H^3(\SS_2^1, \ZZ_2)\otimes \QQ \cong H^3(\SS_2^1, \EE_0) \otimes \QQ$.
From this we conclude 
\[
\piact e = e + \lambda \gchi,
\]
for some $\lambda \in \F_4$. The kernel of $r: H^3(\SS_2^1,\EE_0) \to H^3(\SS_2^1,\EE_0/2)$ is generated by $2e$, thus to show that $\lambda=0$, it suffices to show that $r(e) $ is fixed in the cohomology of $\EE_0/2$.

From \cref{lem:sigmatimes}, remembering that $\wchi \equiv \chi^2$ and $c_4\equiv v_1^4$ modulo 2, we have
\[
\sigma\chi^2 = v_1^4 r(e) \in H^3(\SS_2^1,\EE_8/2).
\]
Now $\sigma$ and $\chi^2\equiv \wchi$ are $\Z_2$-invariant (see bulleted list above), so $v_1^4r(e)$ is an invariant class in $H^3(\SS_2^1,\EE_8/2)$. Thus we have 
\[ \piact (v_1^4 r(e) )  = v_1^4 r(e) = v_1^4 r(e) + \lambda v_1^2\eta^2\chi,\]
where the second equation follows from the above conclusion on the $\pi$-action on $e$, the multiplicativity of this action, as well as the $\pi$-invariance of $v_1$. The element $v_1^2\eta^2\chi$ is what appears as $v_1^4 h^2 \chi$ in \Cref{thm:Einfmod2}; in particular it is non-zero, allowing us to conclude that $\lambda = 0$ as desired.
\end{proof}

\begin{lem}\label{lem:H3fixedMod2}
The action of $\Z_2$ is trivial on $H^3(\SS_2^1, \EE_0/2)$.
\end{lem} 
\begin{proof}
\cref{thm:Einfmod2} implies that $H^3(\SS_2^1, \EE_0/2)$ has dimension $4$ over $\F_4$, with basis detected by $ \nu^2y\Delta_{-1}$, $h^2b_0$, $h\overline{b}_0$ and $\overline{\Delta}_0$ in the ADSS for $\EE_0/2$. The class $\overline{\Delta}_0$ detects $r(e)$ and $h^2b_0$ detects $r\left(\gchi \right)$. The class $\nu^2y\Delta_{-1}$ has ADSS filtration 0, and detects the unique $v_1$-torsion element in $H^3(\SS_2^1, \EE_0/2)$, which we will call $z$ in this proof. 

Let $\hchi$ be any class detected by $hb_0$. We claim that $h\overline{b}_0$ detects the product $\chi \hchi$. After inverting $v_1$, this follows from the structure of $v_1^{-1}H^*(\SS_2^1, \EE_*/2)$ as an algebra as computed in \cite[Theorem 8.2.5]{BGH}.
Therefore, $\chi \hchi $ is detected by $h\overline{b}_0$ modulo $v_1$-torsion. The only $v_1$-torsion element in $H^3(\SS_2^1, \EE_0/2)$ is $z$, but $z$ has ADSS filtration zero, which proves the claim.

The classes which are in the image of $r$ are invariant by the previous \Cref{lem:H3fixed}, while $\chi \hchi $ is a product of $\Z_2$-invariant classes, hence itself invariant.
Finally, $z$
is the only $v_1$-torsion element, thus we must have that $\piact z = z$.
\end{proof}

\begin{lem}\label{lem:H4fixed}
The action of $\Z_2$ is trivial on
\[ H^4(\SS_2^1, \EE_0) \cong \W/8\{k\} \oplus\F_4\left\{\gwchi\right\}.\] 
\end{lem}
\begin{proof}
The action on  $k$ is trivial by \Cref{rem:easyZfixed}(4). Since $\gwchi$ has order $2$, it is an element in $\im( \partial) \subseteq H^4(\SS_2^1,\EE_0)$ for $\partial$ as in \eqref{eq:diagramrecall}. However, $\partial $ is a $\Z_2$-equivariant map with source which is a trivial $\Z_2$-module by \Cref{lem:H3fixedMod2}, so $\im \partial $ is a trivial $\Z_2$-module, and in particular, $\piact \gwchi = \gwchi$. 
\end{proof}

\begin{prop}
For $0\leq t<12$, all classes in $H^{s}(\SS_2^1;\EE_t)$ detected in the ADSS $E_\infty^{p,s-p}(\EE_t)$ in filtrations $p=1,2$ are fixed by $\Z_2$.
\end{prop}

\begin{proof}
We have
\begin{itemize} 
\item $\wchi$ and $[\etachi]$ are fixed, by \Cref{rem:easyZfixed}; 
\item $\gchi$ is fixed, by \Cref{lem:H3fixed};
\item $\gwchi$ is fixed, by \Cref{lem:H4fixed}; 
\item $\kchi$ is fixed, since $H^5(\SS_2^1,\EE_0) $ is one-dimensional over $\F_4$ with $\kchi$ as generator.
\end{itemize}
Along with the multiples of these classes by ANSS classes and $k$, which are fixed by \Cref{rem:easyZfixed}, this includes all classes with $0\leq t\leq 3$. It remains to prove that classes in degrees $4\leq t\leq 10$ are fixed.

The method to address this is to study the image of these classes under the ($\Z_2$-equivariant) map
\[H^*(\SS_2^1,\EE_*) \xrightarrow{r} H^*(\SS_2^1,\EE_*/2) . \]
On $H^s(\SS_2^1,\EE_t)$, the map is injective if $t=6,10$; and for $t=4,8$ when $s>1$. When $s=1$, $2\nu$ and $2\sigma$ map to zero, but these classes are clearly fixed since they come from the ANSS in the sense of \Cref{rem:hurimagelanguage}. 
So, we may assert that if the image of a class is fixed in $ H^*(\SS_2^1,\EE_t/2)$, then it is fixed in $H^*(\SS_2^1,\EE_t) $ for $4\leq t\leq 10$. 

The advantage of passing to $\EE_*/2$ is that in $H^*(\SS_2^1, \EE_*/2)$, we have the $\Z_2$-fixed element $v_1 \in H^0(\SS_2^1, \EE_2/2)$, and we can use $v_1$-multiplication to make arguments. 

The map $v_1^2\colon H^s(\SS_2^1, \EE_t/2) \to H^{s}(\SS_2^1, \EE_{t+4}/2)$ is injective in the cases when $s=2$ and $0\leq t \leq 4$, as well as $2 \leq s \leq 5$ and $t = 6$.
This implies that $[v_1^{2n}\wchi]$, $[v_1^{2n}\etachi]$ for $n=1,2$ are $\Z_2$-fixed.
 Since 
$v_1^2\hbone = \eta\sigma$,
which is fixed, $\hbone$ must also be fixed. 
This accounts for all classes with $s\leq 2$. 

When $s=3$, the class $\hbonebar$ is fixed since $v_1^2\hbonebar = \sigma [\etachi]$, which is fixed. Similarly, when $s=4,5$, the $v_1^2$-multiples of $\ghbone$ and $\ghbonebar$ are 
$\eta$-multiples of fixed classes, hence they must be fixed by the injectivity of $v_1^2$ multiplications in these degrees. 

The multiples of these classes by ANSS elements and $k$ then take care of everything else, proving the claim.
\end{proof}

We summarize the above claims in the following result.
\begin{thm}\label{thm:acttriv}
In the range $0\leq t<12$, the action of $\Z_2$ on $H^*(\SS_2^1, \EE_*)$ is trivial.
\end{thm}

\subsection{Extensions from invariants to coinvariants}
We now turn to solving the extension problem coming from the exact sequence
 \[
\xymatrix@R=1pc{
0 \ar[r] &  H^{s-1}(\SS_2^1,\EE_t)_{\Z_2} \ar[r] &  H^s(\SS_2,\EE_t) \ar[r] & H^s(\SS_2^1, \EE_t)^{\ZZ_2} \ar[r] &  0.
} \]

We will see that in our range, this always splits, giving the following result.

\begin{thm}
In degrees $0\leq t<12$, there is an isomorphism of bi-graded abelian groups
\[ H^*(\SS_2,\EE_t) \cong H^*(\SS_2^1, \EE_t)\otimes_{\WW} E(\zeta) \]
where $E(\zeta)$ is an exterior algebra over $\WW$ on a class $\zeta \in  H^1(\mathbb{S}_2, \EE_0)$. 
\end{thm}
\begin{rem}
The cohomology $H^s(\SS_2, \EE_t)$ in the range $0\leq t<12$, and for $s\leq 6$ is the $\WW$-module generated by the classes listed in \Cref{fig:table2}, with the only additive relation given by the order of the class indicated below the element in the table. The result is also depicted in \Cref{fig:EinfS2}.
\end{rem}

\begin{proof}[Proof sketch.]
We don't argue in each degree as the arguments are easy and repetitive. However, we give an idea of the methods used to make the arguments as well as explicit examples in Lemmas~\ref{lem:split30}, \ref{lem:split40}, \ref{lem:split50} and \ref{lem:split54} below.

As mentioned above, we have exact sequences
\begin{equation}\label{eq:zetaeq}
0 \to H^{s-1}(\SS_2^1, \EE_t)\zeta  \to H^{s}(\SS_2, \EE_t) \to H^{s}(\SS_2^1, \EE_t)  \to 0
\end{equation}
and we need to prove that these are additively split. For this, we need to show that the classes of $H^{s}(\SS_2^1, \EE_t) $ lift to classes of the same order in $H^{s}(\SS_2, \EE_t)$. The argument is a brute force computation using the multiplicative structure of $H^{*}(\SS_2, \EE_*) $ and basic algebra facts.

 In this section, if $x \in H^{s}(\SS_2^1, \EE_t)$, then any class in $H^{s}(\SS_2, \EE_t) $ that maps to $x$ will be called $\overline{x}$. (After this section we will get rid of the redundance in names and just assume we make a choice that we call $x$.) The image of $x$ in $H^{s+1}(\SS_2, \EE_t)$ under multiplication by $\zeta$  will be called $\zeta x$. 

Some key tricks are:
\begin{itemize}
\item The class $k \in H^4(\SS_2^1, \EE_0)$ lifts to a class of order $8$ in $H^{4}(\SS_2, \EE_0)$. This follows from the discussion about $k$ around \Cref{lem:kfixed}. 
\item Any class coming from the ANSS
lifts to a class of the same order in $H^{*}(\SS_2, \EE_*) $ because of the factorization
\[\xymatrix{ \Ext_{BP_*BP}^{*,*}(BP_*,BP_*) \ar[r] & \ar[r] H^*(\mathbb{S}_2, \EE_*)   \ar[r] &  H^*(\SS_2^1, \EE_*)
} . \]
Together with the previous claim about $k$, this implies that all classes detected in filtration $0$ of the ADSS in $H^{*}(\SS_2^1, \EE_*) $ lift to classes of the same order in $H^{*}(\SS_2, \EE_*)$. 
\item This also tells us a bound on the order of many products.  If $x \in  H^*(\SS_2, \EE_*)$  has order $n$ and $y \in H^{*}(\SS_2^1, \EE_*) $ is another class, then $xy \in H^{*}(\SS_2^1, \EE_*) $ lifts to a class of order at most $n$. For example, any class in the image of multiplication by $\eta$ or $\nu^2$ lift to classes of order $2$. Products of the form $xk$ have order bound by that of $x$.
\end{itemize}

These kinds of arguments imply that the short exact sequences \eqref{eq:zetaeq} are split for almost all cases in our range. There are a couple slightly trickier arguments which we isolate in the lemmas that take the rest of this section.
\end{proof}

\begin{lem}\label{lem:split30}
The sequence
\[0 \to \F_4\{\wchi\zeta\}  \to H^{3}(\SS_2, \EE_0) \to  \W\{e\}\oplus \F_4\left\{\gchi\right\} \to 0\]
is additively split.
\end{lem}
\begin{proof}
That $e$ lifts to a class of infinite order follows by the freeness of the summand. If $\overline{\gchi}$ is any lift of $\gchi$, and $2\overline{\gchi }= a \wchi \zeta$, then $0 = 2\eta \overline{\gchi} = a \eta \wchi \zeta$, since $2\eta =0$. But $\eta\wchi \zeta$ is a non-zero class generating an $\F_4$, thus we get that $a$ is zero, and $\overline{\gchi}$ gives a splitting.
\end{proof}

\begin{lem}\label{lem:split40}
The  sequence
\[0 \to \W\{e \zeta\}\oplus \F_4\left\{\gchi\zeta\right\}  \to H^{4}(\SS_2, \EE_0) \to\W/8\{k\}\oplus \F_4\left\{\gwchi\right\} \to 0 \]
is additively split.
\end{lem}
\begin{proof}
We have already shown that $k$ lifts to an element of order $8$. 
Assume that $a\in \W, b\in \F_4 $ are constants such that
\[2 \overline{\gwchi} = a  e \zeta + b \gchi\zeta,\] 
for an arbitrary lift $\overline{\gwchi}$. Multiplying both sides of this equation by $\eta$, we determine that $b=0$ and $a=2a'$ for some $a'\in \WW$. But if that is the case, then $2\left(\overline{\gwchi}-a'e\zeta\right)=0$, and so
$\left(\overline{\gwchi}-a'e\zeta\right)$ is a lift of $\gwchi$ of order $2$, giving the splitting and proving the claim. 
\end{proof}

\begin{lem}\label{lem:split50}
The sequence
\[0 \to  \W/8\{k\zeta\}\oplus \F_4\left\{\gwchi\zeta\right\}  \to H^{5}(\SS_2, \EE_0) \to  \F_4\left\{\kchi\right\} \to 0 \]
is additively split.
\end{lem}
\begin{proof}
Choose a lift $\overline{\kchi}$. Then 
\[2\overline{\kchi} = ak\zeta + b \gwchi\zeta.\] Multiplying by $\eta$ gives $b=0$ and $a=2a'$. Then $\overline{\kchi}-a'k$ is a lift of $\kchi$ of order $2$. 
\end{proof}

\begin{rem}
The same argument as in \Cref{lem:split40} and \Cref{lem:split50} is used to prove that the sequence for $H^{2}(\SS_2, \EE_t)$ splits for $t=4,8$, although when $t=4$, we use $\nu$ multiplication instead of $\eta$ multiplication to make the argument. 
\end{rem}

We will do one more explicit example for good measure.
\begin{lem}\label{lem:split54}
The sequence
\[0 \to  \W/4\{\nu e\zeta\}\oplus \F_4\{\eta^2\wchi\zeta\}  \to H^{5}(\SS_2, \EE_4) \to  \W/4\{\nu k\}\oplus \F_4\left\{\eta^2\gchi, \eta^2e\right\} \to 0\]
is additively split. 
\end{lem}
\begin{proof}
Both $\eta^2\gchi$ and $ \eta^2e$ have lifts of the form $\eta^2 \overline{\gchi}$, $\eta^2\overline{e}$, so they have order $2$.  
Since $\nu $ and $k$ are already defined as elements in $H^*(\SS_2,\EE_*)$, the order of their product $\nu k$ is bounded by the order of $\nu$, which is 4 (in cohomology). 
\end{proof}

  \begin{figure}[H]
 \includegraphics[width=\textwidth]{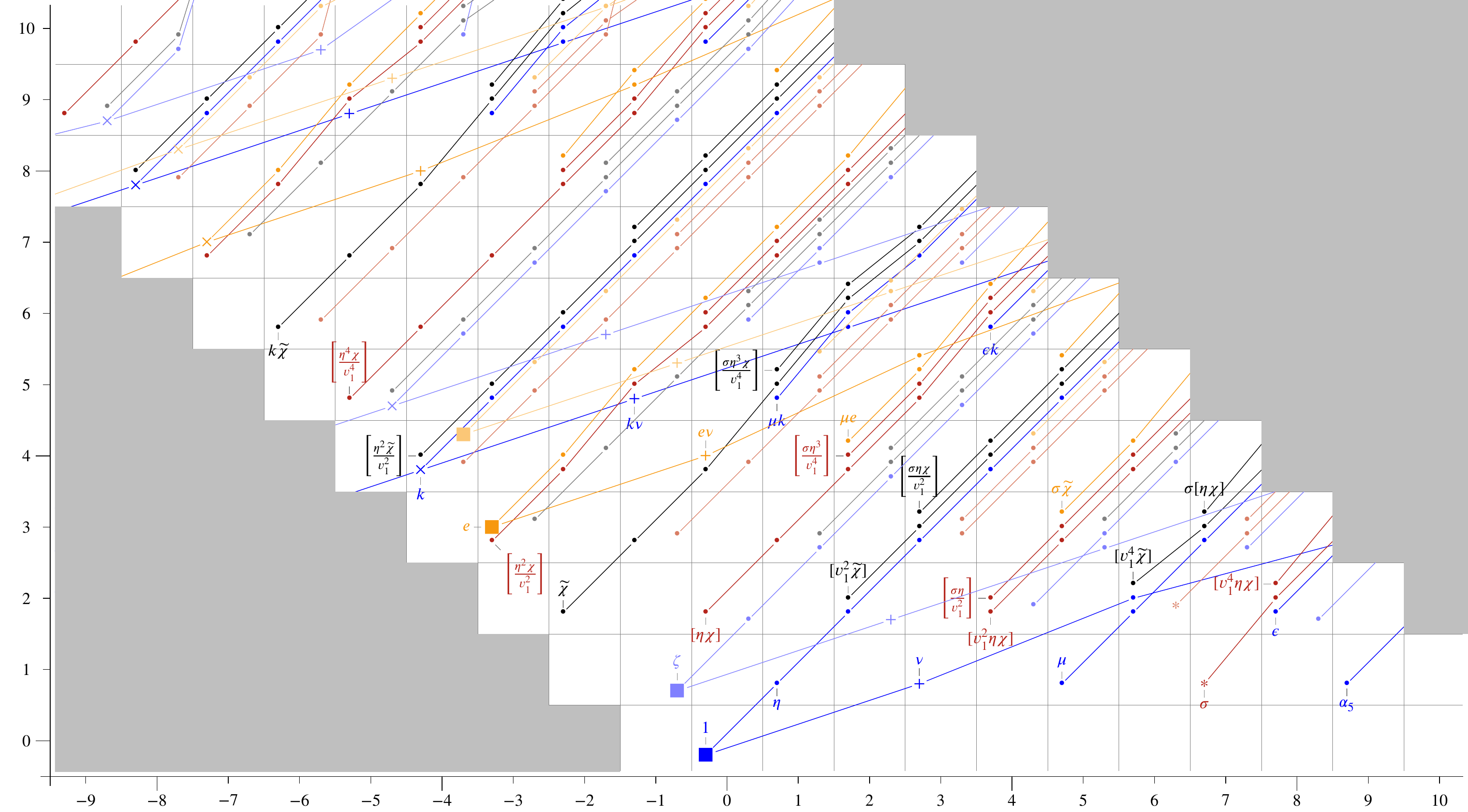}
\caption[$H^s(\SS_2, \EE_t)$ and $H^s(\GG_2, \EE_t)$ in Adams grading $(t-s,s)$]{The cohomology $H^s(\SS_2, \EE_t)$ in Adams grading $(t-s,s)$. A $\blacksquare = \W$, a $\bullet = \F_4$, $+ = \W/4$, $\times =\W/8$ and $\ast =\W/16$. Pale classes are multiples of $\zeta$. The colors are to remind the reader of the ADSS filtration. Note that interpreting $\blacksquare = \Z_2$, a $\bullet = \Z/2$, $+ = \Z/4$, $\times =\Z/8$ and $\ast =\Z/16$ this is also a picture of $H^s(\GG_2,\EE_t)$.
\label{fig:EinfS2}}
\end{figure}

\section{Galois fixed points and the cohomology of $\GG_2$}\label{sec:gal}


It remains to compute $H^*(\GG_2^1, \EE_t)$ and $H^*(\GG_2, \EE_t)$.
By \cite[Lemma 1.3.2]{BobkovaGoerss} 
\[H^*(\SS_2^1,\EE_*) \cong \WW\otimes_{\ZZ_2}H^*(\GG_2^1, \EE_*), \quad \text{and} \quad H^*(\SS_2,\EE_*) \cong \WW\otimes_{\ZZ_2}H^*(\GG_2, \EE_*). \]
From this, it follows that to get the $H^*(\GG_2, \EE_*)$, we can literally just change all the $\W$ to $\Z_2$'s in our answers (and $\F_4$'s to $\Z/2$'s), and similarly for $\GG_2^1$. 

However, the reader should be warned that the Galois group $\Gal(\F_4/\F_2)$ does not act on the duality resolution. Therefore, a priori, it could be the case that a generator we detect from the ADSS does not correspond to a Galois invariant element in $H^*(\SS_2^1, \EE_*)$. But we have been extremely careful in choosing the generators, as follows.
\begin{itemize}
\item  The elements which come from the ANSS in the sense of \cref{rem:hurimagelanguage}, like $\eta$, $\nu$ and $\sigma$ are Galois invariant since we have a factorization
\[\Ext_{BP_*BP}(BP_*,BP_*) \to H^*(\GG_2, \EE_*) \to H^*(\SS_2, \EE_*)^{\Gal}\subseteq H^*(\SS_2, \EE_*);\]
\item The elements that came directly from the cohomology of $\GG_2$, like 
 $k$ (\cref{lem:kfixed}), $\wchi$ and $\zeta$ (\cref{rem:chizeta}), and the Massey product $[\etachi]$ \eqref{eq:massey} are Galois invariant by construction;
 \item The element $v_1 \in H^0(\GG_2, \EE_0/2)$ is Galois invariant by construction, and this implies that the generators of $H^*(\SS_2, \EE_*)$ we named using the reduction modulo 2 map and $v_1$-multiplications in $H^*(\SS_2, \EE_*/2)$, like $\hbonebar$ for example, are also Galois invariant. 
 \end{itemize}
 Therefore, we get more than just an abstract statement about the structure of the groups, we actually get a full set of generators.

 \begin{theorem}\label{thm:G2coh}
  The cohomology $H^s(\GG_2^1, \EE_t)$ in the range $0\leq t<12$, and for $s\leq 6$ is the $\ZZ_2$-module generated by the classes listed in \Cref{fig:table}, with the only additive relation given by the order of the class indicated below the element in the table. The result is also depicted in \Cref{fig:Einf}.
  
 The cohomology $H^s(\GG_2, \EE_t)$ in the range $0\leq t<12$, and for $s\leq 6$ is the $\ZZ_2$-module generated by the classes listed in \Cref{fig:table2}, with the only additive relation given by the order of the class indicated below the element in the table. The result is also depicted in \Cref{fig:EinfS2}.
 \end{theorem}

We also record the following result about the action of $\GG_2/\GG_2^1$ on $H^*(\GG_2^1, \EE_*)$. This follows directly from \cref{thm:G2coh} but we also give an easy proof.
 
 \begin{lem}\label{lem:inv-enu-1} 
The action of $\GG_2/\GG_2^1$ on $H^*(\GG_2^1,\EE_t)$ is trivial, for $0\leq t <12$. 
\end{lem}
\begin{proof}
The action of $\SS_2/\SS_2^1$ on $H^*(\SS_2^1, \EE_t)$ is trivial in this range \cref{thm:acttriv}. The claim for $\GG_2^1$ follows from the fact that $H^*(\GG_2^1,\EE_t) = H^*(\SS_2^1, \EE_t)^{\Gal}$ and the natural map
\[\SS_2/\SS_2^1 \xrightarrow{\cong} \GG_2/\GG_2^1  \]
is an isomorphism. 
\end{proof}
 
\newpage

\begin{table}[H]
  \begin{adjustbox}{addcode={\begin{minipage}{\width}}{\caption{%
      Generators for the cohomology groups $H^s(\SS_2^1, \EE_t)$ and $H^s(\GG_2^1, \EE_t)$. The colors correspond to the ADSS filtration.
 \label{fig:table}      
      }\end{minipage}},rotate=90,center}
      \includegraphics[width=\textheight]{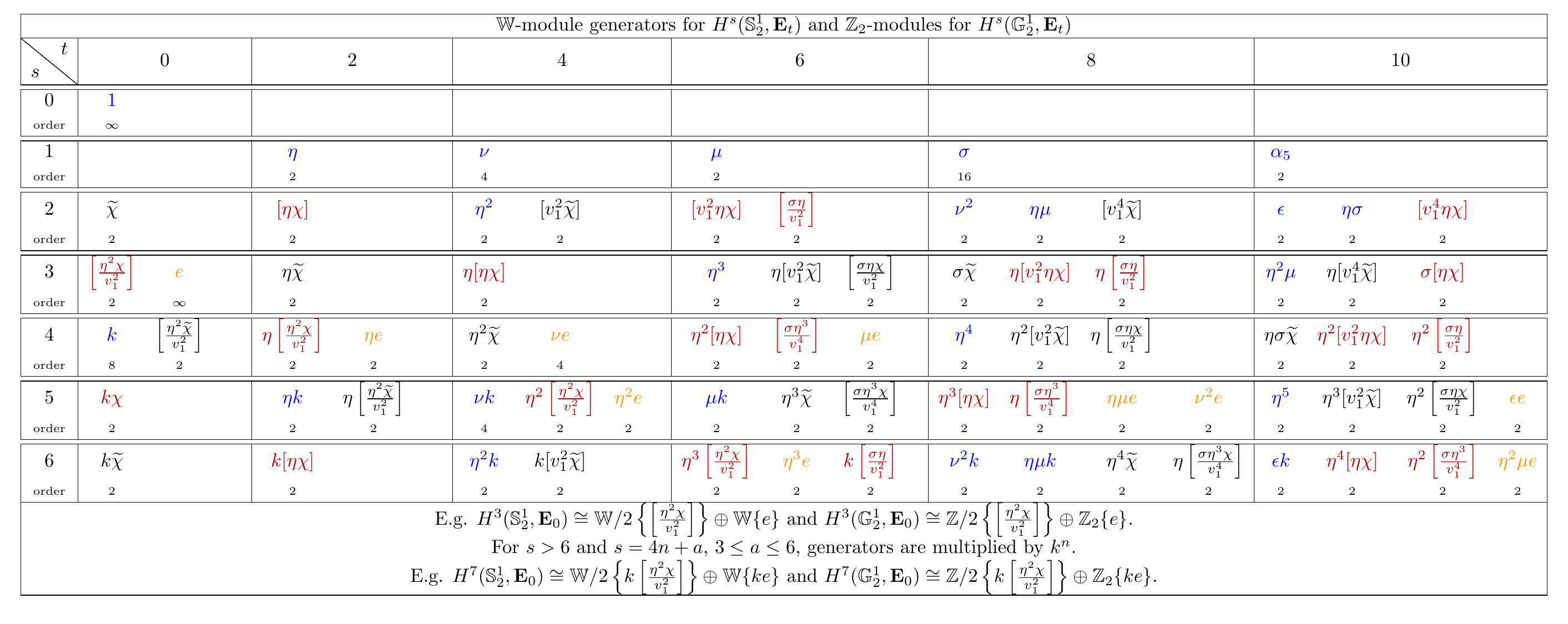}%
  \end{adjustbox}
\end{table}

 \begin{table}[H]
  \begin{adjustbox}{addcode={\begin{minipage}{\width}}{\caption{%
 Generators for the cohomology groups $H^s(\SS_2, \EE_t)$ and $H^s(\GG_2, \EE_t)$
\label{fig:table2}      
      }\end{minipage}},rotate=90,center}
      \includegraphics[width=\textheight]{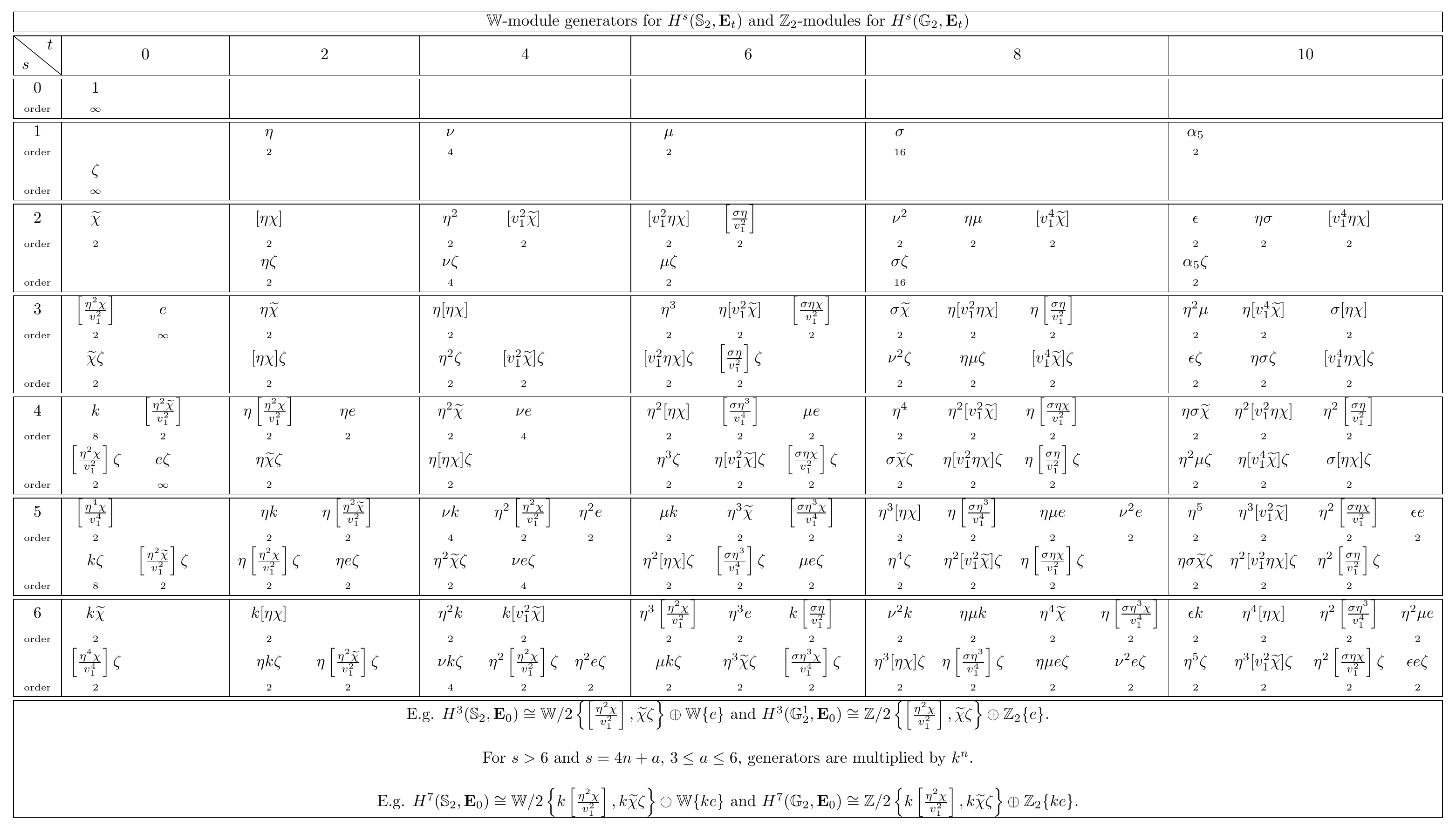}%
  \end{adjustbox}
\end{table}

  \begin{figure}[H]
    \includegraphics[width=\textwidth]{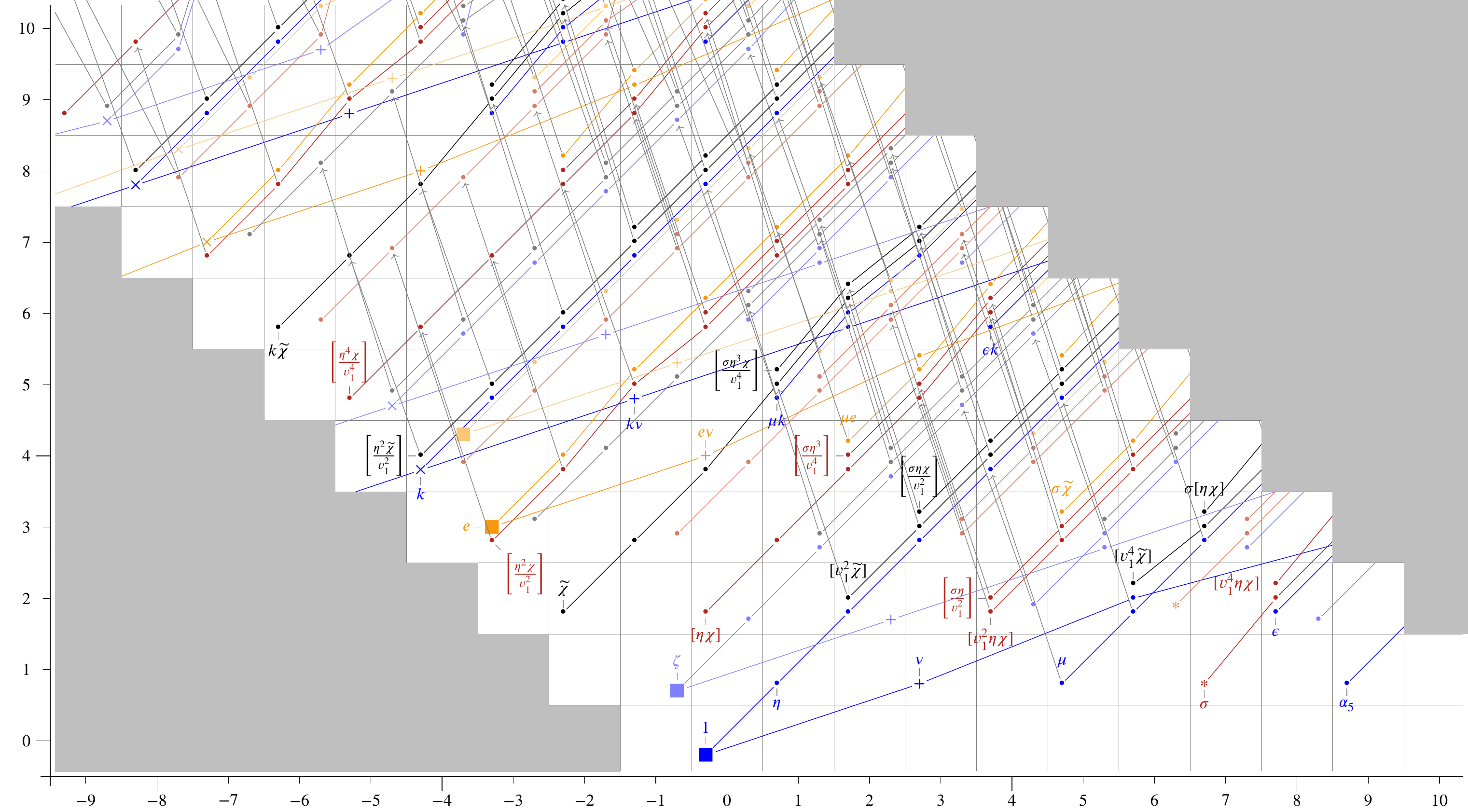}
     \includegraphics[width=\textwidth]{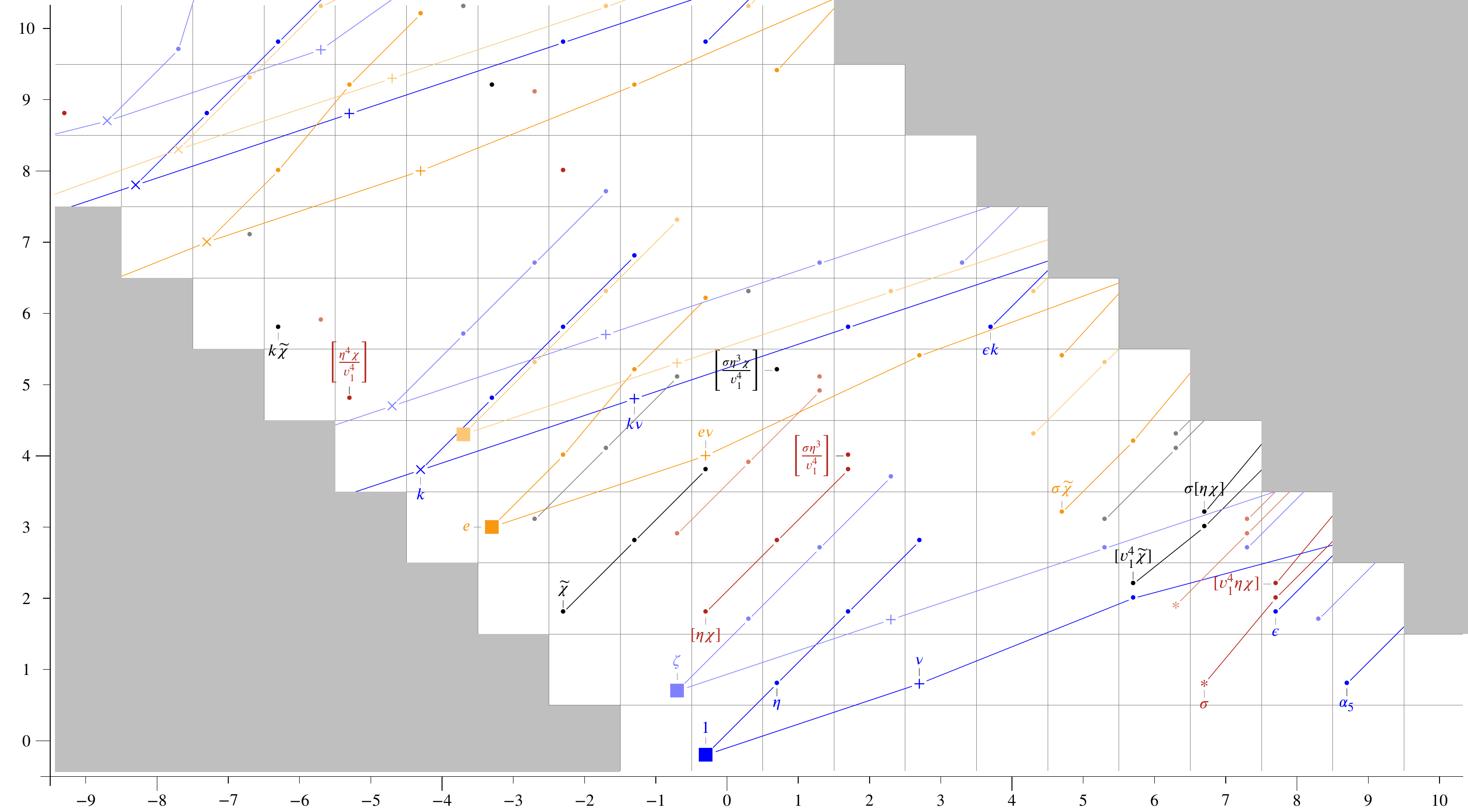}
    \caption{The $d_3$ differentials in the homotopy fixed point spectral sequence $H^s(\GG_2, \EE_t) \Longrightarrow \pi_{t-s}E^{h\GG_2}$ (top) and the $E_5$ page of the homotopy fixed point spectral sequence $H^s(\GG_2, \EE_t) \Longrightarrow \pi_{t-s}E^{h\GG_2}$ (bottom). This is also a picture for the group $\SS_2$.}
    \end{figure}

      \begin{figure}[H]
    \includegraphics[width=\textwidth]{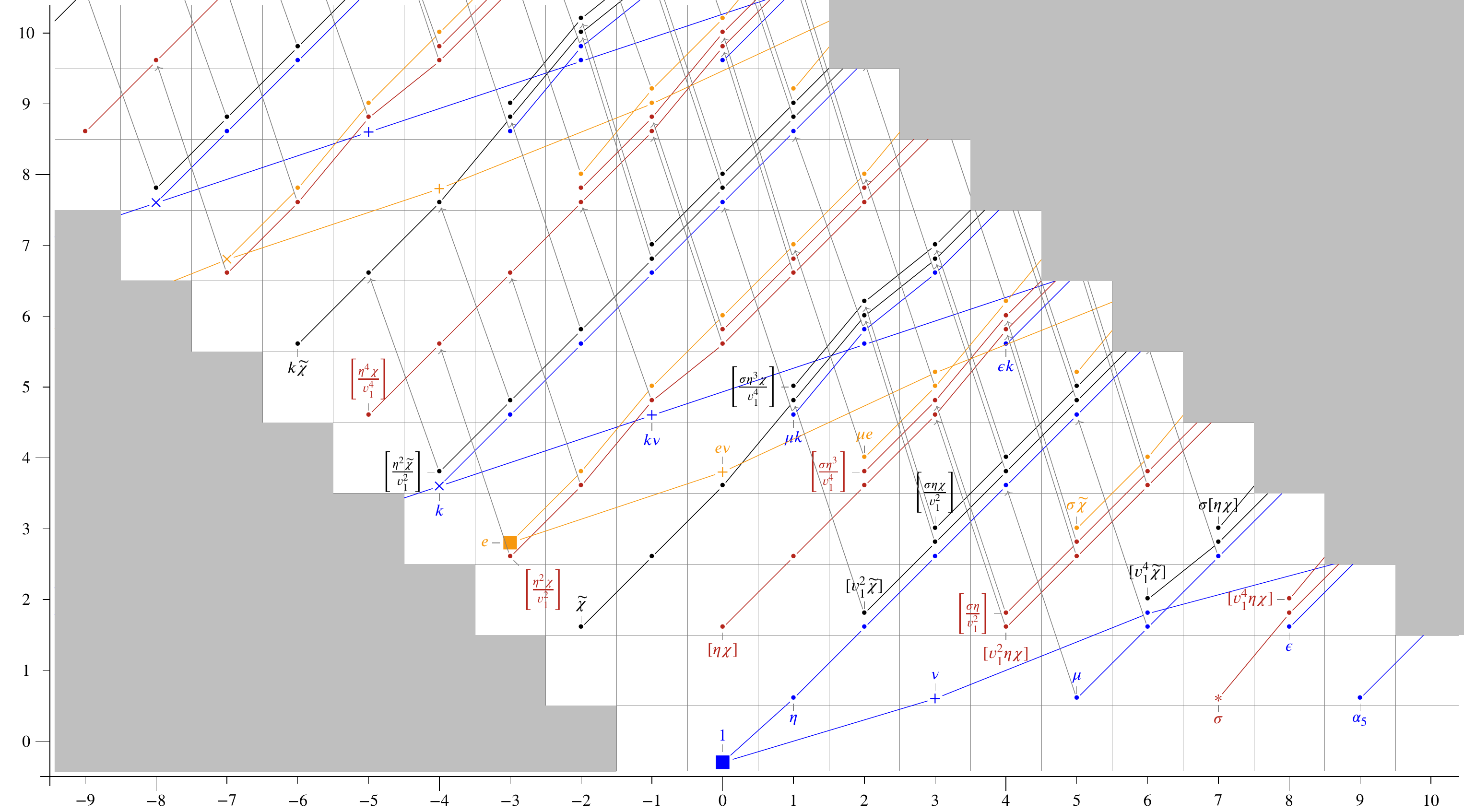}
     \includegraphics[width=\textwidth]{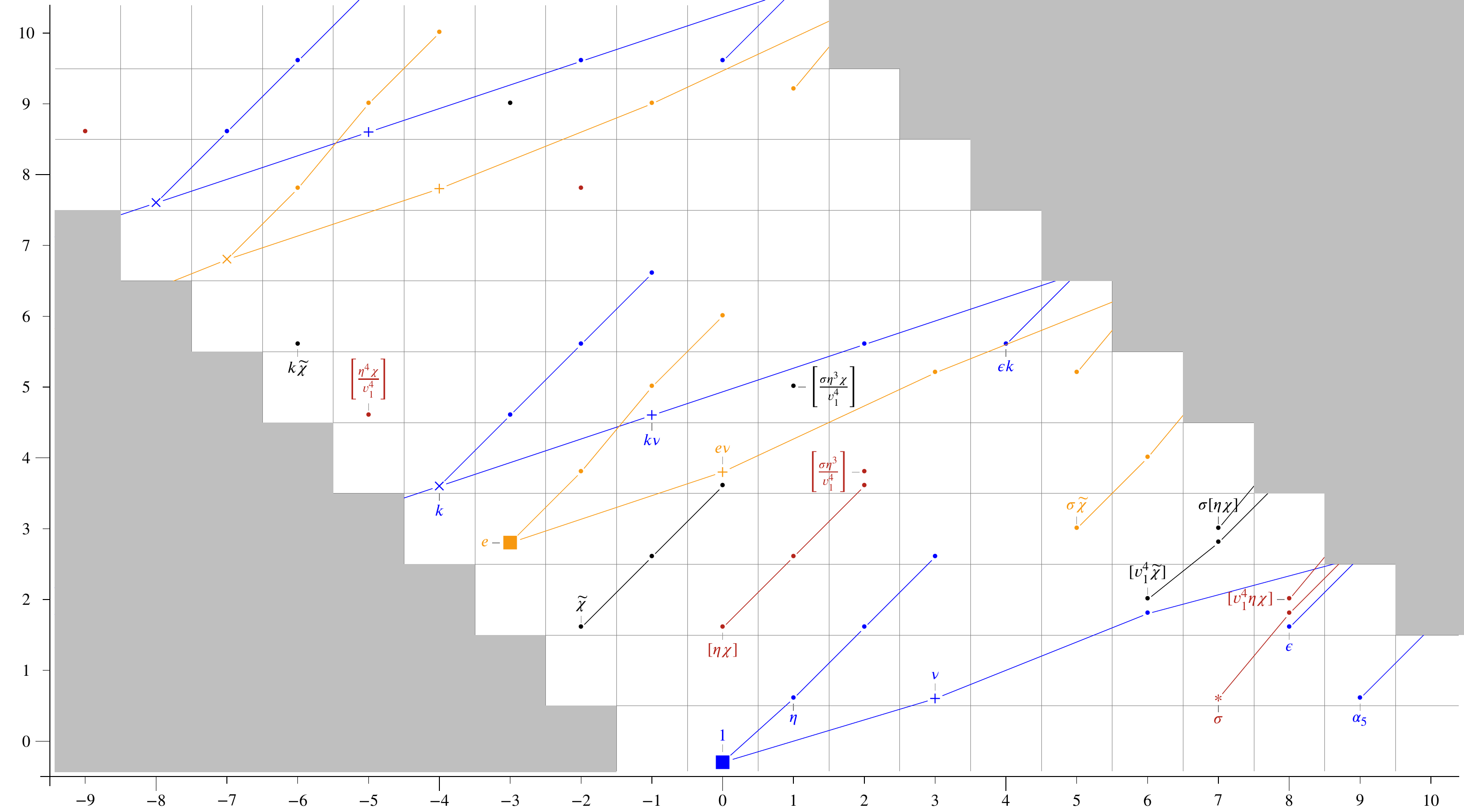}
    \caption{The $d_3$ differentials in the homotopy fixed point spectral sequence $H^s(\GG_2^1, \EE_t) \Longrightarrow \pi_{t-s}E^{h\GG_2^1}$ (top) and the $E_5$ page of the homotopy fixed point spectral sequence $H^s(\GG_2^1, \EE_t) \Longrightarrow \pi_{t-s}E^{h\GG_2^1}$ (bottom). This is also a picture for the group $\SS_2^1$.}
    \end{figure}

\section{Differentials}\label{sec:diffs}


In this section, we compute $d_3$-differentials in the homotopy fixed point spectral sequence of $\EE^{h\mathbb{G}_2^1}$ and $\EE^{h\GG_2}$ in our range.

\begin{rem}
We will use the comparison of spectral sequences
\begin{equation}\label{eq:sscompare}
\xymatrix@R=1pc{
H^*(\GG_2, \EE_*) \ar@{=>}[r] \ar[d]^-r & \pi_*L_{K(2)}S^0\ar[d] \\
H^*(\GG_2, \EE_*/2) \ar@{=>}[r] \ar[d]^-{v_1^{-1}}  & \pi_*L_{K(2)}S^0/2\ar[d] \\
v_1^{-1}H^*(\GG_2, \EE_*/2) \ar@{=>}[r]  & \pi_*L_{K(1)}L_{K(2)}S^0/2 .
}
\end{equation}
where the bottom spectral sequence is the $v_1$-localized $K(2)$-local $\EE_2$-based Adams--Novikov spectral sequence for the Moore spectrum $S^0/2$. Denote by $v_1^{-1}r$ the left-hand side vertical composition, and note that this is a ring map in cohomology. However, the Moore spectrum is not a ring spectrum so the middle and bottom spectral sequences of \eqref{eq:sscompare} are not multiplicative spectral sequences.
The differentials in the $v_1$-localized spectral sequence for the Moore spectrum were completely computed in \cite[Theorem 8.3.5]{BGH}.
\end{rem}

\begin{prop}\label{thm:d3}
Let $G=\GG_2$, $\mathbb{S}_2$, $\GG_2^1$ or $\mathbb{S}_2^1$. 
In the spectral sequence
\begin{equation}\label{eq:HFSSS21-Sec6}
E_2^{s,t}=H^s(G, \EE_t) \Longrightarrow \pi_{t-s}\EE^{hG}
\end{equation}
the classes
\[1, \ \eta, \ \nu, \ \sigma, \ \epsilon, \  \alpha_5, \  \wchi, \ [\etachi],  e, k\]
are $d_3$-cycles. For $G=\mathbb{G}_2$ or $\mathbb{S}_2$, the class $\zeta$ is also a $d_3$-cycle.
\end{prop}
\begin{proof}
It suffices  to do the case $G=\GG_2$, as the claim for the other subgroups then follows by naturality.
The classes $1, \ \eta, \ \nu, \ \sigma, \ \epsilon, \  \alpha_5$ are $d_3$-cycles because they are the images of permanent cycles in the $BP$-based ANSS for $S^0$, see \Cref{sec:Hurewicz}. That the class $\zeta$ is a permanent cycle is a well-known result due to Hopkins--Miller \cite[Theorem 6]{DH}.

The element $\wchi$ maps to $\chi^2$ in $v_1^{-1}H^*(\GG_2, \EE_*/2)$ 
and $\chi^2$ is a $d_3$ cycle in the $v_1$-localized spectral sequence by \cite[Theorem 8.3.5]{BGH}.
Therefore, $d_3(\wchi)$ is an element of
\[H^5(\GG_2,E_2) \cong \F_2\left\{\eta k, \eta\gwchi, \eta\gchi\zeta,\eta e\zeta \right\},\]
and $d_3(\wchi)$ is in the kernel of $v_1^{-1}r$. However, none of these classes map to zero under $v_1^{-1}r$, since $k$ maps to $v_1^{-4}\eta^4$ and $e$ to $v_1^{-4}\sigma\chi^2$ (see \cref{lem:sigmatimes}). Therefore, $\wchi$ is a $d_3$-cycle in the spectral sequence \eqref{eq:HFSSS21-Sec6} when $G=\GG_2$, and hence so is the Massey product $[\etachi]=\langle \wchi,2,\eta \rangle $.

For $k$ and $e$, we use an analogous argument. Both elements map to non-trivial $d_3$-cycles and, again, the targets must be in the kernel of $v_1^{-1}r$. But the kernel of $v_1^{-1}r$ in the bi-degrees of both $d_3(k)$ and $d_3(e)$ is zero. Indeed, $d_3(k)$ is in
\[ H^{7}(\GG_2,E_2) \cong \F_2 \left\{\eta  k\wchi, \eta  \kchi\zeta \right\},\] 
and $d_3(e)$ is in
\[H^6(\GG_2, E_2) \cong \F_2\left\{ k[\eta\chi],\eta k\zeta, \eta\gwchi\zeta \right\}. \qedhere\]
\end{proof}

We also have the following standard differential that can be found in many sources. See, for example, \cite[Remark 4.1.7]{BGH} or \cite[Lemma 2.21]{BobkovaGoerss}.
\begin{lem}\label{lem:diffmu}
Let $G=\GG_2$, $\mathbb{S}_2$, $\GG_2^1$ or $\mathbb{S}_2^1$. For any $z\in E_2^{s,t}  \cong H^s(G,\EE_t)$ we have 
\[d_3(\mu z) = \eta^4 z + \mu d_3(z) . \]
\end{lem}

\begin{lem}\label{lem:etanonzero}
Let $G=\GG_2$, $\GG_2^1$, $\SS_2$ or $\SS_2^1$.
\begin{enumerate}
\item
If $x\in H^*(G, \EE_*/2)$ is a $v_1$-free class, then $\eta x \neq 0$ in $H^*(G, \EE_*/2)$. 
\item 
 If that $x\in H^*(G, \EE_*)$ is a class such that $v_1^{-1}r(x) \neq 0$, then $\eta x \neq 0$ in $H^*(G, \EE_*)$. 
\end{enumerate}
\end{lem}
\begin{proof}
By \cite[Theorems 8.2.5, 8.2.6]{BGH}, $\eta$ multiplication in $v_1^{-1}H^*(G, \EE_*/2)$ is injective. The claims follow from the fact that $v_1^{-1}r(\eta x) = \eta (v_1^{-1}r)(x) \neq 0$ if $ v_1^{-1}r(x) \neq 0$.
\end{proof}

\begin{theorem}
For $G=\GG_2, \SS_2, \GG_2^1, $ or $\SS_2^1$, 
 the non-zero $d_3$-differentials in the homotopy fixed point spectral sequence \eqref{eq:HFSSS21-Sec6} with source in the range $0\leq t<12$ are multiples by $d_3$-cycles of \cref{thm:d3} 
of the following differentials:
\begin{multicols}{2}
\begin{enumerate}[(1)]
\item\label{diffgchi} $d_3\left(\gchi\right) =\eta \kchi = k[\etachi]$;
\item\label{diffgwchi}  $d_3\left( \gwchi \right) = \eta k\wchi$;
\item \label{diffv12wchi} $d_3([v_1^2 \wchi]) = \eta^3 \wchi$;
\item \label{diffmu} $d_3(\mu)=\eta^4$;
\item \label{diffhbone}  \label{diffgother} $d_3\left( \hbone\right) = \eta \ghbone $;
\item \label{diffv12etachi} $d_3([v_1^2 \etachi]) = \eta^3 [\etachi]$;
\item \label{diffhbonebar} $d_3\left( \hbonebar\right) = \eta \ghbonebar $.
\end{enumerate}
\end{multicols}
\end{theorem}
\begin{proof}
As usual, it suffices to prove the claims for $G=\GG_2$.
The differential \eqref{diffmu} is determined by \cref{lem:diffmu}.
The differentials \eqref{diffgchi}--\eqref{diffv12wchi} can be shown by mapping further to the $v_1$-localized spectral sequence, as $v_1^{-1}r$ is injective on their target groups. So, these differentials follow directly from \cite[Lemma 8.3.3 or Theorem 8.3.5]{BGH}. 

For \eqref{diffhbone}--\eqref{diffhbonebar} the target group for the $d_3$-differentials contains $\nu^2$-multiples, which are in the kernel of $v_1^{-1}r$.
Comparing with the $v_1$-localized spectral sequence we get
\begin{enumerate}
\item[(5)]  $d_3\left( \hbone\right) = \eta \ghbone \mod \nu^2e$
\item[(6)]  $d_3([v_1^2 \etachi]) = \eta^3 [\etachi] \mod \nu^2e$
\item[(7)] $d_3\left( \hbonebar\right) = \eta \ghbonebar \mod (\nu^2k, \nu^2e \zeta)$. 
\end{enumerate}

For \eqref{diffhbone}, we have
\[d_3\left(\hbone\right) =\eta \ghbone  + a\nu^2e\]
for some coefficient $a \in \F_2$. By naturality, the same differential happens in the spectral sequence
\begin{align}\label{eq:s21ss}
E_2^{*,*}=H^*(\SS_2^1, \EE_*/2) \Longrightarrow \pi_*\EE^{h\SS_2^1}\wedge S^0/2.
\end{align}
Then $\nu$-multiplication gives that
\[d_3\left(\nu \hbone\right) =  a\nu^3e,\]
since the differentials are $\nu$-linear and $\nu \eta=0$.

The class $\nu^3e$ is non-trivial on the $E_2$-page of \eqref{eq:s21ss}.
However, we will show that
\[\nu \hbone=0 \in H^3(\SS_2^1, \EE_{10}/2),\]
which will imply that $a=0$. 
Indeed, \Cref{thm:Einfmod2} implies that there is no $v_1$-torsion in $H^3(\SS_2^1, \EE_{10}/2)$. Then 
\Cref{lem:etanonzero} implies that
$\eta \colon  H^3(\SS_2^1, \EE_{10}/2)  \to H^4(\SS_2^1, \EE_{12}/2)$ is injective.

The differential \eqref{diffv12etachi} is proved analogously.

For \eqref{diffhbonebar}, we again have
\[d_3\left( \hbonebar\right) = \eta \ghbonebar + a\nu^2k + b\nu^2e \zeta\]
for $a,b\in \F_2$. We again multiply by $\nu$ to get
\[d_3\left( \nu \hbonebar\right) =  a\nu^3k + b\nu^3e \zeta.\]
Since $H^4(\GG_2, \EE_{10})$ maps injectively into $v_1^{-1}H^{4}(\GG_2, \EE_{10}/2)$, we must have that 
\[0=\nu  \hbonebar \in H^4(\GG_2, \EE_{10}) \]
since all classes in $H^4(\GG_2, \EE_{10}) $ support non-trivial $\eta$ multiplications by \cref{lem:etanonzero}. 

So, $a\nu^3k + b\nu^3e \zeta =0$. Mapping to the spectral sequence \eqref{eq:s21ss}, naturality gives a differential 
$ d_3\left( \nu \hbonebar\right) =  a\nu^3k $ in  \eqref{eq:s21ss}. However, since $\nu \hbonebar =0$, we conclude $a=0$. Therefore, to prove that $b=0$, it's enough to prove that $\nu^3e \zeta\neq0$ in $H^7(\GG_2, \EE_{12})$. Furthermore, it's enough to prove this in $H^7(\SS_2, \EE_{12}/2)$.

We have an exact sequence
\[ 0 \to H^6(\SS_2^1, \EE_{12}/2)_{\ZZ_2} \xrightarrow{\zeta} H^7(\SS_2, \EE_{12}/2) \to H^7(\SS_2^1, \EE_{12}/2)^{\ZZ_2} \to 0, \]
so we need to analyze the $\ZZ_2$-action on 
\begin{align}\label{eq:HSixETwelve}
\begin{split}
H^6(\SS_2^1, \EE_{12}/2) &\cong \F_4\{ kv_1\epsilon ,  k \hbtwo, k\btwobar\} \oplus \F_4\{ \nu^3 e\}  \\
& \quad \quad \oplus \F_4\{\eta^6 , \eta^5v_1\chi, \eta^3 \gbone, \eta^4 v_1^2 \chi^2, \eta^3 \hbonebar, v_1^3\eta^3 e \}
\end{split}
\end{align}
The first and second summands generate the $v_1$-power torsion, while the third summand consists of $v_1$-free classes. 
That $\eta$, $v_1$, $\chi$, $\hbonebar$ and $e$ are fixed  under the $\ZZ_2$  action was shown in \cref{sec:Z2action}. The class $\gbone \in H^3(\SS_2^1, \EE_{6}/2)$ is fixed since it is fixed after $v_1$-localization, and $H^3(\SS_2^1, \EE_{6}/2)$ maps injectively into $v_1^{-1}H^3(\SS_2^1, \EE_{*}/2)$. This shows that the action of $\Z_2$ on the third summand is trivial.

Since $ kv_1\epsilon $ and $\nu^3 e$ are products of $\Z_2$-fixed classes, they are fixed. Since $k$ is fixed, we can analyze the action on $k \hbtwo$ and $k\btwobar$ by looking at the action on
\begin{align*}\hbtwo, \btwobar \in H^2(\SS_2^1, \EE_{12}/2)  &\cong \F_4\{v_1\epsilon, \hbtwo, \btwobar \} \\
 & \quad \quad \oplus \F_4\{v_1^4 \eta^2, \eta v_1^5 \chi, v_1^3\hbone, v_1^6 \chi^2, v_1^2 \sigma \chi\}
\end{align*}
The second summand is again generated by $v_1$-free classes, so, since $\hbtwo$ and $\btwobar$ are $v_1$-power torsion, we have 
\begin{align*} \pi_* \hbtwo &= a_1 v_1\epsilon+ a_2 \hbtwo + a_3 \btwobar  \\
\pi_* \btwobar&= b_1 v_1\epsilon+ b_2 \hbtwo + b_3 \btwobar 
\end{align*}
for $a_i,b_i \in \F_4$. Multiplying by $k$, we have that the first summand
in \eqref{eq:HSixETwelve} is  a $\ZZ_2$-submodule. 
Hence 
 the class $\nu^3 e$ is non-trivial in the coinvariants $H^6(\SS_2^1, \EE_{12}/2)_{\ZZ_2}$. Therefore, the class $\zeta \nu^3 e \neq 0$ in $H^7(\SS_2, \EE_{12}/2)$, which finishes the proof.
\end{proof}

\newcommand{\etalchar}[1]{$^{#1}$}


\begin{thebibliography}{GHMR15}

\bibitem[Bau08]{tbauer}
Tilman Bauer.
\newblock Computation of the homotopy of the spectrum {\tt tmf}.
\newblock In {\em Groups, homotopy and configuration spaces}, volume~13 of {\em
  Geom. Topol. Monogr.}, pages 11--40. Geom. Topol. Publ., Coventry, 2008.
\newblock URL: \url{https://doi.org/10.2140/gtm.2008.13.11}.

\bibitem[BBG{\etalchar{+}}22]{BBGHPS_Pic}
Agnes {Beaudry}, Irina {Bobkova}, Paul~G. {Goerss}, Hans-Werner {Henn},
  Viet-Cuong {Pham}, and Vesna {Stojanoska}.
\newblock {The Exotic $K(2)$-Local Picard Group at the Prime $2$}.
\newblock {\em arXiv e-prints}, page arXiv:2212.07858, December 2022.
\newblock \href {http://arxiv.org/abs/2212.07858} {\path{arXiv:2212.07858}}.

\bibitem[Bea15]{BeaudryRes}
Agn\`es Beaudry.
\newblock The algebraic duality resolution at {$p=2$}.
\newblock {\em Algebr. Geom. Topol.}, 15(6):3653--3705, 2015.
\newblock URL: \url{https://doi.org/10.2140/agt.2015.15.3653}.

\bibitem[Bea17]{BeaudryTowards}
Agn\`es Beaudry.
\newblock Towards the homotopy of the {$K(2)$}-local {M}oore spectrum at
  {$p=2$}.
\newblock {\em Adv. Math.}, 306:722--788, 2017.
\newblock URL: \url{https://doi.org/10.1016/j.aim.2016.10.020}.

\bibitem[Bea19]{BeaudryAlpha}
Agn\`es Beaudry.
\newblock The {$\alpha$}-family in the {$K(2)$}-local sphere at the prime 2.
\newblock In {\em Homotopy theory: tools and applications}, volume 729 of {\em
  Contemp. Math.}, pages 1--19. Amer. Math. Soc., [Providence], RI, [2019]
  \copyright 2019.
\newblock \href {https://doi.org/10.1090/conm/729/14689}
  {\path{doi:10.1090/conm/729/14689}}.

\bibitem[BG18]{BobkovaGoerss}
I.~Bobkova and P.~G. Goerss.
\newblock {T}opological resolutions in {$K(2)$}-local homotopy theory at the
  prime {2}.
\newblock {\em Journal of Topology}, 11(4):917--956, 2018.
\newblock \href {https://doi.org/10.1112/topo.12076}
  {\path{doi:10.1112/topo.12076}}.

\bibitem[BGH22]{BGH}
Agn\`es Beaudry, Paul~G. Goerss, and Hans-Werner Henn.
\newblock Chromatic splitting for the {$K(2)$}-local sphere at {$p = 2$}.
\newblock {\em Geom. Topol.}, 26(1):377--476, 2022.
\newblock \href {https://doi.org/10.2140/gt.2022.26.377}
  {\path{doi:10.2140/gt.2022.26.377}}.

\bibitem[BMQ20]{BehrensQuigley}
Mark {Behrens}, Mark {Mahowald}, and J.~D. {Quigley}.
\newblock {The 2-primary Hurewicz image of tmf}.
\newblock {\em arXiv e-prints}, page arXiv:2011.08956, November 2020.
\newblock \href {http://arxiv.org/abs/2011.08956} {\path{arXiv:2011.08956}}.

\bibitem[DFHH14]{tmfbook}
Christopher~L. Douglas, John Francis, Andr\'e~G. Henriques, and Michael~A.
  Hill, editors.
\newblock {\em Topological modular forms}, volume 201 of {\em Mathematical
  Surveys and Monographs}.
\newblock American Mathematical Society, Providence, RI, 2014.

\bibitem[DH04]{DH}
Ethan~S. Devinatz and Michael~J. Hopkins.
\newblock Homotopy fixed point spectra for closed subgroups of the {M}orava
  stabilizer groups.
\newblock {\em Topology}, 43(1):1--47, 2004.
\newblock URL: \url{https://doi.org/10.1016/S0040-9383(03)00029-6}.

\bibitem[GHM14]{GoerssSplit}
Paul~G. Goerss, Hans-Werner Henn, and Mark~E. Mahowald.
\newblock The rational homotopy of the {$K(2)$}-local sphere and the chromatic
  splitting conjecture for the prime 3 and level 2.
\newblock {\em Doc. Math.}, 19:1271--1290, 2014.

\bibitem[GHMR05]{ghmr}
Paul~G. Goerss, Hans-Werner Henn, Mark~E. Mahowald, and Charles Rezk.
\newblock A resolution of the {$K(2)$}-local sphere at the prime 3.
\newblock {\em Ann. of Math. (2)}, 162(2):777--822, 2005.
\newblock URL: \url{https://doi.org/10.4007/annals.2005.162.777}.

\bibitem[GHMR15]{GHMRPicard}
Paul Goerss, Hans-Werner Henn, Mark Mahowald, and Charles Rezk.
\newblock On {H}opkins' {P}icard groups for the prime 3 and chromatic level 2.
\newblock {\em J. Topol.}, 8(1):267--294, 2015.
\newblock \href {https://doi.org/10.1112/jtopol/jtu024}
  {\path{doi:10.1112/jtopol/jtu024}}.

\bibitem[{Isa}14]{IsaksenANSSChart}
Daniel~C. {Isaksen}.
\newblock {Classical and motivic Adams-Novikov charts}.
\newblock {\em arXiv e-prints}, page arXiv:1408.0248, August 2014.
\newblock \href {http://arxiv.org/abs/1408.0248} {\path{arXiv:1408.0248}}.

\bibitem[Rav77]{RavCoh}
Douglas~C. Ravenel.
\newblock The cohomology of the {M}orava stabilizer algebras.
\newblock {\em Math. Z.}, 152(3):287--297, 1977.
\newblock URL: \url{https://doi.org/10.1007/BF01488970}.

\bibitem[Rav78]{ravnovice}
Douglas~C. Ravenel.
\newblock A novice's guide to the {A}dams-{N}ovikov spectral sequence.
\newblock In {\em Geometric applications of homotopy theory ({P}roc. {C}onf.,
  {E}vanston, {I}ll., 1977), {II}}, volume 658 of {\em Lecture Notes in Math.},
  pages 404--475. Springer, Berlin, 1978.

\bibitem[Rav86]{ravgreen}
Douglas~C. Ravenel.
\newblock {\em Complex cobordism and stable homotopy groups of spheres}, volume
  121 of {\em Pure and Applied Mathematics}.
\newblock Academic Press, Inc., Orlando, FL, 1986.

\bibitem[Rez]{Rezk512}
Charles Rezk.
\newblock Supplementary {N}otes for {M}ath 512.
\newblock Available from
  \url{http://www.math.uiuc.edu/~rezk/512-spr2001-notes.pdf}.

\bibitem[Sil86]{Silverman}
Joseph~H. Silverman.
\newblock {\em The arithmetic of elliptic curves}, volume 106 of {\em Graduate
  Texts in Mathematics}.
\newblock Springer-Verlag, New York, 1986.
\newblock URL: \url{https://doi.org/10.1007/978-1-4757-1920-8}.

\bibitem[{Str}18]{Strickland}
Neil {Strickland}.
\newblock {Level three structures}.
\newblock {\em arXiv e-prints}, page arXiv:1803.09962, March 2018.
\newblock \href {http://arxiv.org/abs/1803.09962} {\path{arXiv:1803.09962}}.

\bibitem[SW02]{SW}
Katsumi Shimomura and Xiangjun Wang.
\newblock The {A}dams-{N}ovikov {$E_2$}-term for {$\pi_*(L_2S^0)$} at the prime
  2.
\newblock {\em Math. Z.}, 241(2):271--311, 2002.
\newblock URL: \url{https://doi.org/10.1007/s002090200415}.

\end{thebibliography}

\end{document}